\documentclass{amsart}
\usepackage{amsmath,amssymb,amsfonts}
\usepackage{graphicx,xcolor}
\usepackage{epstopdf}
\usepackage[colorlinks=true]{hyperref}
\hypersetup{urlcolor=blue, citecolor=red}
\usepackage{bbm}

\def\quotient#1#2{%
    \lower 0.2ex\hbox{$#1$}\big\backslash \raise0.2ex\hbox{$#2$}%
    }
\newcommand{\hl}{\sl}
\newcommand{\hd}{\sl}
\newcommand{\ba}{\begin{array}}
\newcommand{\ea}{\end{array}}

\newcommand{\LL}{\mbox{\rm L}}
\newcommand{\Cnt}{\mbox{\rm C}}

\renewcommand{\d}{\mathrm{d}}
\newcommand{\e}{\mathrm{e}}

\newcommand{\SX}{{\mathcal G}}
\newcommand{\reg}{{\mathcal R}}
\newcommand{\zero}{{\mathcal Z}}

\newcommand{\vol}{\mbox{\rm vol}}
\newcommand{\bs}{{\mathcal B}}
\newcommand{\Gr}{{Gr}}
\newcommand{\Cr}{{CR}}

\newcommand{\NN}{\mathbb{N}} \newcommand{\ZZ}{\mathbb{Z}}
 \newcommand{\RR}{\mathbb{R}}

\newcommand{\XX}{X}
\newcommand{\ganz}{\overline{\XX}}
\newcommand{\rand}{\partial\XX}

\newcommand{\Lim}{L_\Gamma}
\newcommand{\radlim}{L_\Gamma^{\small{\mathrm{rad}}}}

\newcommand{\at}{\!\cdot\!}
\newcommand{\width}{\mathrm{width}}
\newcommand{\diam}{\mathrm{diam}}
\newcommand{\ndpt}{\partial_{\infty}}
\newcommand{\Hopf}{H}
\newcommand{\xo}{{o}}

\newcommand{\be}{\begin{eqnarray*}}
\newcommand{\ee}{\end{eqnarray*}}
\newcommand{\ein}{\,\rule[-5pt]{0.4pt}{12pt}\,{}}

\newcommand{\is}{\mbox{Is}}

\newcommand{\supp}{\mbox{supp}}

\newcommand{\st}{\mbox{such}\ \mbox{that}\ }

\newtheorem{theorem}{Theorem}[section]
\newtheorem{corollary}{Corollary}
\newtheorem*{mainA}{Theorem A}
\newtheorem*{mainB}{Theorem B}
\newtheorem*{mainC}{Theorem C}
\newtheorem{lemma}[theorem]{Lemma}
\newtheorem{proposition}{Proposition}

\theoremstyle{definition}
\newtheorem{definition}[theorem]{Definition}
\newtheorem{remark}{Remark}

\title[Equidistribution and counting of orbit points for rank one groups]
      {Equidistribution and counting of orbit points \\ for discrete rank one isometry groups of Hadamard spaces}

\author[Gabriele Link]{}
 

 \email{gabriele.link@kit.edu}
 
\begin{document}
\bibliographystyle{AIMS.bst}
\maketitle

\centerline{\scshape Gabriele Link$^*$}
\medskip
{\footnotesize
 \centerline{Institut f\"ur Algebra und Geometrie}
   \centerline{Karlsruhe Institute of Technology (KIT)}
   \centerline{Englerstr.~2, 76 131 Karlsruhe, Germany}
} 

\bigskip

\begin{abstract}
Let $\XX$ be a proper, geodesically complete Hadamard space,  and $\ \Gamma<\is(\XX)$ a discrete subgroup of isometries of $\XX$ with the fixed point of a rank one isometry of $\XX$ in its infinite limit set. 
In this paper we prove that if $\Gamma$ has non-arithmetic length spectrum, then 
the Ricks'  Bowen-Margulis measure 
-- which generalizes the well-known Bowen-Margulis measure in the CAT$(-1)$ setting -- is mixing. If in addition the Ricks' Bowen-Margulis measure  
is finite, then we also have equidistribution of $\Gamma$-orbit points in $\XX$, which in particular yields an asymptotic estimate for the orbit counting function of $\Gamma$. This generalizes well-known facts for non-elementary discrete isometry groups of  Hadamard manifolds with pinched negative curvature and proper CAT$(-1)$-spaces. 

%
\end{abstract}

\section{Introduction}
Let $(\XX,d)$ be a  proper  Hadamard space, $x$, $y\in\XX$ and $\Gamma<\is(\XX)$ a discrete group.
The {\hl Poincar\'e series} of $\Gamma$ with respect to $x$ and $y$  is defined by
\[ P(s;x,y):=\sum_{\gamma\in\Gamma} \e^{-sd(x,\gamma y)};\] 
its exponent of convergence 
\begin{equation}\label{critexpdef}
\delta_\Gamma :=\inf\{s>0\colon \sum_{\gamma\in\Gamma} \e^{-sd(x,\gamma y)}\ \text{converges}\};\end{equation}
is called the {\hl critical exponent} of $\Gamma$. By the 
 triangle inequality the critical exponent is independent of $x,y\in\XX$. 
 We will require that the critical exponent $\delta_\Gamma$ is {\hl finite}, which  is not a severe restriction as it is always the case when $\XX$ admits a compact quotient oder when $\Gamma$ is finitely generated. 

 Obviously $P(s;x,y)$ converges for
   $s>\delta_\Gamma $ and  diverges for $s<\delta_\Gamma $. The group  $\Gamma$ is said to be {\hl divergent}, if $P(\delta_\Gamma;x,y)\, $ 
 diverges, and   {\hl convergent} otherwise.

Since $\XX$ is proper, the {\hl orbit counting function} with respect to $x$ and $y$ 
\begin{equation}\label{orbitcountdef}
 N_\Gamma:[0,\infty)\to [0,\infty),\quad R\mapsto \#\{\gamma\in\Gamma\colon d(x,\gamma y)\leq R\}\end{equation}
satisfies $N_\Gamma(R)<\infty$ for all $R>0$; moreover, it is related to the critical exponent via the formula
\[ \delta_\Gamma =\limsup_{R\to +\infty}\frac{\ln\bigl(N_\Gamma(R)\bigr)}{R}.\]
%

One goal of this article is to give a precise asymptotic estimate for the orbit counting function for a discrete  {\hl rank one group}  $\Gamma$ as in \cite{LinkHTS} (that is a group with the fixed point of a so-called rank one isometry of $\XX$ in its infinite limit set); 
for precise definitions we refer the reader to 
Section~\ref{rankonegroups}. 
Such a rank one group always contains a non-abelian free subgroup generated by two independent rank one elements, hence its critical exponent $\delta_\Gamma$ is strictly positive.
Notice that 
our assumption on $\Gamma$ obviously imposes severe restrictions on the Hadamard space $\XX$ itself: It can neither be a higher rank symmetric space, a higher rank Euclidean building nor a product of Hadamard spaces. 
 
 Using the Poincar{\'e} series from above, a remarkable $\Gamma$-equivariant family of measures $(\mu_x)_{x\in X}$ supported on the geometric boundary $\rand$ of $\XX$ -- a so-called 
conformal density --  can be constructed in our very general setting (see \cite{MR0450547} and  \cite{MR556586} for the original constructions in hyperbolic $n$-space).

Let $\SX$ denote the  set of parametrized geodesic lines in $\XX$ endowed with the compact-open topology (which can be identified with the unit tangent  bundle $S\XX$ if $\XX$ is a Riemannian manifold) and consider the action of $\RR$ on $\SX$ by reparametrization. This action induces a flow $g_\Gamma$ on the quotient space $\quotient{\Gamma}{\SX}$.  If $\XX$ is {\hl geodesically complete}, then thanks to the construction due to R.~Ricks (\cite[Section~7]{Ricks}) -- which uses the conformal density $(\mu_x)_{x\in\XX}$ described above -- we obtain  a $g_\Gamma$-invariant Radon measure $m_\Gamma$ on $\quotient{\Gamma}{\SX}$. This possibly infinite measure will be called the {\hl Ricks' Bowen-Margulis} measure, since it generalizes the classical Bowen-Margulis measure in the CAT$(-1)$-setting.  

If $\Gamma$ is divergent, then according to Theorem~10.2 in \cite{LinkHTS} the dynamical system $( \quotient{\Gamma}{\SX}, g_\Gamma, m_\Gamma)$ is conservative and ergodic.  We also want to mention here that if $\XX$  is a Hadamard {\hl manifold}, then  Ricks' Bowen-Margulis measure $m_\Gamma$ is equal to  Knieper's  measure first introduced in Section~2 of \cite{MR1652924} for cocompact groups $\Gamma$ (and which was used  in \cite{LinkPicaud} for arbitrary rank one groups). 
In the cocompact case  
Knieper's work  further implies that the Ricks' Bowen-Margulis measure is the unique measure of maximal entropy on the unit tangent bundle of the compact quotient $\quotient{\Gamma}{\XX}$ (see again Section~2 in \cite{MR1652924}).

In this article  we will first address the question under which 
hypotheses the dynamical system $( \quotient{\Gamma}{\SX}, g_\Gamma, m_\Gamma)$ is  mixing.  We remark that in our very general setting we cannot hope to get mixing without further restrictions on the group $\Gamma$: 
F.~Dal'Bo (\cite[Theorem A]{MR1779902}) showed  
that even in the special case of a CAT$(-1)$-Hadamard {\hl manifold} $\XX$, the dynamical system  $( \quotient{\Gamma}{S\XX}, g_\Gamma, m_\Gamma)$ with the classical Bowen-Margulis measure $m_\Gamma$ is {\hl not} mixing, if the length spectrum of $\Gamma$ is arithmetic (that is if the set of lengths of closed geodesics in $\quotient{\Gamma}{\XX}$ is a discrete subgroup of $\RR$). However, we obtain the best possible result:
\begin{mainA}\  Let $X$ be a proper, geodesically complete Hadamard space and\break $\Gamma<\is(X)$ a discrete, divergent rank one group. 
Then with respect to Ricks'  Bowen-Margulis measure the geodesic flow on $\quotient{\Gamma}{\SX}$ is mixing or the length spectrum of $\Gamma$ is arithmetic. \end{mainA}
Notice that  in the CAT$(0)$-setting Theorem~A was already proved by M.~Babillot (\cite[Theorem~2]{MR1910932}) in the special case when $\XX$ is a {\hl manifold} and $\Gamma<\is(\XX)$ is cocompact; moreover, in this case the second alternative cannot occur, that is the length spectrum of $\Gamma$  {\hl cannot} be arithmetic. It was  then  generalized by R.~Ricks (\cite[Theorem~4]{Ricks}) to non-Riemannian proper Hadamard spaces  $\XX$ and discrete rank one groups $\Gamma<\is(\XX)$ with {\hl finite} Ricks' Bowen-Margulis measure. Under the additional hypothesis  that  the limit set of $\Gamma$ is equal to the whole geometric boundary $\rand$ of $\XX$, Ricks also proved that the length spectrum  of $\Gamma$ can only be arithmetic if $\XX$ is isometric to a tree with all edge lengths in $c\NN$ for some $c>0$. Here we allow both infinite Ricks' Bowen-Margulis measure and limit sets that are proper subsets of $\rand$. 

Let us mention that the restriction to divergent groups is quite reasonable: If the measure $m_\Gamma$ is infinite, then the mixing property of $( \quotient{\Gamma}{\SX}, g_\Gamma, m_\Gamma)$ only states that for all Borel sets $A$, $B\subset \quotient{\Gamma}{\SX}$ with $m_\Gamma(A)$, $m_\Gamma(B)$ finite we have
\[ \lim_{t\to\pm\infty} m_\Gamma (A\cap g_\Gamma^{t}B) =0.\]
This condition is very weak and obviously neither implies conservativity nor ergodicity. Actually it is easily seen to hold true 
when $(\quotient{\Gamma}{ \SX}, g_\Gamma, m_\Gamma)$ is dissipative
, which -- according to Theorem~10.2 in \cite{LinkHTS} -- is equivalent to the fact that $\Gamma$ is convergent.

In the second part of the article we use the mixing property in the case of finite Bowen-Margulis measure 
to deduce an equidistribution result for $\Gamma$-orbit points in the vein of Roblin's results for CAT$(-1)$-spaces (\cite[Th\'eor\`eme 4.1.1]{MR2057305}):

 \begin{mainB} \ Let $X$ be a proper, geodesically complete  Hadamard space and\break $\ \Gamma<\is(X)$ a discrete rank one group with non-arithmetic length spectrum and finite Ricks' Bowen-Margulis measure $m_\Gamma$. 
 
  Let $f$ be a continuous function from $\ganz\times \ganz$ to $\RR$, and  $x$, $y\in\XX$. Then  
\[ \lim_{T\to\infty}   \Bigl( 
\delta_\Gamma\cdot  \e^{-\delta_\Gamma T} \sum_{\begin{smallmatrix}{\scriptstyle\gamma\in\Gamma}\\{\scriptstyle d(x,\gamma y )\le T}\end{smallmatrix}} f(\gamma y,\gamma^{-1} x)\Bigr)=\frac1{\Vert m_\Gamma \Vert}  \int_{\rand\times\rand} f(\xi,\eta)\d \mu_x(\xi) \d \mu_y(\eta).\]
\end{mainB}

Finally, from the equidistribution result Theorem~B and its proof we get the following asymptotic estimates for the orbit counting function introduced in (\ref{orbitcountdef}):
   \begin{mainC}\  Let $X$ be a proper, geodesically complete Hadamard space, $x$, $y\in\XX$ and $\Gamma<\is(X)$ a  discrete rank one group. 
  \begin{itemize}
   \item[(a)] If  $\Gamma$ is divergent with non-arithmetic length spectrum and finite Ricks' Bowen-Margulis measure $m_\Gamma$, then
    \[ \lim_{R\to\infty} \delta_\Gamma\cdot \e^{-\delta_\Gamma R} \#\{\gamma\in\Gamma\colon d(x,\gamma y)\leq R\} = \mu_x(\rand)\mu_y(\rand)/ \Vert m_\Gamma\Vert .\]
  \item[(b)] If  $\Gamma$ is divergent with non-arithmetic length spectrum and infinite Ricks' Bowen-Margulis measure, then
  \[ \displaystyle \lim_{R\to\infty} \e^{-\delta_\Gamma R} \#\{\gamma\in\Gamma\colon d(x,\gamma y)\leq R\} =0.\]
    \item[(c)] If $\Gamma$ is convergent, then $\quad \displaystyle \lim_{R\to\infty} \e^{-\delta_\Gamma R} \#\{\gamma\in\Gamma\colon d(x,\gamma y)\leq R\}=0$. 
\end{itemize}
\end{mainC}

Notice that in work in progress with Jean-Claude Picaud we apply the  equidistribution result Theorem~B above  to get  asymptotic estimates for the number of closed geodesics modulo free homotopy in $\quotient{\Gamma}{\XX}$ which are much more general and much more precise than the ones given in \cite{MR2290453}.  

The paper is organized as follows: Section~\ref{prelim}  fixes some notation and recalls basic facts concerning Hadamard spaces and rank one geodesics. 
In Section~\ref{rankonegroups} we introduce the notions of rank one isometry and $\is(\XX)$-recurrence and state some important facts. We also recall the definition of a rank one group and give the   weakest condition which ensures that a discrete group $\Gamma<\is(\XX)$ is rank one. 
In Section~\ref{geodcurrentmeasures} we introduce the notion of geodesic current and describe Ricks' construction of a geodesic flow invariant measure associated to such a geodesic current first on the quotient $\quotient{\Gamma}{[\SX]}$ of parallel classes of parametrized geodesic lines and finally on the quotient $\quotient{\Gamma}{\SX}$ of parametrized geodesic lines. Moreover, we recall from \cite{LinkHTS} a few results about the corresponding dynamical systems. 
Section~\ref{Mixing} is devoted to the proof of Theorem~A, which follows  M.~Babillot's strategy (\cite[Section~2.2]{MR1910932}) and uses cross-ratios of quadrilaterals similar to the ones introduced by R.~Ricks (\cite[Section~10]{Ricks}). In Section~\ref{shadowconecorridor} we introduce the notions of shadows, cones and corridors and state some important properties that are needed in the proof of Theorem~B. 
Section~\ref{RicksBMestimates} gives estimates for the so-called Ricks' Bowen-Margulis measure, which is the Ricks' measure associated to the quasi-product geodesic current coming from a conformal density.  In Section~\ref{equidistribution} we prove Theorem~B, and Section~\ref{orbitcounting} finally deals with the orbit counting function and the proof of Theorem~C.

\section{Preliminaries on Hadamard spaces}\label{prelim} 

%
The purpose of this section is to introduce terminology and notation and to summarize basic results about Hadamard spaces. 
Most of the material can be found in \cite{MR1377265} and \cite{MR1744486} 
(see also \cite{MR656659} and \cite{MR823981}  in  the special case of Hadamard manifolds and \cite{Ricks} for more recent results). 

Let $(\XX,d)$ be a metric space. For $y\in \XX$ and $r>0$ we will denote
$B_y(r)\subset\XX$ the open ball of radius $r$ centered at $y\in\XX$.
A {\hd geodesic} is an isometric map $\sigma$ from a closed interval $I\subset\RR$ or $I=\RR$ to $\XX$. 
For more precision we use the term {\hd geodesic ray} if $I=[0,\infty)$ and  {\hd geodesic line} if $I=\RR$.

We will deal here with  {\hl Hadamard spaces} $(\XX,d)$, that is complete metric spaces in which for any two points $x,y\in\XX$ there exists a geodesic $\sigma_{x,y}$ joining $x$ to $y$ (that is a geodesic $\sigma=\sigma_{x,y}:[0,d(x,y)]\to \XX$ with $\sigma(0)=x$ and $\sigma(d(x,y))=y$) and in which all geodesic triangles satisfy the CAT$(0)$-inequality. This implies in particular that $\XX$ is simply connected and that the geodesic joining an arbitrary pair of points in $\XX$ is unique.   Notice  however that in the non-Riemannian setting completeness of $\XX$ does not imply that every geodesic  can be extended to a geodesic line, so $\XX$ need not be geodesically complete. The geometric boundary $\rand$ of
$\XX$ is the set of equivalence classes of asymptotic geodesic
rays endowed with the cone topology (see for example Chapter~II in \cite{MR1377265}).  We remark that for all $x\in\XX$ and all $
\xi\in\rand$  there exists a unique geodesic ray $\sigma_{x,\xi}$ with origin $x=\sigma_{x,\xi}(0)$ representing $\xi$.

Given two geodesics $\sigma_1:[0,T_1]\to\XX$, $\sigma_2:[0,T_2]\to\XX$ with $\sigma_1(0)=\sigma_2(0)=:x$ the {\hl Alexandrov angle} $\angle(\sigma_1,\sigma_2)$ is defined by
\[ \angle (\sigma_1,\sigma_2):=\lim_{t_1,t_2\to 0} \angle_{\overline{x}}\bigl(\overline{\sigma_1(t_1)},\overline{\sigma_2(t_2)}\bigr),\]
where the angle on the right-hand side denotes the angle of a comparison triangle  in the Euclidean plane of the triangle with vertices $\sigma_1(t_1)$, $x$ and $\sigma_2(t_2)$ (compare \cite[Proposition II.3.1]{MR1744486}). By definition, every Alexandrov angle has values in $[0,\pi]$. For $x\in\XX$, $y,z\in\ganz\setminus\{x\}$ the angle $\angle_x(y,z)$ is then defined by
\begin{equation}\label{Alexandrovangle}
\angle_x(y,z):=\angle (\sigma_{x,y},\sigma_{x,z}).
\end{equation}

From here on we will require that $\XX$ is proper; in this case the geometric boundary $\rand$ is compact and the space $\XX$ is a dense and open subset of the compact space $\ganz:=\XX\cup\rand$.
Moreover, the action of the isometry group $\is(\XX)$ on $\XX$ naturally extends to an action by homeomorphisms on the geometric boundary. 


If $x, y\in \XX$, $\xi\in\rand$ and $\sigma$ is a geodesic ray in the
class of $\xi$, we set 
\begin{equation}\label{buseman}
 \bs_{\xi}(x, y)\,:= \lim_{s\to\infty}\big(
d(x,\sigma(s))-d(y,\sigma(s))\big).
\end{equation}
This number exists, is independent of the chosen ray $\sigma$, and the
function
\[ \bs_{\xi}(\cdot , y):
 \XX \to  \RR,\quad
x \mapsto \bs_{\xi}(x, y)\]
is called the {\hl Busemann function} centered at $\xi$ based at $y$ (see also Chapter~II in~\cite{MR1377265}). 
Obviously we have
\[ \bs_{g\cdot\xi}(g\at x,g\at y) = \bs_{\xi}(x, y)\quad\text{for all }\ x,y\in\XX\quad\text{and }\  g\in\is(\XX),\]
and the cocycle identity  
\begin{equation}\label{cocycle}
\bs_{\xi}(x, z)=\bs_{\xi}(x, y)+\bs_{\xi}(y,z)
\end{equation}
holds for all $x,y,z\in\XX$.

Since $\XX$ is non-Riemannian in general, we consider (as a substitute of the unit tangent bundle $S\XX$) the set  of parametrized geodesic lines in $\XX$ which we will denote $\SX$. We endow this set  with the distance function $d_1$ given by
\begin{equation}\label{metriconSX} d_1(u,v):=\sup \{ \e^{-|t|} d\bigl(v(t), u(t)\bigr) \colon t\in\RR\}\ \mbox{ for} \ u,v\in \SX ;\end{equation}
this distance function induces the compact-open topology, and every isometry of $\XX$ naturally extends to an isometry of the metric space $(\SX,d_1)$.


Moreover, there is a natural map $p:\SX\to\XX$ defined as follows: To a geodesic line  $v:\RR\to \XX$ in $\SX$ we assign its origin 
$pv:=v(0)\in\XX$. 
Notice that $p$ is proper, $1$-Lipschitz and $\is(\XX)$-equivariant; if $\XX$ is geodesically complete, then $p$ is surjective. 

For a geodesic line $v\in \SX$ we denote its extremities $v^-:=v(-\infty)\in\rand$ and  $v^+:=v(+\infty)\in\rand$ the negative and positive end point of $v$; in particular, we can define the end point map
\[ \ndpt:\SX\to \rand\times\rand,\quad v\mapsto (v^-,v^+).\]
For $v\in\SX$ we define the parametrized geodesic $-v\in\SX$ by
\[ (-v)(t):=v(-t)\quad\text{for all}\quad t\in\RR.\]

We say that a point $\xi\in\rand$ can be joined to  $\eta\in\rand$  by a
geodesic $v\in \SX$ if   
$v^-=\xi$ and $v^+=\eta$. Obviously the  set of pairs $(\xi,\eta)\in\rand\times\rand$   \st $\xi$ and $\eta$ can be joined by a geodesic coincides with  $ \ndpt\SX $, the image of $\SX$ under the end point map $\ndpt$. It is well-known that if $\XX$ is 
CAT$(-1)$, 
then any pair of distinct boundary points $(\xi,\eta)$ belongs to $\ndpt\SX $, and  the geodesic joining $\xi$ to $\eta$ is unique up to reparametrization. In general however, 
the set $\ndpt\SX $ is much smaller compared to $\rand\times\rand$ minus the diagonal due to the possible existence of flat subspaces in $\XX$. 
For $(\xi,\eta)\in\ndpt\SX $ we denote by
\begin{equation}\label{joiningflat}
(\xi\eta):=p\bigl(\{ v\in \SX \colon v^-=\xi,\ v^+=\eta\}\bigr)=p\circ \ndpt^{-1}(\xi,\eta)
\end{equation}
the subset of points in $\XX$ which lie on a geodesic line  joining $\xi$ to  $\eta$. It is well-known 
that $(\xi\eta)=(\eta\xi)\subset \XX$ is a closed and convex  subset of $\XX$ which is   isometric to a  product $C_{(\xi\eta)}\times\RR$, where  $C_{(\xi\eta)}=C_{(\eta\xi)}$ is again a closed and convex set. 

For $x\in\XX$ and $(\xi,\eta)\in\ndpt\SX$ we denote  
\begin{equation}\label{orthogonalproj}
 v= v(x;\xi,\eta)\in \SX\end{equation}
  the unique  parametrized geodesic   line satisfying the conditions $v\in\ndpt^{-1}(\xi,\eta)$ and\break  $d\bigl(x, v(0)\bigr)=d\bigl(x,(\xi\eta)\bigr)$. Notice that 
its origin $pv=v(0)$ is precisely the orthogonal projection onto the closed and convex subset $C_{(\xi\eta)}$.
Obviously we have
\[v(x;\eta,\xi)=-v(x;\xi,\eta)\quad\text{and}\quad \gamma v( x;\xi,\eta)=v(\gamma x;\gamma \xi,\gamma \eta)\quad\text{for all}\ \ \gamma\in\is(\XX).\]

In order to describe the sets $(\xi\eta)$ and $C_{(\xi\eta)}$ more precisely and for later use 
 we introduce as in \cite[Definition 5.4]{Ricks} for $x\in\XX$ the so-called {\hd Hopf parametrization} map 
 \begin{equation}\label{HopfPar}
\Hopf_x: \SX\to \ndpt\SX \times \RR,\quad v\mapsto \bigl(v^-,v^+,\bs_{v^-}(v(0),x)\bigr)
\end{equation}
of $\SX$ with respect to $x$. We remark that compared to \cite[Definition 5.4]{Ricks} and (5) in \cite{LinkHTS}  we changed the sign in the last coordinate in order to make (\ref{geodflowHopf}) below consistent. 
It is immediate that for a CAT$(-1)$-space $\XX$ this map is a homeomorphism; in general  it is only continuous and surjective. 
Moreover, it depends on the point $x\in\XX$ as follows: If $y\in \XX$,
$v\in \SX$ and  $\Hopf_x(v)=(\xi,\eta,s)$, then 
\[ \Hopf_y(v)=\bigl(\xi,\eta,s+\bs_{\xi}(x,y)\bigr)\]
by the cocycle identity~(\ref{cocycle}) for the Busemann function 
(compare also \cite[Section~3]{MR1207579}).

The Hopf parametrization map allows to define an equivalence relation $\sim$ on $\SX$ as follows: If $u,v\in \SX$, then $u\sim v$ if and only if $\Hopf_x(u)=\Hopf_x(v)$. Notice that this definition does not depend on the choice of $x\in\XX$ and that every point $(\xi,\eta,s)\in\ndpt\SX\times \RR$ uniquely determines an equivalence class $[v]$ with $v\in\SX$. 
The {\hl width} of $v\in\SX$ is defined by 
\begin{equation}\label{widthdef}
 \width(v):= \sup\{ d\bigl(u(0),w(0)\bigr)\colon u,w\in [v]\}=\diam\left(C_{(v^-v^+)}\right).\end{equation}
Notice that if  $\XX$ is CAT$(-1)$ then for all $v\in\SX$ we have  $[v]=\{v\}$ and hence $\width(v)=0$; in general, if $v(\RR)$ is contained in an isometric image of a Euclidean plane, then 
the width of $v$ is infinite. 

This motivates the following definitions: A geodesic line $v\in \SX$  is called  
{\hl rank one} if its width is finite; it is said to have zero width if $\width(v)=0$.   In the sequel we will use as in \cite{Ricks} the notation 
\begin{align*}
\mathcal{R}&:=\{v\in \SX\colon v\ \text{is rank one}\}\quad\text{respectively}\\
\mathcal{Z}&:=\{v\in \SX\colon v\ \text{is rank one of zero width}\}.
\end{align*}
We remark that the existence of a rank one geodesic  imposes severe restrictions on the Hadamard space $\XX$. For example,  $\XX$ can neither be a symmetric space or  Euclidean building of higher rank  nor a product of Hadamard spaces.

The following important lemma states that even though we cannot join any two distinct points in the geometric boundary $\rand$ of the Hadamard space $\XX$ by a geodesic in $\XX$, given a rank one geodesic we can at least join all points in a neighborhood of its end points by a geodesic in $\XX$. More precisely,  we have the following result which is a reformulation of  Lemma~III.3.1 in \cite{MR1377265}:
\begin{lemma}[Ballmann]\label{joinrankone} \
Let $v\in\mathcal{R}$ be a rank one geodesic and $c>\width(v)$.
Then there exist open disjoint neighborhoods $U^-$ of $\,v^-$ and $U^+$ of $\,v^+$ in $\ganz$ with the following properties:
If $\xi\in U^-$ and $\eta \in U^+$ then there exists a rank one geodesic joining $\xi$ and $\eta$. For any such geodesic $w\in\mathcal{R}$ we have   $d(w(t), v(0))< c$ for some $t\in\RR$ and $\width(w)\le 2c$.
\end{lemma}
This lemma implies that the set $\reg$ is open in $\SX$; we emphasize that $\zero$ in general need not be an open subset of $\SX$: In every open neighborhood of a {\hl zero width} rank one geodesic there may exist a rank one geodesic of arbitrarily small but strictly positive width.


Let us now get back to the Hopf parametrization map defined in~(\ref{HopfPar}): 
As stated in \cite[Proposition 5.10]{Ricks} the $\is(\XX)$-action on $\SX$ descends to an action on $\ndpt\SX \times \RR=\Hopf_x(\SX)$ by homeomorphisms via 
\[ \gamma (\xi,\eta, s):=\bigl(\gamma \xi,\gamma \eta, s+\bs_{\gamma\xi}(\gamma x,x)\bigr)\quad\text{for}\quad \gamma\in\is(\XX).\]

Moreover, the action of $\is(\XX)$ is well-defined on  the set of equivalence classes  $[\SX]$ of elements in $\SX$, and the (well-defined) map
\begin{equation}\label{equivHopf} [\SX]\to \ndpt\SX\times \RR,\quad [v]\mapsto \Hopf_x(v)\end{equation}
is an $\is(\XX)$-equivariant homeomorphism. 
 For convenience we will frequently identify  $
\ndpt\SX\times \RR$ with  $[\SX]$. 
We also remark that the end point map $\ndpt:\SX\to \rand\times \rand$ induces a well-defined map $[\SX]\to\rand\times\rand$ which we will also denote $\ndpt$. 

As in Definition~5.4 of \cite{Ricks} we will say that a  sequence $(v_n)\subset\SX$ {\hd  converges weakly} to $v\in \SX$ if and only if 
  \begin{equation}\label{defweakconvergence}
  v_n^-\to v^-,\quad v_n^+\to v^+\quad\text{and }\  \bs_{v_n^-}\bigl(v_n(0),x\bigr)\to \bs_{v^-}\bigl(v(0),x\bigr);\end{equation}
  notice that this definition is independent of the choice of $x\in\XX$. 
 Obviously, weak convergence $v_n\to v$ is equivalent to the convergence $[v_n]\to [v]$ in $[\SX]$, and  $v_n\to v$ in $\SX$ always implies $[v_n]\to [v]$ in $[\SX]$.  
 
 We will also need the following 
result due to R.~Ricks, which implies that the  restriction of the Hopf parametrization map~(\ref{HopfPar})   to the subset $\mathcal{R}$ is proper:
\begin{lemma}[\cite{Ricks}, Lemma~5.9]\label{weakimpliesstrong}
If a sequence $(v_n)\subset\SX$ converges weakly to \underline{$v\in\reg$}, then some subsequence of $(v_n)$ converges to some $u\in\SX$ with $u\sim v$. 
\end{lemma}

 The topological space $\SX$ can be endowed with the {\hd geodesic flow} $(g^t)_{t\in\RR}$  which  is naturally defined by reparametrization of $v\in \SX$. In particular we have
\[ (g^t v)(0)=v(t) \quad\text{for all } \ v\in \SX\quad\text{and all }\ t\in\RR.\]
The geodesic flow induces a flow on the set of equivalence classes $[\SX]$ which we will also denote $(g^t)_{t\in\RR}$; via the $\is(\XX)$-equivariant homeomorphism $[\SX]\to\ndpt\SX\times \RR\,$ the action of the geodesic flow $(g^t)_{t\in\RR}$ on $[\SX]$ is equivalent to the translation action on the last factor of $\ndpt\SX \times \RR$ given by
\begin{equation}\label{geodflowHopf}
 g^t (\xi,\eta,s):=(\xi,\eta, s+t).\end{equation}


\section{Rank one isometries and rank one groups}\label{rankonegroups}

As in the previous section we let $(\XX,d)$ be a proper Hadamard space and denote $\is(\XX)$ the isometry group of $\XX$. 

\begin{definition}\label{hypaxiso}\
An isometry $\gamma\in\is(\XX)$ 
 is called {\hd axial}  if there exists a constant
$\ell=\ell(\gamma)>0$ and a geodesic $v\in \SX$ \st $\gamma v=g^{\ell} v$. 
We call
$\ell(\gamma)$ the {\hd translation length} of $\gamma$, and $v$
an {\hd invariant geodesic} of $\gamma$. The boundary point
$\gamma^+:=v^+$ (which is independent of the chosen invariant geodesic $v$) is called the {\hd attractive fixed
point}, and $\gamma^-:=v^-$ the {\hd repulsive fixed
point} of $\gamma$. 

An axial isometry $h$ is called {\hd rank one} if one (and hence any) invariant geodesic of $h$ belongs to  $\reg$; the  {\hd width} of $h$ is then defined as the width of an arbitrary invariant geodesic of $h$. 
\end{definition}
Notice that if $\gamma\in\is(\XX)$ is axial, then $\ndpt^{-1}(\gamma^-,\gamma^+)\subset\SX$ is the set of parametrized invariant geodesics of $\gamma$, and every axial isometry $\widetilde\gamma$ commuting with $\gamma$ satisfies $p\circ \ndpt^{-1}(\widetilde\gamma^-,\widetilde\gamma^+)=p\circ \ndpt^{-1}(\gamma^-,\gamma^+)$.  If $h$ is rank one, then  the fixed point set of $h$ equals $\{h^-, h^+\}$; moreover,  if $g$ is an axial isometry  commuting with $h$, then $g$ and $h$ clearly generate a virtually cyclic subgroup of  $\is(\XX)$. 

The following important lemma describes the north-south dynamics of rank one isometries:
\begin{lemma}[\cite{MR1377265}, Lemma III.3.3]\label{dynrankone} \
\ Let $h$ be a rank one isometry. Then
\begin{enumerate}
\item[(a)]  every point $\xi\in\rand\setminus\{h^+\}$ can be joined
to $h^+$ by a geodesic, and all these geodesics are  rank one,
\item[(b)] given neighborhoods $U^-$ of $h^-$ and $U^+$ of $h^+$ in $\ganz$
there exists $N\in\NN$ \st
 $\  h^{-n}(\ganz\setminus U^+)\subset U^-$ and
$h^{n}(\ganz\setminus U^-)\subset U^+$ for all $n\ge N$.
\end{enumerate}
\end{lemma}

We next prepare for an extension of part (a) of the lemma above,  
which replaces the fixed point $h^+$ of the rank one isometry $h$ by the end point of a certain geodesic:
\begin{definition}[compare Section~5 in \cite{Ricks}]\label{weakstrongrecurrencedef} Let $G<\is(\XX)$ be any subgroup. An element   $v\in\SX$ is said to {\hd (weakly) $G$-accumulate} on $u\in\SX$ if  
there exist sequences $(g_n)\subset G$ and $(t_n)\nearrow \infty$  \st $g_n g^{t_n} v$ converges (weakly) to $u$ as $n\to\infty$; $v$ is said to be {\hd (weakly) $G$-recurrent} if $v$ (weakly) $G$-accumulates on $v$. 
 \end{definition} 
Notice that if $v$ is an invariant geodesic of an axial isometry $\gamma\in\is(\XX)$, then $v$ is $\langle \gamma\rangle$-recurrent and hence in particular $\is(\XX)$-recurrent. Moreover, if $v\in\SX$ weakly $G$-accumulates on  \underline{$u\in\reg$}, then by Lemma~\ref{weakimpliesstrong} $v$ $G$-accumulates on some element $w\sim u$.  In particular, if \underline{$v\in\zero$} is weakly $G$-recurrent, then it is already $G$-recurrent.

The following statements show the relevance of the previous notions. 
%
\begin{lemma}[see Section~6 in \cite{Ricks} or  Lemma~3.11 in \cite{LinkHTS}]\label{jointoweakrecurrent}
If $v\in\reg$ is weakly\break $\is(\XX)$-recurrent then for every $\,\xi\in\rand\setminus\{v^+\}$ there exists $w\in\reg$ satisfying
\[ \width(w)\le \width(v)\quad\text{and }\quad  (w^-,w^+)=(\xi, v^+).\] 
\end{lemma}
We will also need the following generalization of a statement originally due to G.~Knieper in the manifold setting; recall the definiton of the distance function $d_1$ from (\ref{metriconSX}). 
\begin{lemma}[Lemma~7.1 in \cite{LinkHTS} or  Proposition~4.1 in \cite{MR1652924}] \label{KniepersProp}\ 
 Let $u\in {\mathcal Z}$ be an $\is(\XX)$-recurrent rank one geodesic of zero width. Then for all $v\in \SX$ with $v^+=u^+$ and $\bs_{v^+}(v(0),u(0))=0$ we have
$$ \lim_{t\to\infty} d_1(g^t v, g^tu)=0.$$
\end{lemma}

We will now deal with discrete subgroups $\Gamma$ of the isometry group $\is(\XX)$ of $\XX$.   The {\hl geometric limit set} $\Lim$ of $\Gamma$  is defined by
$\Lim:=\overline{\Gamma\cdot x}\cap\rand,$
where $x\in\XX$ is an arbitrary point.  

If $\XX$ is a CAT$(-1)$-space, then a  discrete group $\Gamma<\is(\XX)$ is called {\hl non-elementary} if its limit set $L_\Gamma$ is infinite. 
It is well-known that this implies that $\Gamma$ contains two axial isometries  with disjoint fixed point sets (which are actually rank one of zero width as $\SX=\zero$ for any CAT$(-1)$-space). In the general setting this motivates the following 
\begin{definition}
We say that  two rank one isometries $g,h\in\is(\XX)$ are {\hd independent} if and only if $\{g^+,g^-\}\cap \{h^+,h^-\}\ne\emptyset$ (see for example Section~2 of \cite{MR2629900} and Section~2 in \cite{MR2585575}).  Moreover, 
 a group $\Gamma< \is(\XX)$ is called  
{\hl rank one} if $\Gamma$ contains a pair of independent rank one elements.
\end{definition}
Obviously, if $\XX$ is CAT$(-1)$ then every non-elementary discrete isometry group 
is rank one. In  general however, the notion of rank one group seems very restrictive at first sight. Nevertheless we have the following weak hypothesis which ensures that a discrete group $\Gamma<\is(\XX)$ is rank one:
\begin{lemma}[\cite{LinkHTS}, Lemma~4.4]\label{inflimset}
If  $\Gamma<\is(\XX)$ is a discrete 
subgroup with infinite limit set $\Lim$ containing the positive end point $v^+$  of a weakly $\is(\XX)$-recurrent element $v\in\reg$, then $\Gamma$ is a rank one group.
\end{lemma}
Notice that the conclusion is obviously true when $v^+$ is a fixed point of a rank one isometry of $\XX$. 

We will now define an important subset of the limit set $\Lim$ of $\Gamma$. For that we let $x$, $y\in\XX$ arbitrary.  
A point $\xi\in\rand$ is called a {\hd radial limit point} if there exists $c>0$ and  sequences $(\gamma_n)\subset\Gamma$ and $(t_n)\nearrow\infty$ such that
\begin{equation}\label{radlimpoint} d\bigl(\gamma_n y, \sigma_{x,\xi}(t_n)\bigr)\le c\quad\text{for all }\ n\in\NN.\end{equation}
Notice that by the triangle inequality this condition is independent of the choice of $x$, $y\in\XX$.
The {\hd radial limit set} $\radlim\subset\Lim$ of $\Gamma$ is defined as the set of radial limit points. 

We will further denote
\begin{equation}\label{definezerorec} \zero_\Gamma^{\small{\mathrm{rec}}}:=\{v\in\zero\colon v\ \text{and } -v\ \text{are }  \Gamma\text{-recurrent}\}\end{equation}
the set of zero width parametrized geodesics which are $\Gamma$-recurrent in both directions. Notice that if $v\in\zero$ is weakly $\Gamma$-recurrent, then it is already $\Gamma$-recurrent according to the remark below Definition~\ref{weakstrongrecurrencedef}. We will also need the following
\begin{definition}\label{posnegrec}
An element $v\in\quotient{\Gamma}{\SX}$ is called {\hd positively and negatively recurrent}, if it possesses a lift $\widetilde v\in\SX$ \st both $\widetilde v$ and $-\widetilde v$ are $\Gamma$-recurrent.
\end{definition}

%

\section{Geodesic currents and Ricks' measure}\label{geodcurrentmeasures}

In this section we want to describe the construction of Ricks' measure from an arbitrary geodesic current on $\ndpt\reg$. We will also recall the properties of the Ricks' measure which are relevant for our purposes.  Our main references here are \cite[Section~7]{Ricks} and \cite[Section~5]{LinkHTS}.  

From here on $\XX$ will always be a proper Hadamard space and $\Gamma<\is(\XX)$ a discrete rank one group with 
\[  {\mathcal Z}_\Gamma:=\{v\in\zero\colon v^-,v^+\in\Lim\}\ne \emptyset. \]
Notice that according to Proposition~1 in \cite{LinkHTS} the latter hypothesis is always satisfied when $\XX$ is geodesically complete. For later use we further fix a point $\xo\in\XX$.

Recall that the support of a Borel measure $\nu$ on a topological space $Y$ is defined as the set
\begin{equation}\label{supportdef}
\supp(\nu)=\{y\in Y \colon \nu(U)>0\ \text{ for every open neighborhood}\  U \ \text{of}\ y\}.\end{equation}
We also recall that a set $A\subset Y$ is said to have full $\nu$-measure, if $\nu(Y\setminus A)=0$.

We start with two finite Borel measures $\mu_-$, $\mu_+$ on $\rand$ with $\supp(\mu_\pm)
=\Lim$, and let  $\overline\mu$ be a $\Gamma$-invariant Radon measure on $\ndpt\reg$ which is absolutely continuous with respect to the product measure $(\mu_-\otimes \mu_+)\ein_{\lower 0.4ex\hbox{$\scriptstyle\ndpt\reg$}}$.  Such $\overline\mu$ is called a {\hl quasi-product geodesic current}  on $\ndpt\reg$ (see for example Definition~5.2 of \cite{LinkHTS}). 

Following Ricks' approach we can define a Radon measure $\overline m=\overline\mu\otimes \lambda $ on $[\reg]\cong \ndpt\reg\times \RR$, where $\lambda$ denotes Lebesgue measure on $\RR$.  Now according to Lemma~\ref{joinrankone} $\Gamma$ acts  properly on $[\reg]\cong\ndpt\reg\times \RR$ which  admits a proper metric. Since the action is by homeomorphisms and  preserves the Borel  measure $\overline m=\overline\mu\otimes\lambda$, there is (see for instance, \cite[Appendix A]{RicksThesis}) a unique Borel quotient measure $\overline m_\Gamma$ on $\quotient{\Gamma}{[ \reg ] }$ 
satisfying the characterizing property
\begin{equation}\label{chardesmeasure} \int_{\overline A} \widetilde h\d \overline m=\int_{\tiny\quotient{\Gamma}{[\reg]}} \bigl( h\cdot f_{\overline A}\bigr) \d \overline m_\Gamma\end{equation}
for all Borel sets  $ \overline A\subset [\reg]$ and $\Gamma$-invariant Borel maps 
$ \widetilde h:[\reg]\to [0,\infty]$ and\break $\widetilde f_{\overline A}:[\reg]\to [0,\infty]$ defined by $\widetilde f_{\overline A}([v]):= \#\{\gamma\in\Gamma\colon \gamma [v]\in \overline A\}$ for $[v]\in\reg$,  and with   $h$ and $f_{\overline A}$ the maps on $\quotient{\Gamma}{[\reg]}$ induced from $\widetilde h$ and $\widetilde f_{\overline A}$. 


Our final goal is to construct from a weak Ricks' measure $\overline m_\Gamma$ a geodesic flow invariant measure on $\quotient{\Gamma}{\SX}$. 
So let us first remark that $\zero\subset\reg$ is a Borel subset  by 
semicontinuity of the width function~(\ref{widthdef}) (see Lemma~8.4 in \cite{Ricks}), and
that $\Hopf_\xo\ein_{\lower 0.4ex\hbox{$\scriptstyle\zero$}}:\zero\to\ndpt\zero\times\RR\cong [\zero]$ is a homeomorphism onto its image; hence $ [\zero]\subset[\reg]$ is also a Borel subset.  So if $\quotient{\Gamma}{[\zero]}$ 
has positive mass with respect to the weak Ricks' measure $\overline m_\Gamma$   we may define (as in \cite[Definition~8.12]{Ricks})   a geodesic flow and $\Gamma$-invariant measure $m^0$ on $\SX$ by setting 
\begin{equation}\label{defstrongRicks}
 m^0(E):= \overline m \bigl(\Hopf_\xo(E\cap \zero)\bigr)\quad\text{for any Borel set }\ E\subset\SX;
 \end{equation}
this measure $m^0$ then induces the {\hd Ricks' measure} $m^0_\Gamma$ on $\quotient{\Gamma}{ \SX}$. 

Notice that in general $\overline m_\Gamma (\quotient{\Gamma}{[\zero]})= 0\ $ is possible; obviously this is always the case when $\zero=\emptyset$. However, 
  Theorem~6.7 and Corollary~2  in \cite{LinkHTS} immediately imply
\begin{theorem}\label{currentisproduct}
Let $\XX$ be a proper Hadamard space, and $\Gamma<\is(\XX)$ a discrete rank one group with $\zero_\Gamma\ne \emptyset$ (which is always the case if $\XX$ is geodesically complete). Let $\mu_-$, $\mu_+$ be non-atomic, finite Borel measures on $\rand$ with 
$\mu_\pm (\radlim)=\mu_\pm(\rand)$, and $\overline\mu\sim\bigl( \mu_-\otimes \mu_+\bigr)\ein_{\lower 0.4ex\hbox{$\scriptstyle\ndpt\reg$}}$ a  quasi-product geodesic current.

Then for the set $\zero_\Gamma^{\small{\mathrm{rec}}}$ defined in~(\ref{definezerorec})  we have  
 \[ (\mu_-\otimes\mu_+)(\ndpt\zero_\Gamma^{\small{\mathrm{rec}}})=(\mu_-\otimes\mu_+)(\ndpt\reg) =\mu_-(\rand)\cdot\mu_+(\rand),\] 
 and in particular $\overline\mu\sim \mu_-\otimes\mu_+$.
 \end{theorem}
 So in this case the Ricks' measure $m^0_\Gamma$ is actually 
equal to the weak Ricks' measure $\overline m_\Gamma$ used for its construction. Moreover, for the measure $m$ on $\SX$, from which the Ricks' measure descends, we have the formula 
\begin{equation}\label{measureformula} m(E)=\int_{\ndpt\zero } \lambda\bigl(p(E)\cap (\xi\eta)\bigr) \d \overline\mu(\xi,\eta),\end{equation}
where $\lambda$ again denotes Lebesgue measure, and $E\subset\SX$ is an arbitrary Borel set.
We further remark that if $\XX$ is a manifold, then the Ricks' measure is also equal to Knieper's measure $m^{\text{\scriptsize Kn}}_\Gamma$ associated to $\overline\mu$ which descends from 
\[ m^{\text{\scriptsize Kn}}(E):= \int_{\ndpt\SX} \vol_{(\xi\eta)}\bigl(p(E)\cap(\xi\eta)\bigr)\mathrm{d}\overline\mu(\xi,\eta)\quad\text{for any Borel set }\  E\subset \SX,\]
where $\vol_{(\xi\eta)}$ denotes the induced Riemannian volume element on the submanifold $(\xi\eta)\subset\XX$. 

From here on we will therefore denote the Ricks' measure  $m_\Gamma$ instead of $m_\Gamma^0$. 

%
For la\-ter reference we want to summarize what we know from Theorem~7.4 and Lemma~7.5 in \cite{LinkHTS}. Before we can state the result we denote 
${\mathcal B}(R)\subset\SX$  the set of all parametrized geodesics $v\in\SX$ with origin  $pv=v(0)\in B_R(\xo)$ 
and define
\begin{equation}\label{Deltadef} \Delta:=\sup \Big\{ \frac{\ln  \overline\mu\bigl(\ndpt{\mathcal B}(R)\bigr)}{R}\colon R>0\Big\}. \end{equation}

\begin{theorem}\label{propertiesofRicks}
 Let $\XX$, $\Gamma<\is(\XX)$, $\mu_-$, $\mu_+$ and $\overline\mu\,$ as in Theorem~\ref{currentisproduct}. We further assume that the constant $\Delta$ defined via (\ref{Deltadef}) 
  is finite. 
Then the dynamical systems $(\rand\times\rand, \Gamma,\mu_-\otimes\mu_+)$,  $\bigl(\ndpt \SX,\Gamma, \overline\mu\bigr)\,$ and $\bigl(\quotient{\Gamma}{\SX},g_\Gamma, m_\Gamma\bigr)$ are ergodic. 
\end{theorem} 
We will repeatedly make use of the following argument, which is immediate by Fubini's theorem:  
\begin{corollary}\label{useFubini}
 Let $\XX$, $\Gamma<\is(\XX)$, $\mu_-$, $\mu_+$, $\overline\mu$ and $\,\Delta<\infty$ as in Theorem~\ref{propertiesofRicks}. Then if $\Omega\subset\quotient{\Gamma}{\zero}$ is a subset of full $m_\Gamma$-measure, and 
$\widetilde\Omega\subset\zero$ the preimage of $\Omega$ under the projection map $\zero\mapsto \quotient{\Gamma}{\zero}$, the sets
\begin{align*}
E^-&:= \{\xi\in\rand\colon (\xi,\eta')\in \ndpt\widetilde\Omega\ \text{ for }\ \mu^+\text{-almost every }\ \eta'\in\rand\}\quad\text{and}\\
E^+&:= \{\eta\in\rand\colon (\xi',\eta)\in \ndpt\widetilde\Omega\ \text{ for }\ \mu^-\text{-almost every }\ \xi'\in\rand\}
\end{align*}
satisfy $\,\mu_-(E^-)=\mu_-(\rand)\ $ and $\,\mu_+(E^+)=\mu_+(\rand)$. 
\end{corollary}

Before proving the important Lemma~\ref{lem:fundwith0measureboundary} 
we want to recall a few notions from topology and geometric group theory: If $Y$ is a topological space, then a collection of subsets of $Y$ is said to be {\hl locally finite}  if every $y\in Y$ has an open neighborhood that intersects only finitely many sets in the collection. Notice that if 
the collection $\{U_\lambda\colon \lambda\in\Lambda\}\subset Y$ (with $\Lambda$ a countable set) is locally finite, then the collection of the closures $\{\overline{U_\lambda}\colon \lambda\in\Lambda\}\subset Y$ is also locally finite. Moreover, for the closure of the countable union $\bigcup_{\lambda\in\Lambda} U_\lambda$ we have 
 \begin{equation}\label{genclosureofinfiniteunion} \overline{{\textstyle \bigcup_{\lambda\in\Lambda}} U_\lambda}= {\textstyle \bigcup_{\lambda\in\Lambda}} \overline{ U_\lambda}.\end{equation}
 Indeed, if $(y_n)\subset {\textstyle\bigcup_{\lambda\in\Lambda}} U_\lambda$ is a sequence converging to a point $y\in Y$, we let $U\subset Y$ be an open neighborhood of $y$ such that $U\cap U_\lambda=\emptyset$ for all but finitely many $\lambda\in\Lambda$; denote the finite set of exceptions by  
 \[F:=\{\lambda\in\Lambda\colon U\cap U_\lambda\ne \emptyset\}.\]  
 Then for $n$ sufficiently large we have
 \begin{align*}
  y_n & \in U\cap {\textstyle \bigcup_{\lambda\in\Lambda}} U_\lambda\subset {\textstyle \bigcup_{\lambda\in F}} U_\lambda,\quad\text{hence }\ \ 
  y  \in \overline{{\textstyle \bigcup_{\lambda\in F}} U_\lambda}={\textstyle \bigcup_{\lambda\in F}} \overline{U_\lambda}
  \subset {\textstyle \bigcup_{\lambda\in\Lambda}}  \overline{ U_\lambda}.
\end{align*}
 The converse inclusion is trivial.

Assume now that a discrete group $G$ acts by isometries on a proper metric 
space $Y$. An open set $D\subset Y$ is called a {\hl fundamental domain} for the action of $G$ on $Y$, if 
\[Y=\bigcup_{g\in G} g\cdot \overline {D},\quad\text{and }\  \,gD\cap D=\emptyset\quad\text{for all }\ g\in G\setminus\{e\};\] 
it is said to be locally finite (for the action of $G$ on $Y$) if the collection of sets $\{g\cdot D\colon g\in G\}$ is locally finite; notice that this is equivalent to the fact that for any compact set $K\subset Y$ the number 
\[ \#\{g\in G \colon K\cap g\cdot \overline{D}\ne\emptyset\}\]
is finite.


%
%
%
For our purposes we will  need a 
fundamental domain for the action of $\Gamma$ on 
$\SX$ whose boundary is negligable with respect to the measure $m$ on $\SX$ inducing the Ricks' measure on $\quotient{\Gamma}{\SX}$ (which is defined by (\ref{measureformula})):
\begin{lemma}\label{lem:fundwith0measureboundary}
 Let $\XX$, $\Gamma<\is(\XX)$, $\mu_-$, $\mu_+$, $\overline\mu$ and $\,\Delta<\infty$ as in Theorem~\ref{propertiesofRicks}. 
Then there exists a $\Gamma$-invariant subset $\SX'\subset\SX$ of full $m$-measure and a locally finite 
 fundamental domain $\mathcal{D}\subset \SX'$ for the action of $\,\Gamma$ on $\SX'$  which satisfies  $m(\partial\mathcal{D})=0$. 
\end{lemma}
\begin{proof}
We denote 
\[\mathcal{F}:= \{v\in\SX\colon \gamma v=v \ \text{ for some}\ \gamma\in\Gamma\setminus\{e\}\}\]
the set of parametrized geodesics in $\SX$ which are fixed by a non-trivial element in $\Gamma$. Notice that this set is non-empty only if $\Gamma$ contains elliptic elements. 

Obviously $\mathcal{F}$ is closed, $\Gamma$-invariant and invariant by the geodesic flow. Moreover, $\mathcal{F}\cap \zero$ is a proper subset of the support of $m$. 
By ergodicity of $m_\Gamma$ we conclude that  
$m(\mathcal{F})=0$. 

Choose a point $x\in\XX$ with trivial stabilizer in $\Gamma$. Let  ${\mathcal D}_\Gamma\subset\SX $ denote the open {\hl Dirichlet domain} for $\Gamma$ with center $x$, that is the set of all parametrized geodesic lines with origin in 
\[ \{z\in\XX\colon d(z,x)< d(z,\gamma x)\ \text{ for all}\ \gamma\in\Gamma\setminus\{e\}\};\]
then by choice of $x$ we have 
\[ \gamma {\mathcal D}_\Gamma\cap {\mathcal D}_\Gamma=\emptyset\ \text{ for all}\  \gamma\in\Gamma\setminus\{e\}.\] 
Moreover, ${\mathcal D}_\Gamma$ is locally finite as $\XX$ is proper and $\Gamma $ is discrete. 
Notice that in general ${\mathcal D}_\Gamma$ need  not be a fundamental domain for the action of $\Gamma$ on $\SX$, because   
\[{\textstyle\bigcup_{\gamma\in\Gamma}} \gamma \overline{{\mathcal D}_\Gamma}\subsetneq \SX\]
is possible as the following example provided by the anonymous referee shows: If $\XX$ is the universal cover of a bouquet of circles of length $1$ (that is a regular tree),  $\Gamma<\is(\XX)$ the group of deck transformations (which does not contain elliptic elements) and $x\in\XX$ the midpoint of an edge $E$ of $\XX$, then the closure of the Dirichlet domain ${\mathcal D}_\Gamma\subset\SX $ with center $x$ consists of all parametrized geodesics $v$ with origin $v(0)\in\overline{ E}$  and $E\subset v(\RR)$. But if  $w\in\SX$ is a parametrized geodesic with $w(\RR)\cap \overline{E}=\{w(0)\}$, then $w\notin \bigcup_{\gamma\in\Gamma} \gamma \overline{{\mathcal D}_\Gamma}$.  

For this reason we consider the ``enlarged boundary"
\begin{align*}
 \widetilde \partial {\mathcal D}_\Gamma =\{v\in\SX\colon & d(v(0),x)=d(v(0),\gamma x) \ \text{ for some}\ \gamma\in\Gamma\setminus\{e\}\\
 &\hspace*{4mm} \text{and }\ d(v(0),x)\le d(v(0),\gamma x) \ \text{ for all}\ \gamma\in\Gamma\}\supset \partial {\mathcal D}_\Gamma\end{align*}
and use the set 
\[\widehat {  {\mathcal D}_\Gamma} :={\mathcal D}_\Gamma \cup \widetilde\partial {\mathcal D}_\Gamma\]
instead of the closure $\overline{ {\mathcal D}_\Gamma}$ of the Dirichlet domain. Then obviously 
\[ {\textstyle \bigcup_{\gamma\in\Gamma} } \gamma \widehat{{\mathcal D}_\Gamma}= \SX,\quad\text{and }\ \ {\mathcal F}\cap\widehat{{\mathcal D}_\Gamma}\subset\widetilde \partial {\mathcal D}_\Gamma.\]

However, the problem is that in general the boundary $\partial {\mathcal D}_\Gamma$ of the Dirichlet domain (and also the enlarged boundary $\widetilde \partial {\mathcal D}_\Gamma$) is very complicated, and in particular $m(\widetilde \partial {\mathcal D}_\Gamma)\ge m( \partial {\mathcal D}_\Gamma) >0$ is  possible.  

In order to  get a fundamental domain with boundary of zero 
$m$-measure 
 we will therefore modify the Dirichlet domain ${\mathcal D}_\Gamma$ in a neighborhood of the enlarged boundary $\widetilde \partial {\mathcal D}_\Gamma$ as proposed by T.~Roblin (\cite[p.~13]{MR2057305}): We first choose a covering of $\widetilde \partial {\mathcal D}_\Gamma \setminus \mathcal{F} $ by a locally finite 
family of open sets 
$\{V_n\colon n\in\NN\}\subset \SX\setminus \mathcal{F}$ with a uniform upper  bound on the  diameter with respect to the distance function $d_1$ introduced in (\ref{metriconSX}) such that for all $n\in\NN$ 
we have 
\[ m(\partial V_n)=0\quad\text{and }\ \overline{V_n}\cap \gamma\overline{V_n}=\emptyset\quad\text{for all } \ \gamma\in\Gamma\setminus\{e\}.\] 

We first claim that the family of subsets 
$\{\Gamma\cdot V_n\colon n\in\NN\}\subset \SX\setminus \mathcal{F}$ is still locally finite.  
For the proof we choose for each $j\in\NN$ a point $v_j\in V_j\cap\widetilde\partial{\mathcal{D}_\Gamma}$; there exists $r>0$ such that $V_j\subset B_r(v_j)$ for all $j\in\NN$. Since the map $p:\SX\to\XX,\ v\mapsto v(0)\ $ is $1$-Lipschitz, we also have $pV_j\subset B_r(pv_j)$ for all $j\in\NN$. Moreover, 
$v_j\in \widetilde\partial{\mathcal{D}_\Gamma}$ implies that $d(pv_j,\gamma x)\ge d(pv_j,x)$ for all $\gamma \in\Gamma$.

Now assume that the family $\{\Gamma\cdot V_n\colon n\in\NN\}\subset \SX\setminus \mathcal{F}$ is not locally finite. Then there exists an open set $U\subset \SX\setminus \mathcal{F}$ and  infinite sets $\{ \gamma_k\colon k\in\NN\}\subset \Gamma$, $\{j_k\colon k\in\NN\}\subset \NN$ such that  $U\cap \gamma_k V_{j_k}\ne \emptyset$ for all $k\in\NN$.
Let $R>0$  such that $p U\subset B_R(x)$, where $x$ is the center of the Dirichlet domain. For  $k\in\NN$ we pick $u_k\in   U\cap \gamma_k V_{j_k}$; passing to a subsequence if necessary we can assume that $(u_k)$ converges to a point $u\in \overline{U}$. Since $\Gamma$ is discrete and $\{ \gamma_k\colon k\in\NN\}$ is infinite, we know that $d(x,\gamma_k^{-1} p u)\to\infty$ as $k\to\infty$, hence for all $k$ sufficiently large  we have
$\gamma_k^{-1} p u_k\notin \overline{B_{R+2r}(x)}$. 

Let $k\in \NN$ such that $ d(x,\gamma_k^{-1} p u_k)>R+2r$. 
Notice that  $u_{k}\in \gamma_k V_{j_k}\subset \gamma_k B_r(v_{j_k})$ implies $d(\gamma_k^{-1} pu_k,  p v_ {j_k})<r$ and hence 
 \[ d(x, p v_{j_k} )\ge d( x, \gamma_k^{-1} p u_k )- d( p v_ {j_k},\gamma_k^{-1} p u_k)  >R+2r-r=R+r.\]
By choice of $v_{j_k}\subset 
\widetilde\partial{\mathcal{D}_\Gamma}$ we further know that $d(pv_{j_k},\gamma x)\ge d(pv_{j_k},x)$ for all $\gamma \in\Gamma$, hence in particular
\[ d(x,\gamma_k pv_{j_k}) \ge  d(x,pv_{j_k}) >R+r,\]
and therefore 
\begin{align*} 
d(x, p u_k) & \ge d(x, \gamma_k p v_{j_k})- d( pu_k, \gamma_k p v_{j_k})>
d(x, p v_{j_k})-r
>R;\end{align*}
this is an obvious contradiction to  $p U\subset  B_{R}(x)$.

We are now going to construct the desired fundamental domain. We start with the Dirichlet domain ${\mathcal D}_0:={\mathcal D}_\Gamma\,$ from above and   set ${\mathcal D}_1:=\bigl({\mathcal D}_{0}\setminus \Gamma \cdot \overline{V_1}\bigr)\cup V_1\subset\SX\setminus \mathcal{F}$. This set is open as a union of two open sets, and it is still locally finite  for the action of $\Gamma$ on $\SX\setminus \mathcal{F}$; obviously we have $\gamma \cdot {\mathcal D}_1\cap {\mathcal D}_1=\emptyset$ for all $\gamma\in\Gamma\setminus\{e\}$. Hence  defining ${\mathcal D}_n:=\bigl({\mathcal D}_{n-1}\setminus \Gamma\cdot \overline{V_n}\bigr)\cup V_n$ for $n\in \NN$, we get a sequence  of open subsets of $\SX\setminus \mathcal{F}$ each of which is locally finite for the action of $\Gamma$ on $\SX\setminus \mathcal{F}$. The limit of this sequence exists and equals 
 \[ {\mathcal D}=\bigl({\mathcal D}_0\setminus \cup_{i=1}^\infty \Gamma\cdot\overline{V_i}\bigr)\sqcup {\textstyle \bigcup}_{j=1}^\infty \bigl(V_j\setminus \cup_{i>j} \Gamma\cdot\overline{V_i}\bigr).\] 
We claim that ${\mathcal D}$  is a locally finite fundamental domain  for the action of $\Gamma$ on $\SX\setminus \mathcal{F}$, but now with boundary $\partial{\mathcal D}$  of $m$-measure zero as it is contained in 
\[ {\textstyle \bigcup_{j=1}^\infty} \Gamma\cdot \partial V_j \cup \mathcal{F}.\]

We first show that $\SX\setminus \mathcal{F}\subset\Gamma\cdot\overline{\mathcal D}$: So let $v\in\SX\setminus \mathcal{F}$ arbitrary. As $\bigcup_{\gamma\in\Gamma} \gamma\cdot \widehat{ {\mathcal D}_0}=\SX$ we may assume without loss of generality that $v\in\widehat{{\mathcal D}_0}$. For $v\in {\mathcal D}_0\setminus \cup_{i=1}^\infty \Gamma\cdot \overline{V_i}\subset{\mathcal D}\,$ 
we are done, so let 
\[\displaystyle v\in\widetilde\partial {\mathcal D}_0\cup{\textstyle \bigcup_{i=1}^\infty }\Gamma\cdot \overline{V_i}\subset {\textstyle \bigcup_{i=1}^\infty }\Gamma\cdot \overline{V_i}\] 
(as $\{V_n\colon n\in\NN\}$ is an open covering of $\widetilde\partial  {\mathcal D}_0\setminus \mathcal{F}$).
Let $\ell\in\NN$ be the largest integer such that $v\in \Gamma\cdot \overline{V_\ell}$; such $\ell$ exists by local finiteness of the family $\{\Gamma\cdot \overline{V_n}\colon n\in\NN\}$. Hence for some $\gamma\in\Gamma$ we have $v\in \gamma  \overline{V_\ell}\setminus \cup_{i>\ell} \Gamma\cdot \overline{V_i}\subset  \gamma \overline{\mathcal D}$, which proves the claim.

We next show that $\mathcal{D}$ is locally finite for the action of $\Gamma$ on $\SX\setminus\mathcal{F}$. Notice that $\mathcal{D}\subset {\mathcal D_0}\cup {\textstyle \bigcup_{j=1}^\infty} V_j$; as ${\mathcal D_0}$ is locally finite   it suffices to prove that the collection of sets $\{\gamma\cdot \textstyle \bigcup_{j=1}^\infty V_j\colon\gamma\in\Gamma\}\subset \SX\setminus\mathcal{F}$ is locally finite. But this follows directly from the local finiteness of the family of sets $\{\Gamma\cdot V_n\colon n\in\NN\}\subset \SX\setminus\mathcal{F}$. 


We finally show that $\partial{ \mathcal D}\subset \bigcup_{j=1}^\infty \Gamma\cdot \partial V_j \cup \mathcal{F}$. As 
\[\partial{ \mathcal D}\subset \partial \bigl({\mathcal D}_0\setminus \cup_{i=1}^\infty \Gamma\cdot \overline{V_i}\bigr) \cup \partial \Bigl({\textstyle \bigcup_{j=1}^\infty }\bigl(V_j\setminus \cup_{i>j} \Gamma\cdot \overline{V_i}\bigr)\Bigr)\cup \mathcal{F},\] 
the claim will follow from the inclusion
\[ \partial \bigl({\textstyle \bigcup_{i=1}^\infty }\Gamma\cdot V_i\bigr)\cap \SX\setminus\mathcal{F}\subset {\textstyle \bigcup_{i=1}^\infty }\Gamma\cdot \partial V_i .\] 
But $v\in \partial  \bigl({\textstyle \bigcup_{i=1}^\infty }\Gamma\cdot V_i\bigr) \cap \SX \setminus \mathcal{F}$ implies $v\notin \bigcup_{i=1}^\infty \Gamma\cdot V_i$ and 
\[v\in \overline{ {\textstyle \bigcup_{i=1}^\infty }\Gamma \cdot V_i} \cap \SX \setminus \mathcal{F} \subset  {\textstyle \bigcup_{i=1}^\infty }\Gamma \cdot \overline{V_i}\] 
according to (\ref{genclosureofinfiniteunion}), 
hence the assertion is true.\\[-4mm]
\end{proof}

\section{Mixing of the Ricks' measure}\label{Mixing}

Let $\XX$ be a proper Hadamard space as before, and  $\Gamma<\is(\XX)$ a discrete rank one group with $\zero_\Gamma\ne \emptyset$. Notice that if $\XX$ is geodesically complete, then according to Proposition~1 in \cite{LinkHTS} the latter condition is automatically satisfied. We further fix a point $\xo\in\XX$. 

From here on we will assume that 
$\mu_-$, $\mu_+$ are non-atomic, finite Borel measures on $\rand$ with 
$\mu_\pm (\radlim)=\mu_\pm(\rand)$. 
We will further require that for the quasi-product geodesic current $\overline\mu\sim \bigl(\mu_-\otimes \mu_+\bigr)\ein_{\lower 0.4ex\hbox{$\scriptstyle\ndpt\reg$}}$ on $\ndpt\reg$ the constant $\Delta$ defined in (\ref{Deltadef}) is finite. 

From Theorem~\ref{currentisproduct} and Definition~\ref{posnegrec} we immediately get that the  set
\[ 
 \{ u\in\quotient{\Gamma}{\SX} \colon  u\ \text{is positively and negatively recurrent}\}\]
has full $m_\Gamma$-measure (which is equivalent to conservativity of the dynamical system $\bigl(\ndpt\SX ,g_\Gamma, \overline\mu\bigr)$). Moreover, according to Theorem~\ref{propertiesofRicks} the dynamical system  $\bigl(\ndpt\SX ,g_\Gamma, \overline\mu\bigr)$ is ergodic and we can use its Corollary~\ref{useFubini}. 


Our proof of mixing will closely follow M.~Babillot's idea from \cite{MR1910932}.  However, as she only gives the proof for cocompact rank one isometry groups of Hadamard {\hl manifolds}, for the convenience of the reader we want to give a detailed proof  in our more general setting, which includes arbitrary discrete rank one isometry groups of non-Riemannian Hadamard spaces. We also emphasize that her set $\reg$ in \cite{MR1910932} is defined as the set of unit tangent vectors $v\in S\XX\cong\SX$ which do not admit a parallel perpendicular Jacobi field;  
this is in general a proper open subset of our set $\reg$ (which was defined as the set of  parametrized geodesic lines with finite width) which is contained in $\zero$. 
In particular, her Proposition-Definition below\break \cite[Lemma~2]{MR1910932} is not true when considering our set $\reg$ instead of hers.  We therefore have to work on the set  $\zero$ (which is not open in $\reg $) and use -- up to a constant factor -- the 
cross-ratio introduced by R.~Ricks in \cite[Definition~10.2]{Ricks} instead of Babillot's. 

From the Busemann function introduced in (\ref{buseman}) we first define 
for $(\xi,\eta)\in\ndpt\SX$ the {\hl Gromov product} of $(\xi,\eta)$ 
with respect to $y\in \XX$ via
\begin{equation}\label{GromovProd}
\Gr_y(\xi,\eta)=\frac12\bigl(\bs_\xi(y,z)+\bs_\eta(y,z)\bigr), 
\end{equation}
where $z\in (\xi\eta)$ is an arbitrary point on 
a geodesic  line joining $\xi$ and $\eta$. 
It is related to R.~Ricks' definition following \cite[Lemma~5.1]{Ricks} via the formula \\
$ \Gr_y(\xi,\eta)=-2 \beta_y(\xi,\eta)$ for all $(\xi,\eta)\in\ndpt\SX$.
We then make the following
\begin{definition}[Definition~10.1  in \cite{Ricks}]\label{Quadrilateraldef}
A quadrupel of points $(\xi_1,\xi_2,\xi_3,\xi_4)\in\bigl(\rand\bigr)^4$ is called a {\hl quadrilateral}, if there exist $v_{13}$,  $v_{14}$, $v_{23}$, $v_{24} \in\reg$ \st
\[ \ndpt v_{ij} =(\xi_i,\xi_j)\quad\text{for all }\ (i,j)\in\{(1,3), (1,4), (2,3), (2,4)\}.\] 
The set of all quadrilaterals is denoted $\mathcal{Q}$, and we define
\[ \mathcal{Q}_\Gamma= \mathcal{Q}\cap \bigl(\Lim\bigr)^4.\]
\end{definition}
\begin{definition}[compare Definition~10.2 in \cite{Ricks}]\label{Crossratiodef}\ \\ 
For a quadrilateral $(\xi,\xi',\eta,\eta')\in\mathcal{Q}$ we define its {\hd cross-ratio} by
\[ \Cr(\xi,\xi',\eta,\eta')=\Gr_\xo(\xi,\eta)+\Gr_\xo(\xi',\eta')- \Gr_\xo(\xi,\eta')-\Gr_\xo(\xi',\eta).\]
\end{definition}
Notice that our definition corresponds to Ricks' via
\[ \Cr(\xi,\xi',\eta,\eta')=- 2\mathrm{B}(\xi,\xi',\eta,\eta').\]
The properties of a cross-ratio listed in Proposition~10.5 of \cite{Ricks} are therefore satisfied for our cross-ratio $\Cr$. We further have 
\begin{lemma}[Lemma~10.6 in \cite{Ricks}]
If $g\in\is(\XX)$ is axial, then its translation length $\ell(g)$ is given by
\[ \ell(g)=\Cr(g^-,g^+,\xi,g\xi).\]
\end{lemma}
From this we immedately get the following
\begin{proposition}\label{lengthsubsetcr}
The length spectrum $ \{\ell(\gamma)\colon \gamma\in\Gamma\}$ of $\Gamma$ is a subset of the cross-ratio spectrum $\Cr(\mathcal{Q}_\Gamma)$.
\end{proposition}

\begin{theorem}\label{mixthm}
Let $\Gamma<\is(\XX)$ be a discrete rank one group with  non-arithmetic length spectrum and $\zero_\Gamma\ne\emptyset$.  Let $\mu_-$, $\mu_+$ be non-atomic finite Borel measures on $\rand$ with 
$\mu_\pm(\radlim)=\mu_\pm(\rand)$, and   
\[\overline\mu\sim 
(\mu_-\otimes\mu_+)\ein_{\lower 0.4ex\hbox{$\scriptstyle\ndpt\reg$}}\]
 a quasi-product geodesic current defined on $\ndpt\reg$ for 
which the constant $\Delta$ defined by~(\ref{Deltadef}) 
is finite. 
Let $m_\Gamma$ be the associated Ricks' measure  on $\quotient{\Gamma}{ \SX}$. 
Then the dynamical system $(\quotient{\Gamma}{ \SX}, g_\Gamma, m_\Gamma)$ is mixing, that is for all Borel sets $A,B\subset\quotient{\Gamma}{ \SX}$ with $m_\Gamma(A)$ and $m_\Gamma(B)$ finite  we have (with the abbreviation $\Vert m_\Gamma\Vert =m_\Gamma\big(\quotient{\Gamma}{ \SX}\bigr)$)
\[ \lim_{t\to\pm \infty} m_\Gamma(A\cap g_\Gamma^{-t} B)=\left\{\begin{array} {cl}\displaystyle \frac{m_\Gamma(A)\cdot m_\Gamma(B)}{\Vert m_\Gamma\Vert} & \text{ if } \ m_\Gamma  \text{ is finite},\\[3mm]
0 & \text{ if } \ m_\Gamma  \text{ is infinite}.\end{array}\right.\]
\end{theorem}
%
\begin{proof}
We first remark that mixing is equivalent to the fact that for every square integrable function $\varphi\in\LL^2(m_\Gamma)$ on $\quotient{\Gamma}{\SX}$ the functions $\varphi\circ g_\Gamma^t$ converge weakly in $\LL^2(m_\Gamma)$ to the constant
\[ \frac{1}{\Vert m_\Gamma\Vert }\int \varphi \d m_\Gamma\]
as $t\to\pm\infty$. 
Moreover, since the continuous functions with compact support are dense in $\LL^2(m_\Gamma)$ it suffices to show that for every $f\in \Cnt_c(\quotient{\Gamma}{\SX})$
\[ f \circ g_\Gamma^t \to  \frac{1}{\Vert m_\Gamma\Vert }\int f \d m_\Gamma\]  weakly in $\LL^2(m_\Gamma)$ as $t\to\pm\infty$.

We argue by contradiction and assume that $m_\Gamma $ is not mixing. Then there exists a function $f\in \Cnt_c(\quotient{\Gamma}{\SX})$ (without loss of generality we may assume $\int f\d m_\Gamma=0\,$ if $m_\Gamma$ is finite) and a sequence $(t_n)\nearrow \infty$ \st $f \circ g_\Gamma^{t_n}$ does not converge to $0$ weakly in $\LL^2(m_\Gamma)$ as $n\to\infty$. By   \cite[Lemma~1]{MR1910932} there exists a sequence $(s_n)\nearrow\infty$ and a {\hl non-constant} function $\Psi\in \LL^2(m_\Gamma)$ \st 
\[ f \circ g_\Gamma^{s_n} \to \Psi\quad\text{and }\ f \circ g_\Gamma^{-s_n} \to \Psi\]
weakly in $\LL^2(m_\Gamma)$ as $n\to\infty$. 
Without loss of generality we may assume that $\Psi$ is defined on all of $\quotient{\Gamma}{\SX}$. 
%
Let $\widetilde\Psi:\SX\to\RR$ denote the lift of $\Psi$  to $\SX$ and smooth it
along the flow by considering for $\tau>0$ the function
\[ \widetilde\Psi_\tau:\widetilde \Omega \to\RR,\quad  v\mapsto \int_0^\tau \widetilde \Psi(g^s  v)\d s.\]
For fixed $\varepsilon>0$ sufficiently small $\widetilde\Psi_\varepsilon$ 
is still non-constant,  and now there  exists a set $E''\subset \ndpt \SX$ of full $\overline\mu$-measure \st for all 
$ v\in\ndpt^{-1}E''$ the function 
\[h_{ v}:\RR\to \RR,\quad  t\mapsto \widetilde\Psi_\varepsilon(g^t  v) \] 
is continuous. Notice that  according to Theorem~\ref{currentisproduct} we can assume $E''\subset \ndpt \zero_\Gamma^{\small{\mathrm{rec}}}$ as $\ndpt \zero_\Gamma^{\small{\mathrm{rec}}} $ has full  $\overline\mu$-measure in $\ndpt \SX$. To any such function we associate the set of its periods  which is a closed subgroup of $\RR$; it only depends on  $( v^-, v^+)\in E''$. This gives a map from $E''$ into the set of closed subgroups of $\RR$ which is $\Gamma$-invariant as $\widetilde\Psi_\varepsilon$ is. By ergodicity of $\overline\mu\,$ (Theorem~\ref{propertiesofRicks}) this map is constant $\overline\mu$-almost everywhere. 

Assume that this constant image is the group $\RR$. Hence for  $\overline\mu$-almost every $( v^-, v^+)\in E''$ every real number is a period of $h_v$ for some $v\in\ndpt^{-1}(v^-, v^+)$ which is only possible if $h_{ v}$ is independent of $t$. In this case $\widetilde\Psi_\varepsilon$ induces a $\Gamma$-invariant function on a subset $E'\subset E''\subset \ndpt \zero_\Gamma^{\small{\mathrm{rec}}}$ of full $\overline\mu$-measure. Again by ergodicity of $\overline\mu\,$ this function is constant, which finally gives a contradiction to the fact that $\widetilde\Psi_\varepsilon$ is non-constant. So we conclude that there exists a subset $E'\subset \ndpt\zero_\Gamma^{\small{\mathrm{rec}}}$  of full $\overline\mu$-measure 
and $a\ge 0$ such that the constant image of the map above restricted to $E'$ is the closed subgroup $2 a\ZZ$. 

In order to get the desired contradiction, we will next show that the cross-ratio spectrum $\Cr(\mathcal{Q}_\Gamma)$ is contained in the closed subgroup $a\ZZ$. 
We denote $\widetilde f:\SX\to\RR$  the lift of $f$  to $\SX$, and define
 \[\widetilde f_\varepsilon:\SX \to\RR, \quad  v\mapsto \int_0^\varepsilon \widetilde f(g^s  v)\d s.\]
 Since $\widetilde f$ is $\Gamma$-invariant, $\widetilde f_\varepsilon$ is also $\Gamma$-invariant and therefore descends to a function $f_\varepsilon$ on $\quotient{\Gamma}{\SX}$. 
Moreover, 
\[ f_\varepsilon\circ g_\Gamma^{s_n}\to \Psi_\varepsilon\quad\text{and}\quad  f_\varepsilon\circ g_\Gamma^{-s_n}\to \Psi_\varepsilon\]
weakly in 
$ \LL^2(m_\Gamma)$ as $n\to\infty$, where $\Psi_\varepsilon\in \LL^2(m_\Gamma)$ is the function induced from the $\Gamma$-invariant function $\widetilde\Psi_\varepsilon$  above. 
%
%
According to the classical fact stated and proved in \cite[Section~1]{MR1910932}  there exists a  sequence $(n_k)\subset\NN$ \st  $\Psi_\varepsilon$ is the almost sure limit of the Cesaro averages for positive and negative times 
\[\frac{1}{K^2}\sum_{k=1}^{K^2} f_\varepsilon \circ g_\Gamma^{s_{n_k}}\quad\text{and }\quad \frac{1}{K^2}\sum_{k=1}^{K^2} f_\varepsilon \circ g_\Gamma^{-s_{n_k}}.\]
We denote $\widetilde \Psi_\varepsilon^+$, $\widetilde \Psi_\varepsilon^-$ the lifts of the almost sure limits of the Cesaro averages above  and consider the set  
\begin{align*}
\widetilde \Omega &:=\{ u\in\zero_\Gamma^{\small{\mathrm{rec}}} \colon  
\widetilde\Psi_\varepsilon^+(u),\ \widetilde\Psi_\varepsilon^-(u)  \ \text{ exist and  }\  \widetilde\Psi_\varepsilon^+(u)=\widetilde\Psi_\varepsilon^-(u)=\widetilde\Psi_\varepsilon(u)\} ;\end{align*}
from the previous paragraph and the fact that $\ndpt \zero_\Gamma^{\small{\mathrm{rec}}}$ has full $\overline\mu$-measure  we know that $\ndpt\widetilde\Omega$ has full $\overline\mu$-measure. 
The same is true for the set  $E:=E'\cap \ndpt\widetilde\Omega$, where $E'\subset  \ndpt\zero_\Gamma^{\small{\mathrm{rec}}}$ is the set of full $\overline\mu$-measure from the first part of the proof. So in particular  $v\in\ndpt^{-1}E$ implies that the periods of the continuous function $h_v\in\Cnt(\RR)$ are contained in the closed subgroup $2a\ZZ$. 

Since $\widetilde f$ is the lift of  a function $f\in \Cnt_c(\quotient{\Gamma}{\SX})$, both $\widetilde f$ and $\widetilde f_\varepsilon$ are uniformly continuous. So if $ u,v\in\widetilde \Omega\subset\ndpt^{-1}E$ are arbitrary, then according to Lemma~\ref{KniepersProp} we have the following statements:
\begin{itemize}
\item[(a)] 
If  $ u^+= v^+$ and  $\bs_{ v^+}( u(0), v(0))=0$, then 
$\ \widetilde \Psi_\varepsilon^+( u)= \widetilde \Psi_\varepsilon^+( v)$.
\item[(b)] 
If  $ u^-= v^-$ and  $\bs_{ v^-}( u(0), v(0))=0$, then  
$\ \widetilde \Psi_\varepsilon^-( u)= \widetilde \Psi_\varepsilon^-( v)$.
\end{itemize}
Now according to Corollary~\ref{useFubini}  the sets
\begin{align*}
E^-&:= \{\xi\in\rand\colon (\xi,\eta')\in E\ \text{ for }\ \mu^+\text{-almost every }\ \eta'\in\rand\}\quad\text{and}\\
E^+&:= \{\eta\in\rand\colon (\xi',\eta)\in E\ \text{ for }\ \mu^-\text{-almost every }\ \xi'\in\rand\}
\end{align*}
satisfy $\mu_-(E^-)=\mu_-(\rand)$, $\mu_+(E^+)=\mu_+(\rand)$, hence $E^-\times E^+$ has full $\overline\mu$-measure. 

We first consider the set of special quadrilaterals
\begin{align*}
 \mathcal{S} &=\{(\xi,\xi',\eta,\eta')\colon (\xi,\eta)\in E\cap (E^-\times E^+),\
  (\xi',\eta'), (\xi,\eta'), (\xi',\eta)\in E\}\subset \mathcal{Q}_\Gamma. \end{align*} 
So let $(\xi,\eta)\in E\cap (E^-\times E^+)$ and choose $(\xi',\eta')\in E\,$ such that  $(\xi',\eta)$ and $(\xi,\eta')$ also belong to $E$. In order to show that the cross-ratio
$ \Cr(\xi,\xi',\eta,\eta')$ belongs to $a\ZZ$ we start with a geodesic $v\in\ndpt^{-1}(\xi,\eta)$. 

Let $v_1\in \ndpt^{-1}(\xi',\eta)$ \st $\bs_{\eta}\bigl(v(0), v_1(0)\bigr)=0$, $v_2\in \ndpt^{-1}(\xi',\eta')$ \st $\bs_{\xi'}\bigl(v_1(0), v_2(0)\bigr)=0$, $v_3\in \ndpt^{-1}(\xi,\eta')$ \st $\bs_{\eta'}\bigl(v_2(0), v_3(0)\bigr)=0$ and finally
$v_4\in \ndpt^{-1}(\xi,\eta)$ \st $\bs_{\xi}\bigl(v_3(0), v_4(0)\bigr)=0$. Then according to (a)
\[ \widetilde \Psi_\varepsilon^+(v)=\widetilde \Psi_\varepsilon^+(v_1)=\widetilde \Psi_\varepsilon^-(v_1)\]
by choice of $\widetilde\Omega$. Moreover (b) gives
\[ \widetilde \Psi_\varepsilon^-(v_1)=\widetilde \Psi_\varepsilon^-(v_2)=\widetilde \Psi_\varepsilon^+(v_2).\]
Again by (a) we get
\[ \widetilde \Psi_\varepsilon^+(v_2)=\widetilde \Psi_\varepsilon^+(v_3)=\widetilde \Psi_\varepsilon^-(v_3)\]
and by (b)
\[\widetilde \Psi_\varepsilon^-(v_3)=\widetilde \Psi_\varepsilon^-(v_4)=\widetilde \Psi_\varepsilon^+(v_4).\]
Altogether this shows $\widetilde \Psi_\varepsilon(v_4)=\widetilde \Psi_\varepsilon(v)$, and since $\ndpt v_4=\ndpt v$ we know that there exists $t\in\RR$ \st $v=g^t v_4$. Hence $t$ is a period of the function $h_v$ and therefore  $t\in 2 a\ZZ$ (as $\ndpt v\in E'$). On the other hand, we have 
\begin{align*}
2\Cr(\xi,\xi',\eta,\eta') &= 2 \bigl(\Gr_\xo(\xi,\eta)+\Gr_\xo(\xi',\eta')- \Gr_\xo(\xi,\eta')-\Gr_\xo(\xi',\eta)\bigr)\\
&= \bs_{\xi}(\xo, v(0))+\bs_{\eta}(\xo, v(0)) + \bs_{\xi'}(\xo, v_2(0))+\bs_{\eta'}(\xo, v_2(0)) \\
&\quad -  \bs_{\xi}(\xo, v_3(0))-\bs_{\eta'}(\xo, v_3(0)) - \bs_{\xi'}(\xo, v_1(0))-\bs_{\eta}(\xo, v_1(0))\\
&= \underbrace{\bs_{\eta}(v_1(0), v(0))}_{=0} + \underbrace{\bs_{\xi'}(v_1(0), v_2(0))}_{=0}+\underbrace{\bs_{\eta'}(v_3(0), v_2(0))}_{=0}\\
&\quad +   \underbrace{\bs_\xi(v_4(0), v_3(0))}_{=0}+ \bs_{\xi}(v_3(0), v(0))\\
& = \bs_{\xi}(v_4(0), v(0)) =\bs_\xi(v_4(0),v_4(t))=t\in 2a\ZZ,
\end{align*}
hence $\Cr(\xi,\xi',\eta,\eta') \in a\ZZ$. This proves that $\Cr(\mathcal{S})\subset a\ZZ$.

Finally, since the  cross-ratio is continuous and the set of special quadrilaterals
$ \mathcal{S}$ is dense in $\mathcal{Q}_\Gamma$, the cross-ratio spectrum $\Cr(\mathcal{Q}_\Gamma)$ is included in the discrete subgroup $a\ZZ$ of $\RR$. So according to Proposition~\ref{lengthsubsetcr} the length spectrum is arithmetic in contradiction to the hypothesis of the theorem.
\end{proof}

We will often work in the universal cover $\XX$ of $\quotient{\Gamma}{\XX}$ and therefore need the following 

\begin{corollary}\label{mixcor}
Let $\Gamma<\is(\XX)$ be a discrete rank one group with  non-arithmetic length spectrum and $\zero_\Gamma\ne\emptyset$.  Let $\mu_-$, $\mu_+$ be non-atomic finite Borel measures on $\rand$ with 
$\mu_\pm(\radlim)=\mu_\pm(\rand)$, and   
\[\overline\mu\sim 
(\mu_-\otimes\mu_+)\ein_{\lower 0.4ex\hbox{$\scriptstyle\ndpt\reg$}}\]
 a quasi-product geodesic current defined on $\ndpt\reg$ for 
which the constant $\Delta$ defined in (\ref{Deltadef}) is  finite. 
Let $m_\Gamma$ be the associated Ricks' measure  on $\quotient{\Gamma}{ \SX}$, $A$, $B\subset\quotient{\Gamma}{\SX}$ Borel sets with  
 $m_\Gamma(A)$ and $m_\Gamma(B)$ finite, and $\widetilde A$, $\widetilde B\subset\SX$ lifts of $A$ and $B$. Then  
\[ \lim_{t\to\pm \infty} \Bigl( \sum_{\gamma\in\Gamma} m(\widetilde A\cap g^{-t}\gamma \widetilde B)\Bigr)=\left\{\begin{array} {cl}\displaystyle \frac{m(\widetilde A)\cdot m(\widetilde B)}{\Vert m_\Gamma\Vert} & \text{ if } \ m_\Gamma  \text{ is finite},\\[3mm]
0 & \text{ if } \ m_\Gamma  \text{ is infinite}.\end{array}\right.\]
\end{corollary}
\begin{proof}
Accordung to Lemma~\ref{lem:fundwith0measureboundary} there exists a  $\Gamma$-invariant subset  $\SX'\subset\SX$ of full $m$-measure and a locally finite fundamental domain $\mathcal{D}\subset\SX'$  for the action of $\Gamma$ on $\SX'$   with $m(\partial{\mathcal D})=0$. 
Notice that for any measurable function  
$h\in \LL^1(m_\Gamma)$ with lift $\widetilde h:\SX\to\RR$ 
the integral $\int_{\mathcal D} \widetilde h \d m$ is independent of the chosen fundamental domain ${\mathcal D}\subset\SX$  
as above. Moreover, we obviously get  from (\ref{chardesmeasure}) and (\ref{defstrongRicks})
 \[ \int_{{\mathcal D}} \widetilde h \d m =\int_{\scriptsize\quotient{\Gamma}{\SX}} h\d m_\Gamma.\] 

Now let
 $A$, $B\subset\quotient{\Gamma}{\SX}$ be Borel sets with  
 $m_\Gamma(A)$ and $m_\Gamma(B)$ finite, and $\widetilde A$, $\widetilde B\subset\SX$ lifts of $A$ and $B$.  Without loss of generality we may assume that $\widetilde A$, $\widetilde B\subset \overline{\mathcal D}$. For $t\in\RR$ consider the function $h_t \in \LL^1(m_\Gamma)$ 
 defined by 
 \[ h_t=\mathbbm{1}_{A\cap g_\Gamma^{-t} B}.\]
 For its lift $\widetilde h_t$ and $v\in\SX$ we have
  \[ \widetilde h_t(v)=1\quad\text{if }\ \gamma' v\in \widetilde A\cap g^{-t}\gamma\widetilde B\ \text{ for some } \gamma',\gamma\in\Gamma,\]
 and $\widetilde h_t(v)=0$ otherwise. So 
 \begin{align*}
 m_\Gamma (A\cap g_\Gamma^{-t}B)&= \int_{\scriptsize\quotient{\Gamma}{\SX}} h_t \d m_\Gamma =\int_{\mathcal D} \widetilde h_t \d m =\sum_{\gamma\in\Gamma} m(\widetilde A\cap g^{-t}\gamma\widetilde B).
 \end{align*}
The claim now follows from Theorem~\ref{mixthm}, because 
\[m_\Gamma(A)= \int_{\scriptsize\quotient{\Gamma}{\SX}} \mathbbm{1}_A \d m_\Gamma =\int_{\mathcal D} \mathbbm{1}_{\Gamma \widetilde A} \d m =m(\widetilde A)\quad\text{and}\quad m_\Gamma (B)=m(\widetilde B).\]
 \end{proof}
Notice that in general it is not so easy to determine whether a discrete rank one group has arithmetic length spectrum or not. As mentioned before, if $\Gamma<\is(\XX)$ has finite Ricks' Bowen-Margulis measure and satisfies $\Lim=\rand$, then according to Theorem~4 in \cite{Ricks} the length spectrum of $\Gamma$ is arithmetic if and only if $\XX$ is a tree with all edge lengths in $c\NN$ for some $c>0$. This includes Babillot's observation that for cocompact discrete rank one groups  of a Hadamard {\hl manifold} the length spectrum is non-arithmetic. Moreover, we recall a few further results: 

\begin{proposition}
Let $\XX$ be a proper CAT$(-1)$ Hadamard space.  A discrete rank one group $\Gamma<\is(\XX)$ has non-arithmetic length spectrum if
\begin{itemize}
\item  $\Gamma$ contains a parabolic isometry (\cite{MR1617430}),
\item the limit set $\Lim$ possesses a connected component which is not reduced to a point (\cite{MR1341941}),
\item $\XX$ is a manifold with constant sectional curvature (\cite[Proposition~3]{MR841080}),
\item $\XX$ is a Riemannian surface (\cite{MR1703039}).
\end{itemize}

\end{proposition}

\section{Shadows, cones and corridors}\label{shadowconecorridor}

We keep the notation and conditions from the previous section. So in particular $\XX$ is a proper Hadamard space and $\Gamma<\is(\XX)$ a discrete rank one group. 
For our proof of the equidistibution theorem we will need  a few definitions and preliminary statements. 
Recall that for $y\in \XX$ and $r>0$ $B_r(y)\subset\XX$  denotes the open ball of radius $r$ centered at $y\in\XX$.
The {\hd shadow} of $B_r(y)\subset\XX$ viewed from the source $x\in \XX$ is defined by 
\[ {\mathcal O}_{r}(x,y):=\{\eta \in\rand\colon \sigma_{x,\eta}(\RR_+)\cap B_r(y)\neq\emptyset\};\] 
this is an open subset of the geometric boundary $\rand$. If $\xi\in\rand$ we define
\begin{align*} {\mathcal O}_{r}(\xi,y)&:=\{\eta \in\rand\colon \exists \ v\in\ndpt^{-1}(\xi,\eta) \  \text{with }\ v(0)\in B_r(y)\}\\
&= \{\eta\in\rand\colon (\xi,\eta)\in\ndpt\SX\ \text{ and }\ d\bigl(y,(\xi\eta)\bigr)<r\}.\end{align*}
Notice that due to the possible existence of flat subspaces in $\XX$ a shadow ${\mathcal O}_{r}(\xi,y)$ with source $\xi\in\rand$ need not be open: In a Euclidean plane such a shadow always consists of a single point in the boundary, no matter how large $r$ is. In our context, the shadows with source $\xi$ in the boundary $\rand$ will be larger, but still not necessarily open. 
\begin{remark}\label{shadowsfrominfinity}
 If $\xi$ is the positive end point $v^+$ of a weakly $\is(\XX)$-recurrent geodesic $v\in\zero$, then Lemma~\ref{jointoweakrecurrent} and Lemma~\ref{joinrankone} imply that ${\mathcal O}_{r}(\xi,y)$ is open for any $y\in \XX$.

More generally, if there exists a geodesic $u\in \zero$ with $u^+=\xi$ and $u(0)\in B_r(y)$, then according to Lemma~\ref{joinrankone} the shadow ${\mathcal O}_{r}(\xi,y)$ contains an open neighborhood of $u^-$ in $\rand$, but need not be open: If $u$ is not $\is(\XX)$-recurrent, then this open neighborhood of $u^-$ can be much smaller than ${\mathcal O}_{r}(\xi,y)$, and there might exist a point $\eta\in {\mathcal O}_{r}(\xi,y)$ such that $(\xi\eta)$ is isometric to a Euclidean plane.  But  $\xi$ cannot be joined to any point in the boundary of $(\xi\eta)$ different from $\eta$, no matter how close it is to $\eta$. In this case, every open neighborhood of $\eta$ intersects the complement of the shadow ${\mathcal O}_{r}(\xi,y)$ in $\rand$ non-trivially (as this complement includes all the boundary points which cannot be joined to $\xi$ by a geodesic), hence $\eta\in \partial  {\mathcal O}_{r}(\xi,y)$.
\end{remark}

We will now prove that  this cannot happen if $\eta$ is the end point of an $\is(\XX)$-recurrent geodesic $v\in\zero$, that is if $\eta$ belongs to 
the set 
\begin{equation}\label{endpointsofzerowidthrecurrent}
\rand^{\small{\mathrm{rec}}}:=\{\eta\in\rand\colon \exists\, v\in\zero\ \, \is(\XX)\text{-recurrent with}\ \eta=v^+\}.\end{equation}
\begin{lemma}\label{lem:boundaryofshadow}
Let $\xi\in\rand$, $x\in\XX$ and $r>0$  arbitrary. Then for the closure $\overline{{\mathcal O}_{r}(\xi,x)}$ and the boundary $\partial{\mathcal O}_{r}(\xi,x)$ 
of the shadow ${\mathcal O}_{r}(\xi,x)\subset\rand$ we have\vspace*{1mm}
\begin{enumerate}
\item[(a)] $\quad \displaystyle \overline{{\mathcal O}_{r}(\xi,x)} \subset \{\zeta\in\rand\colon  (\xi,\zeta)\in\ndpt\SX\ \text{ and }\  
d\bigl(x, (\xi\zeta)\bigr)\le r\}$, \vspace*{1mm} 
\item[(b)] $\quad \displaystyle \partial {\mathcal O}_{r}(\xi,x)\cap \rand^{\small{\mathrm{rec}}} \subset \{\zeta \in\rand^{\small{\mathrm{rec}}}\setminus\{\xi\}\colon d\bigl(x,(\xi\zeta)\bigr)=r\}.$\vspace*{1mm}
\end{enumerate}
\end{lemma}
\begin{proof}
%
%

In order to prove (a) we let $\zeta\in \overline{ {\mathcal O}_{r}(\xi,x)}$ arbitrary. Then there exists a sequence $(\zeta_n)\subset \mathcal O_{r}(\xi,x)$  with $\zeta_n\to\zeta$ as $n\to\infty$.
  For $n\in\NN$ we let $v_n=v(x;\xi,\zeta_n)\in \SX$ as defined in (\ref{orthogonalproj}), hence in particular $v_n^-=\xi$, $v_n^+=\zeta_n$ and $v_n(0)\in B_r(x)$. Passing to a subsequence if necessary we may assume that $v_n(0)$ converges to a point $z\in \overline{B_r(x)}$ (as $\overline{B_r(x)}$ is compact). Recall the definiton of the Alexandrov angle from (\ref{Alexandrovangle}). According to Proposition~II.9.2 in \cite{MR1744486} we have
\[ \angle_z(\xi,\zeta)\ge \limsup_{n\to\infty}\angle_{v_n(0)}(\xi,\zeta_n)=\pi,\]
since $v_n(0)$ is a point on the geodesic $v_n$ joining $\xi$ to $\zeta_n$. From $\angle_z(\xi,\zeta)\in [0,\pi]$ we therefore get $\angle_z(\xi,\zeta)=\pi$, hence $z\in \overline{B_r(x)}$ is a point on a geodesic joining $\xi$ to $\zeta$, and  in particular $(\xi,\zeta)\in\ndpt \SX $. 
This proves (a).

For the proof of  (b) we let $\zeta\in \partial {\mathcal O}_{r}(\xi,x)\cap \rand^{\small{\mathrm{rec}}}$ be arbitrary. By definition of the boundary we know that $\zeta\in\overline{ {\mathcal O}_{r}(\xi,x)}$ and that there exists a sequence $(\eta_n)\subset \rand\setminus {\mathcal O}_{r}(\xi,x)$ with  $\eta_n\to\zeta$ as $n\to\infty$. From (a) we know that $(\xi,\zeta)\in\ndpt \SX$, hence in particular $\zeta\ne \xi$, and  that
$d\bigl(x,(\xi\zeta)\bigr)\le r$. So it only remains to 
prove that 
$ d\bigl(x,(\xi\zeta)\bigr)\ge r$. 

We will prove that every point $\eta\in \bigl(\rand^{\small{\mathrm{rec}}}\setminus\{\xi\}\bigr)\cap  {\mathcal O}_{r}(\xi,x)$ is an interior point of $ {\mathcal O}_{r}(\xi,x)$: Then 
$d\bigl(x,(\xi\zeta)\bigr)< r$ would imply that $\zeta$ is an interior point of ${\mathcal O}_{r}(\xi,x)$ and therefore cannot be the limit of a sequence 
$(\eta_n)\subset\rand\setminus  {\mathcal O}_{r}(\xi,x)$, in contradiction
to $\zeta\in \partial {\mathcal O}_{r}(\xi,x)$. 

So let $\eta\in \bigl(\rand^{\small{\mathrm{rec}}}\setminus\{\xi\}\bigr)\cap  {\mathcal O}_{r}(\xi,x)$ be arbitrary. From Lemma~\ref{jointoweakrecurrent} we get that   $(\xi,\eta)\in\ndpt\zero$, and with  $v:=v(x;\xi,\eta)\in\zero$ we have 
$d\bigl(x,v(0)\bigr)=d\bigl(x, (\xi\eta)\bigr)< r$. Fix $\varepsilon=\frac12\left(r-d\bigl(x, (\xi\eta)\bigr)\right)>0$. According to Lemma~\ref{joinrankone} there exists an open neighborhood $U\subset\rand$ of $\eta$ such that any $u\in\SX$ with $u^-=\xi$ and $u^+\in U$ satisfies $u\in\reg$ and $d\bigl(v(0),u(\RR)\bigr)<\varepsilon$.  Let $\eta'\in U$ arbitrary and $u\in\ndpt^{-1}(\xi,\eta')$ be parametrized \st $d\bigl(v(0), u(0)\bigr)<\varepsilon$. Then 
\begin{align*} d\bigl(x,(\xi\eta')\bigr)&\le d\bigl(x,u(0)\bigr)\le d\bigl(x,v(0)\bigr)+d\bigl(v(0),u(0)\bigr)<d\bigl(x,(\xi\eta)\bigr) +\varepsilon\\
&<d\bigl(x,(\xi\eta)\bigr)+\frac12\left(r-d\bigl(x,(\xi\eta)\bigr)\right)<r.\end{align*}
\end{proof}

Instead of using the boundary $\partial {\mathcal O}_{r}(\xi,x)$ we will work in the sequel with the set
\begin{equation}\label{otherboundaryofshadow}
 \widetilde\partial {\mathcal O}_{r}(\xi,x):=\{\eta\in\rand\colon  (\xi,\eta)\in\ndpt\SX\ \text{ and }\  d\bigl(x, (\xi\eta)\bigr)=r\}\end{equation}
 whose intersection with $\rand^{\small{\mathrm{rec}}}$ may be strictly larger than
$\partial {\mathcal O}_{r}(\xi,x)\cap \rand^{\small{\mathrm{rec}}}$.  
Notice that every point $\eta\in\bigl(\rand^{\small{\mathrm{rec}}}\setminus\{\xi\} \bigr)\cap\bigl(\rand \setminus\widetilde\partial {\mathcal O}_{r}(\xi,x)\bigr)$  is an interior point of the complement $\rand \setminus \widetilde\partial {\mathcal O}_{r}(\xi,x)$  of $\widetilde\partial {\mathcal O}_{r}(\xi,x)$ in $\rand$. 
\begin{remark}
The converse inclusions ``$\supset$" in the above Lemma~\ref{lem:boundaryofshadow} are  wrong in general: If $\XX$ is a $4$-regular tree with all edge lengths equal to $1$, then 
\[ \overline{ {\mathcal O}_{r}(\xi,x)}= {\mathcal O}_{r}(\xi,x) = \{\eta\in\rand\setminus\{\xi\}\colon d\bigl(x,(\xi\eta)\bigr)\le \lceil{r}\rceil-1\},\]
where $\lceil{r}\rceil\in\NN$ is the smallest integer bigger than or equal to $r$. 
So for $n\in\NN$ we have 
\[ \overline{ {\mathcal O}_{n}(\xi,x)}\subsetneq \{\eta\in\rand\setminus\{\xi\}\colon d\bigl(x,(\xi\eta)\bigr)\le n\}.\]
Moreover, 
\begin{align*} \widetilde\partial {\mathcal O}_{n}(\xi,x)&= \{\eta\in\rand\setminus\{\xi\} \colon    d\bigl(x, (\xi\eta)\bigr)=n\}\\
&= \{\eta\in\rand\setminus\{\xi\} \colon  n\le   d\bigl(x, (\xi\eta)\bigr)<n+1\} \\
&={\mathcal O}_{n+1}(\xi,x)\setminus {\mathcal O}_{n}(\xi,x)\ne\emptyset,\end{align*}
whereas the boundary  $\partial {\mathcal O}_{r}(\xi,x)$ is always empty. Since all points in $\rand$ are $\is(\XX)$-recurrent, this shows that for all $n\in\NN$ 
\[ \emptyset = \partial {\mathcal O}_{n}(\xi,x)\cap \rand^{\small{\mathrm{rec}}} \subsetneq  \widetilde\partial {\mathcal O}_{n}(\xi,x) = \{\zeta \in\rand^{\small{\mathrm{rec}}}\setminus\{\xi\}\colon d\bigl(x,(\xi\zeta)\bigr)=n\}.\]
%
%
\end{remark}

We will further need the following refined versions of the shadows above which were first introduced  by  T.~Roblin (\cite{MR2057305}):
For
$r>0$ and $x,y\in\XX$ 
 we set
\begin{align*}
{\mathcal O}^-_{r}(x,y) &:= \{\eta\in\rand\colon \forall\, z\in B_r(x)\
\mbox{we have}\,\ \sigma_{z,\eta}(\RR_+)\cap B_r(y)\neq\emptyset\},\nonumber \\
{\mathcal O}^+_{r}(x,y) &:= \{\eta\in\rand \colon \exists\, z\in B_r(x)\ \mbox{such that}\,\ 
\sigma_{z,\eta}(\RR_+)\cap B_r(y)\neq\emptyset\}. 
\end{align*}
It is obvious from the definitions that for any $\rho>0$ and for all $x',y'\in\XX$ 
we have
\begin{equation}\label{inclusionoflargeshadows}
d(x,x')<\rho\ \text{ and }\  d(y,y')<\rho\quad\Longrightarrow\quad {\mathcal O}^+_{r}(x,y)\subset {\mathcal O}^+_{r+\rho}(x',y').
\end{equation}

Notice also that   ${\mathcal O}^-_{r}(x,y)$ need not be open as it is an uncountable intersection of open sets 
$ {\mathcal O}_{r}(z,y)$ with $z\in B_r(x)$ (for a concrete example see Remark~\ref{Pittetcounterex} below). 
If $\xi\in\rand$, we set
\[ {\mathcal O}^-_{r}(\xi,y)={\mathcal O}^+_{r}(\xi,y)={\mathcal O}_{r}(\xi,y).\]
\begin{remark}\label{noconvergenceofshadows}
In the middle of page~58 of \cite{MR2057305} it is stated that in a CAT$(-1)$-space $\XX$ every sequence  
$(z_n)\subset\ganz$ converging to a point $\xi\in\rand$ satisfies 
\[ \lim_{n\to\infty}  {\mathcal O}_{r}^{\pm}(z_n,x) = {\mathcal O}_{r}(\xi,x).\] 
This is not true in a CAT$(0)$-space as the following example shows:

Let $\XX$ be the Euclidean plane, $x\in\XX$ the origin $(0,0)$, and  identify $\rand$ with $\mathbb{S}^1\cong [0,2\pi)$. Let $\xi=\pi$ and $r>0$. Then obviously  ${\mathcal O}_{r}(\xi,x)=\{0\}$. 

For $n\in\NN$ we define $\varphi_n:=1/n$ and 
$z_n:=\bigl(-rn\cos(\varphi_n),-rn\sin(\varphi_n)\bigr)$, hence
\[ \sigma_{x,z_n}(-\infty)= \sigma_{z_n,x}(\infty)=\varphi_n \quad\text{and}\quad (z_n)\to\xi=\pi.\]
By elementary Euclidean geometry we further have $\, {\mathcal O}^-_{r}(z_n,x)=\{\varphi_n\}$, and thus 
\[ \lim_{n\to\infty} {\mathcal O}^-_{r}(z_n,x)=\emptyset\ne \{0\}= {\mathcal O}_{r}(\xi,x).\]
\end{remark}
However, the following statement 
will be sufficient for our purposes. 
%
\begin{proposition}\label{liminfsupofshadows}
Let $\xi\in\rand$, $x\in\XX$, $r>0$ and recall the definitions of $\widetilde\partial{\mathcal O}_{r}(\xi,x)$ from (\ref{otherboundaryofshadow}) and of $\rand^{\small{\mathrm{rec}}}$ from (\ref{endpointsofzerowidthrecurrent}). Then for 
every sequence $(z_n)\subset\ganz$ converging to $\xi$  the following inclusions hold:
\begin{enumerate}
\item[(a)]
$\ \displaystyle  \limsup_{n\to\infty} ({\mathcal O}^{\pm}_r(z_n,x)\cap \rand^{\small{\mathrm{rec}}})  \subset \bigl({\mathcal O}_r(\xi,x)\cup \widetilde\partial {\mathcal O}_{r}(\xi,x)\bigr)\cap \rand^{\small{\mathrm{rec}}}$,
\item[(b)]
 $\ \displaystyle  \liminf_{n\to\infty} ({\mathcal O}^{\pm}_r(z_n,x)\cap \rand^{\small{\mathrm{rec}}})  \supset {\mathcal O}_r(\xi,x)\cap \rand^{\small{\mathrm{rec}}}$.
 \end{enumerate}
\end{proposition}
\begin{proof} 
Let us first prove (a), which 
will follow from
\[ \limsup_{n\to\infty} ({\mathcal O}^+_r(z_n,x)\cap \rand^{\small{\mathrm{rec}}})  \subset \bigl({\mathcal O}_r(\xi,x)\cup \widetilde\partial {\mathcal O}_{r}(\xi,x)\bigr)\cap\rand^{\small{\mathrm{rec}}}.\]
 If $\zeta\in  \limsup_{n\to\infty} ({\mathcal O}^{+}_r(z_n,x)\cap\rand^{\small{\mathrm{rec}}})$, then for all $n\in\NN$ there exists $k_n\ge n$ \st $\zeta\in {\mathcal O}^{+}_r(z_{k_n},x)\cap\rand^{\small{\mathrm{rec}}}$. 
  Moreover, by definition of $\rand^{\small{\mathrm{rec}}}$ and 
  Lemma~\ref{jointoweakrecurrent} 
  there exists $w\in\zero$ with $w^-=\zeta$ and $w^+=\xi$. Reparametrizing $w$ if necessary we may assume that its origin $w(0)$ satisfies $\bs_{\zeta}(x,w(0))=0$. 
 
 
 Passing to a subsequence of $(z_{k_n})$ if necessary we may assume that either $(z_{k_n})\subset\rand$  or $(z_{k_n})\subset\XX$. 
 
 Fix $n\in\NN$. If $z_{k_n}\in\rand$ we choose a geodesic line $u_n\in \ndpt^{-1}(\zeta, z_{k_n})$ with $u_n(\RR)\cap B_r(x)\ne \emptyset$. If $z_{k_n}\in\XX$ we first let $\sigma_n$ be a geodesic ray in the class of $\zeta$ with $\sigma_n(0)\in B_r(z_{k_n})$ and $\sigma_n(\RR_+)\cap B_r(x)\ne \emptyset$, and then $u_n\in\SX$  a geodesic line with $u_n^-=\zeta$ whose image in $\XX$ contains  $\sigma_n(\RR_+)$ (that is $-u_n\in\SX$  extends the ray $\sigma_n$). From $\zeta\in \rand^{\small{\mathrm{rec}}}$ and   Lemma~\ref{jointoweakrecurrent} we know that in both cases $u_n\in\zero$;
up to reparametrization we can further assume that $\bs_\zeta(x, u_n(0))=0$.

By choice of $u_n$ we further know that $d(x,u_n(\RR))<r$; we fix $s_n\in\RR$ \st $d(x, u_n(s_n))=d(x,u_n(\RR))$ (which is equivalent to $g^{s_n} u_n= v(x;\zeta,\xi)$). Then
\begin{align}\label{snbounded}
|s_n|&= \big|\bs_\zeta\bigl(u_n(0),u_n(s_n)\bigr)\big|= \big|\bs_\zeta\bigl(x,u_n(s_n)\bigr)\big|\le d\bigl(x,u_n(s_n)\bigr)<r.
\end{align}
%
In the easy case that  $(z_{k_n})\subset \rand$ we have $(u_n^+)=(z_{k_n})\to\xi$, so $(u_n)$ converges weakly to $w\in\zero$.  
 
Otherwise, for $n\in\NN$  we  choose $t_n\in\RR$ \st $u_n(t_n)=\sigma_n(0)\in B_r(z_{k_n})$.  Since $(z_{k_n})$ converges to $\xi$ we also have $u_n(t_n)\to\xi$, hence $(t_n)\nearrow \infty$. From the estimate (\ref{snbounded}) we get $d\bigl(x,u_n(0)\bigr)<2r$, so   $u_n(t_n)\to\xi$  implies $u_n^+\to\xi$, which proves that also in this case $(u_n)$ converges weakly  to $w\in\zero$. 

Passing to a subsequence if necessary we may now assume that the sequence $(s_n)$ from above converges to $s\in [-r,r]$ and that $(u_n)$ converges to $w$ in $\SX$ (by Lemma~\ref{weakimpliesstrong}). This finally gives 
\begin{align*}
 d(x,w(\RR)) &\le d\bigl(x,w(s)\bigr)\\
 &\le \lim_{n\to\infty} \bigl( \underbrace{d\bigl(x,u_n(s_n)\bigr)}_{<r}+ \underbrace{d\bigl(u_n(s_n),w(s_n)\bigr)}_{\to 0}+\underbrace{d\bigl(w(s_n),w(s)\bigr)}_{=|s_n-s|\to 0}\bigr)\le r.  \end{align*}

For the proof of (b) we let $\zeta\in {\mathcal O}_r(\xi,x)\cap\rand^{\small{\mathrm{rec}}}$ be arbitrary. By definition of $\rand^{\small{\mathrm{rec}}}$ and Lemma~\ref{jointoweakrecurrent} there exists  $w\in\zero$ with $w^-=\xi$, $w^+=\zeta$. Reparametrizing $w$ if necessary we may assume that $w=v(x;\xi,\zeta)$, hence $d(x,w(0))<r$. 

Since $B_r(x)$ is open,  there exists $\epsilon>0$ \st $B_{\epsilon}\bigl(w(0)\bigr)\subset B_r(x)$. According to Lemma~\ref{joinrankone} there exist neighborhoods $U$, $V\subset\ganz$ of $w^-$, $w^+$ \st any two points in $U$, $V$ can be joined by a rank one geodesic $u\in\reg$ with $d(u(0),w(0))<\epsilon$ and $\width(u)<2\epsilon$. As $z_n\to\xi=w^-$ there exists $n\in\NN$ \st for all $k\ge n$ we have $B_r(z_k)\subset U$ if $z_k\in\XX$ respectively $z_k\in U$ if $z_k\in\rand$; for these $k$ we immediately get $\zeta=w^+\in {\mathcal O}^-_r(z_k,x)\subset {\mathcal O}^+_r(z_k,x)$ (since $B_{\epsilon}\bigl(w(0)\bigr)\subset B_r(x)$).
\end{proof}

We now fix non-atomic finite Borel measures $\mu_-$, $\mu_+$ on $\rand$ with 
$\mu_\pm(\radlim)=\mu_\pm(\rand)$ and such that $\overline\mu\sim (\mu_-\otimes\mu_+)\ein_{\lower 0.4ex\hbox{$\scriptstyle\ndpt\reg$}}$ is a quasi-product geodesic current on $\ndpt\reg$ for which the constant $\Delta$ defined by (\ref{Deltadef}) is finite. 
We will need the following 
\begin{lemma}\label{boundaryhasmeasurezero}
Let $\xi\in\rand$, $x\in\XX$ and recall definition (\ref{otherboundaryofshadow}). Then the set 
\[\{r>0\colon \mu_+\bigl(\widetilde\partial {\mathcal O}_r(\xi,x)\bigr)>0\}\] is at most countable.
\end{lemma}
\begin{proof}
We first notice that the sets $\widetilde\partial {\mathcal O}_r(\xi,x)$ are disjoint for different values of $r$.
Hence by finiteness of $\mu_x$ we know that for  $n\in\NN$ arbitrary the set
\[A_n:=\{r>0\colon \mu_+(\widetilde\partial {\mathcal O}_r(\xi,x))>1/n\}\]
 is finite.
Therefore the set
$\displaystyle \   \{r>0\colon \mu_+(\widetilde\partial {\mathcal O}_r(\xi,x))>0\}={\textstyle \bigcup_{n\in\NN}} A_n\ $
is at most countable.
\end{proof}
From Proposition~\ref{liminfsupofshadows} 
we get the following  estimate on the $\mu_{+}$-measure of the small and large shadows with source in the neighborhood of a given boundary point.
\begin{corollary}\label{measureofshadowsisclose}
Let $\xi\in\rand$, $x\in\XX$ and $r>0$ \st 
\[ \mu_+\bigl({\mathcal O}_r(\xi,x)\bigr)>0\quad\text{and}\quad \mu_+\bigl(\widetilde\partial{\mathcal O}_r(\xi,x)\bigr)=0.\]
Then for all $\varepsilon>0$ there exists a neighborhood $U\subset\ganz$ of $\xi$ \st for all $z\in U$
\[ \e^{-\varepsilon} \mu_+\bigl({\mathcal O}_r(\xi,x)\bigr)<\mu_+\bigl({\mathcal O}_r^{\pm}(z,x)\bigr) <\e^{\varepsilon} \mu_+\bigl({\mathcal O}_r(\xi,x)\bigr).\]
\end{corollary}
\begin{proof} 
We first recall the definition of $\zero_\Gamma^{\small{\mathrm{rec}}}$ from (\ref{definezerorec}) and notice that $\quotient{\Gamma}{\zero_\Gamma^{\small{\mathrm{rec}}}}$ has full  $m_\Gamma$-measure by Theorem~\ref{currentisproduct}. So according to Corollary~\ref{useFubini} we have  
\[
\mu_+\bigl(\{\zeta\in\rand\colon (\eta,\zeta)\in\ndpt \zero_\Gamma^{\small{\mathrm{rec}}}\ \text{ for } \ \mu_-\text{-almost every }\ \eta\in \rand\}\bigr)=\mu_+(\rand).\]
Hence from the obvious inclusion \[ \{\zeta\in\rand\colon (\eta,\zeta)\in\ndpt \zero_\Gamma^{\small{\mathrm{rec}}}\ \text{ for } \ \mu_-\text{-almost every }\ \eta\in \rand\}\subset \rand^{\small{\mathrm{rec}}}\] 
we obtain $\mu_+(\rand^{\small{\mathrm{rec}}})=\mu_+(\rand)$.  

Since $\mu_+$ is a finite Borel measure, 
Proposition~\ref{liminfsupofshadows}  implies  
\begin{align*}
\mu_+\bigl({\mathcal O}_r(\xi,x)\bigr)& =\mu_+\bigl({\mathcal O}_r(\xi,x)\cap \rand^{\small{\mathrm{rec}}}\bigr)\stackrel{\text{\scriptsize{(b)}}}{\le} 
\mu_+\bigl(\liminf_{n\to\infty} ({\mathcal O}^{\pm}_r(z_n,x)\cap\rand^{\small{\mathrm{rec}}})\bigr)\\
&\le \liminf_{n\to\infty} \mu_+\bigl({\mathcal O}^{\pm}_r(z_n,x)\cap\rand^{\small{\mathrm{rec}}}\bigr)\le \limsup_{n\to\infty} \mu_+\bigl({\mathcal O}^{\pm}_r(z_n,x)\cap\rand^{\small{\mathrm{rec}}}\bigr)\\
&\le  \mu_+\bigl(\limsup_{n\to\infty} ({\mathcal O}^{\pm}_r(z_n,x)\cap\rand^{\small{\mathrm{rec}}})\bigr) \\
&\stackrel{\text{\scriptsize{(a)}}}{\le}   \mu_+\bigl(({\mathcal O}_r(\xi,x)\cup \widetilde\partial {\mathcal O}_r(\xi,x))\cap\rand^{\small{\mathrm{rec}}}\bigr)= \mu_+\bigl({\mathcal O}_r(\xi,x)\bigr),
\end{align*}
because $\mu_+\bigl(\widetilde\partial{\mathcal O}_r(\xi,x)\bigr)=0$. So we conclude
\[ \lim_{n\to\infty} \mu_+\bigl({\mathcal O}^{\pm}_r(z_n,x)\bigr)=\lim_{n\to\infty} \mu_+\bigl({\mathcal O}^{\pm}_r(z_n,x)\cap \rand^{\small{\mathrm{rec}}} \bigr)=\mu_+\bigl({\mathcal O}_r(\xi,x)\bigr),\]
hence the claim.
\end{proof}

For a subset $A\subset \rand$ we next define the small and large cones
\begin{align}\label{slcones}
{\mathcal C}^-_{r}(x,A) &:= \{z\in\XX\colon {\mathcal O}^+_{r}(x,z)\subset A\},\\
{\mathcal C}^+_{r}(x,A) &:= \{z\in\XX \colon {\mathcal O}^+_{r}(x,z)\cap A\ne \emptyset\}.\nonumber
\end{align}
Notice that our definition of the small cones ${\mathcal C}^-_{r}$ differs slightly from Roblin's in order to get Lemma~\ref{smallcones}. Moreover, they are related to our large cones via
\[{\mathcal C}^-_{r}(x,A)\subset {\mathcal C}^+_{r}(x,A)\quad\text{and}\quad {\mathcal C}^-_{r}(x,A)= \ganz\setminus{\mathcal C}^+_{r}(x,\rand\setminus A).\]
From the latter equality  and (\ref{inclusionoflargeshadows}) we immediately get 
\begin{lemma}\label{changepoint}
Let $\rho>0$, $x_0\in B_\rho(x)$ and $y_0\in B_\rho(y)$. 
Then 
\begin{enumerate}
\item[(a)] $\quad y\in {\mathcal C}^+_{r}(x,A)\quad  \Longrightarrow \quad 
y_0 \in {\mathcal C}^+_{r+\rho}(x_0,A)$, 
\item[(b)] $\quad y\in {\mathcal C}^-_{r+\rho}(x,A)\quad  \Longrightarrow \quad 
y_0 \in {\mathcal C}^-_{r}(x_0,A)$.
\end{enumerate}
\end{lemma}
%
This shows in particular that  for  $r<r'$ we have
\begin{equation}\label{coneesti}
{\mathcal C}^+_{r}(x,A)\subset {\mathcal C}^+_{r'}(x,A)\quad\text{and }\quad {\mathcal C}^-_{r}(x,A)\supset {\mathcal C}^-_{r'}(x,A).\end{equation}

In Sections~\ref{equidistribution} and~\ref{orbitcounting} we will frequently need the following 
\begin{lemma}\label{orbitpointsincones}
Let $x,y\in\XX$, $r>0$, and $\widehat V\subset\ganz$, $V\subset\rand$ be arbitrary open sets.
\begin{enumerate}
\item[(a)] For $A\subset \rand$ with $\overline{A}\subset \widehat V\cap\rand$ only finitely many $\gamma\in\Gamma$ satisfy
\[ \gamma y \in {\mathcal C}_r^{\pm} (x,A) \setminus \widehat V.\]
\item[(b)] For $\widehat A\subset \ganz$ with $\overline{\widehat A}\cap\rand\subset V$  only finitely many $\gamma\in\Gamma$ satisfy
\[ \gamma y \in \widehat A\setminus {\mathcal C}_r^{\pm} (x,V).\]
\end{enumerate}
\end{lemma}
\begin{proof}
We begin with the proof of (a) by contradiction. 
Assume that there exists a sequence $(\gamma_n)\subset\Gamma$ \st
$ \gamma_n y \in {\mathcal C}_r^+ (x,A) \setminus \widehat V$ for all $n\in\NN$. As $\Gamma$ is discrete, every accumulation point of $(\gamma_n y)\subset\XX$ belongs to $\rand$. Passing to a subsequence if necessary we will assume that $\gamma_n y\to \zeta\in\Lim\subset\rand$ as $n\to\infty$.   

From $\gamma_n y\in  {\mathcal C}_r^+ (x,A)$ we know that 
$  {\mathcal O}^+_{r}(x,\gamma_n y)\cap A\ne \emptyset$. We choose a geodesic line $v_n\in \SX$ with $v_n^+\in A$ whose image intersects 
$B_r(x)$ and then $ 
B_r(\gamma_n y)$. Up to reparametrization we can assume that $\bs_{v_n^+}(x,v_n(0))=0$ and $\bs_{v_n^+}(\gamma_n y, v_n(t_n))=0$ for some $t_n>0$. 
Then by an easy geometric estimate analogous to the one in the proof of Proposition~\ref{liminfsupofshadows} (a) we have $d(x,v_n(0))<2r$ and  $d\bigl(\gamma_n y, v_n(t_n)\bigr)<2r$.  By convexity of the distance function and $\sigma_{x,v_n^+}(\infty)=v_n(\infty)=v_n^+$ we get
\[ d\bigl(\sigma_{x,v_n^+}(T), v_n(T)\bigr)<2r\quad\text{for all}\quad  T>0.\] Hence
\[ d\bigl(\gamma_n y, \sigma_{x,v_n^+}(t_n)\bigr)\le d\bigl(\gamma_n y,v_n(t_n)\bigr)+d\bigl(v_n(t_n), \sigma_{x,v_n^+}(t_n)\bigr)<4r\]
which implies $v_n^+\to \zeta$ and therefore $\zeta\in\overline{A}\subset \widehat V\cap\rand$. 

On the other hand, as $\widehat V$ is open and $\gamma_n y \notin \widehat V$ for all $n\in\NN$, we obviously have $\zeta\notin \widehat V\cap \rand$, hence a contradiction.
The claim for ${\mathcal C}_r^- (x,A) \setminus \widehat V$ follows from the obvious inclusion ${\mathcal C}_r^- (x,A) \subset {\mathcal C}_r^+ (x,A) $.

For the proof of (b) we assume that there exists a sequence $(\gamma_n)\subset\Gamma$ \st
$ \gamma_n y \in \widehat A\setminus {\mathcal C}_r^- (x,V)$ for all $n\in\NN$. Passing to a subsequence if necessary we will assume as above that $\gamma_n y\to \zeta\in\Lim\subset\rand$ as $n\to\infty$.  Here $\gamma_n y\in \widehat A$ for all $n\in\NN$ obviously implies $\zeta\in\overline{\widehat A}\cap \rand\subset V $.

From $\gamma_n y\notin  {\mathcal C}_r^- (x,V)$ we know that 
$  {\mathcal O}^+_{r}(x,\gamma_n y)\not\subset V$. We choose a geodesic line $v_n\in \SX$ with $v_n^+\not\in V$ whose image intersects 
$B_r(x)$ and then $ 
B_r(\gamma_n y)$. As in the proof of (a) we get $v_n^+\to \zeta$, and therefore 
$\zeta\in\overline{\rand\setminus V}=\rand\setminus V$ since $V$ is open; this is an obvious contradiction to $\zeta\in V$.
Again, the claim for $\widehat A\setminus {\mathcal C}_r^+ (x,V)$ follows from the obvious inclusion ${\mathcal C}_r^+ (x,V) \supset {\mathcal C}_r^- (x,V) $.
%
\end{proof}

Before we proceed we will state some results concerning the following 
corridors first introduced  by  T.~Roblin (\cite{MR2057305}):
For
$r>0$  and $x,y\in\XX$ 
 we set
\begin{align}
{\mathcal L}_{r}(x,y) 
&=  \{(\xi,\eta)\in\ndpt\SX\colon \exists\, v\in\ndpt^{-1}(\xi,\eta)\ \exists\, t>0\ \st \label{Lrc} \\ 
&\hspace*{4.8cm}v(0)\in B_r(x),\ v( t)\in B_r(y) \}.\nonumber
\end{align}
Notice that if $(\xi,\eta)\notin\ndpt\zero$, then the element $v\in\ndpt^{-1}(\xi,\eta)$ satisfying the condition on the right-hand side is in general different from $v(x;\xi,\eta)$  (and from $g^{-t}v(y;\xi,\eta)$).

\begin{remark}\label{Pittetcounterex}
The inclusion ${\mathcal O}^-_{r}(y,x)\times  {\mathcal O}^-_{r}(x,y)\subset {\mathcal L}_{r}(x,y) $ claimed 
in the middle of page 58 of \cite{MR2057305}  is wrong even in the hyperbolic plane $\mathbb{H}^2$ as the following counterexample provided by C.~Pittet shows: Let $x=1+\mathrm{i}$,  $y=\e^4+\mathrm{i} \e^4 $ and $r=d(x,\sqrt2\mathrm{i})=d(y,\sqrt2\e^4\mathrm{i})$ (which is equal to the hyperbolic distance of $x$ respectively $y$ to the imaginary axis). Then elementary hyperbolic geometry shows that the geodesic line \[\sigma:\RR\to \mathbb{H}^2,\quad t\mapsto \e^{t}\mathrm{i}\] satisfies $\sigma(-\infty)\in \mathcal{O}_r^-(y,x)$, $\sigma(\infty)\in \mathcal{O}_r^-(x,y)$, but $\bigl(\sigma(-\infty),\sigma(\infty)\bigr)\notin \mathcal{L}_r(x,y)$ (since $\sigma(\RR)$ is tangent to the open balls $B_r(x)$ and $B_r(y)$). Notice in particular that  none of the sets $\mathcal{O}_r^-(y,x)$, $\mathcal{O}_r^-(x,y)$ is open.\\
As a replacement for the above inclusion we will prove Lemma~\ref{smallcones} below.
\end{remark}
From here on we fix $r>0$, $\gamma\in\is(\XX)$, points $x$, $y\in\XX$ and subsets $ A,B\subset \rand$. 
The following  results relate  corridors to  cones and large shadows. The proof of the first one is straightforward.

\begin{lemma}\label{largecones}
If $\  (\zeta,\xi)\in  {\mathcal L}_{r}(x,\gamma y) \cap  (\gamma B\times A ),$ then 
\begin{align*}
(\gamma y, \gamma^{-1}x) & \in {\mathcal C}^+_{r}(x,A)\times {\mathcal C}^+_{r}(y,B)\ \text{ and}\quad 
( \zeta,\xi )\in  {\mathcal O}^+_{r}(\gamma y,x)\times {\mathcal O}^+_{r}(x,\gamma y).
\end{align*}
\end{lemma}

 \begin{lemma}\label{smallcones}
If $(\gamma y, \gamma^{-1}x)  \in {\mathcal C}^-_{r}(x,A)\times {\mathcal C}^-_{r}(y,B)$, then  
\begin{align*}
 {\mathcal L}_{r}(x,\gamma y) \cap (\gamma B\times A)\supset \{(\zeta, \xi)\in\rand\times\rand\colon \xi\in 
 {\mathcal O}^-_{r}(x,\gamma y),\ \zeta\in {\mathcal O}_{r}(\xi,x)\}.
\end{align*}
\end{lemma}
\begin{proof}
From $\ \zeta\in {\mathcal O}_{r}(\xi,x)\ $ we know that the geodesic line $w=v(x;\xi,\zeta)\in\SX$ defined by (\ref{orthogonalproj}) has origin $w(0)\in B_r(x)$. Then $v:=-w\in\ndpt^{-1}(\zeta,\xi)$ satisfies $v(0)\in B_r(x)$, so  $v^+=\xi\in  {\mathcal O}^-_{r}(x,\gamma y)\ $ implies $v(t)=\sigma_{v(0),\xi}(t)\in B_{r}(\gamma y)$ for some $t>0$. We conclude 
$\ (\zeta,\xi)\in  {\mathcal L}_{r}(x,\gamma y)$.

It remains  to prove that $\zeta\in \gamma B$ and $\xi\in A$.  By definition~(\ref{slcones})  $\gamma y\in  {\mathcal C}^-_{r}(x,A)$ immediately gives ${\mathcal O}^-_{r}(x,\gamma y)\subset {\mathcal O}^+_{r}(x,\gamma y)\subset A$, hence $\xi\in A$. 
Moreover, from $(\zeta,\xi)\in  {\mathcal L}_{r}(x,\gamma y)$ we directly get $\zeta\in {\mathcal O}^+_{r}(\gamma y,x)$. So  $\gamma^{-1} \zeta\in {\mathcal O}^+_{r}( y,\gamma^{-1} x)$, and from $\gamma^{-1}x\in {\mathcal C}^-_{r}(y,B)$ we know that   $
{\mathcal O}^+_{r}( y,\gamma^{-1} x)\subset B$  according to definition~(\ref{slcones}). Hence $\gamma^{-1}\zeta\in B$ which is equivalent to $\zeta\in\gamma B$.
\end{proof}

 We will further need the following Borel subsets of $\SX$ which up to small details were already introduced by T.~Roblin in \cite{MR2057305}: 
 \begin{eqnarray}\label{Kpm}
  && K_r(x) =  \{ g^s v(x;\xi,\eta) \colon (\xi,\eta)\in \ndpt\zero\ \text{with}\ d(x,(\xi\eta))<r,\ s\in (-r/2,r/2)\},\nonumber\\
&& K_r^+(x,A) = \{ v\in K_r(x)\colon v^+\in A\}=:K^+ ,\\
&& K_r^-(y,B)= \{ w\in K_r( y) \colon  w^-\in B\}=:K^-.\nonumber
 \end{eqnarray}
Notice that by definition the orbit of an element $v\in\zero $ either never enters one of the sets above or spends precisely time $r$ in them.
 
Moreover, we have the following relation to  the corridors ${\mathcal L}_{r}(x,\gamma y)$ introduced in~(\ref{Lrc}):
\begin{lemma}\label{K+-corridor} For all $\gamma\in\is(\XX)$ with $d(x,\gamma y)\ge 3r$  
 we have
\[ \ndpt\big(\{  K^+\cap g^{-t}\gamma  K^-\colon t>0 \}\big)= {\mathcal L}_{r}(x,\gamma y)\cap \ndpt\zero\cap (\gamma B\times A) \] 
\end{lemma}
\begin{proof} 
For the  inclusion ``$\subset$" we let $v\in K^+\cap g^{-t}\gamma K^-$ for some $t>0$. Then obviously $(\zeta,\xi):= (v^-,v^+)\in\ndpt\zero$,   $\xi=v^+\in A$ and $\zeta=v^-\in\gamma B$.  Now consider 
$u:= v(x;\zeta,\xi)\in\zero$ and let $\tau\in\RR$ \st 
\[ v(\gamma y, \zeta,\xi) = g^{\tau} u;\] such $\tau$ exists because $(\zeta,\xi)\in\ndpt\zero$. From the definition of $K_r(x)$ and $K_r(\gamma y)$
we further get $|d(x,\gamma y)-\tau|<2r$; since  $d(x,\gamma y)\ge 3r$ this implies $\tau>r>0$. 
Hence $(\zeta,\xi)=(u^-, u^+)\subset {\mathcal L}_{r}(x,\gamma y)$. 

For the converse inclusion ``$\supset$" we let $(\zeta,\xi)\in {\mathcal L}_{r}(x,\gamma y) \cap \ndpt\zero \cap (\gamma B\times A) $ be arbitrary. Then  by definition~(\ref{Lrc}) there exists $v\in \zero $ and $t'>0$ with 
\[(v^-,v^+)=(\zeta,\xi),\quad  d(x, v(0))<r\quad\text{and }\ d\bigl(\gamma y, v(t')\bigr)<r.\]
As above we set  
$u:= v(x;\zeta,\xi)$ and let $\tau\in\RR$ \st 
\[ v(\gamma y, \zeta,\xi) = g^{\tau} u.\] 
Since $d(x,u(0))\le d(x,v(0))<r$ and $u^+=\xi\in A$ we have $u\in K^+$, and 
from  $d(\gamma y, u(\tau))\le d(\gamma y,v(t'))<r$ and  $u^-=\zeta\in\gamma B$ we further get  $
 g^{\tau} u\in\gamma K^-$. Moreover we have $\tau>r>0$ as above, so the claim is proved. 
%
\end{proof}

\section{Ricks' Bowen-Margulis measure and some useful estimates}\label{RicksBMestimates}
As before  $\XX$ will denote a proper Hadamard space 
and $\Gamma<\is(\XX)$ a discrete rank one group with $\zero_\Gamma\ne\emptyset$. 
In order to get the equidistribution result Theorem~B from the introduction we will have to work with the so-called Ricks' Bowen-Margulis measure: This is the Ricks' measure from  Section~\ref{Mixing} associated to a particular  quasi-product geodesic current $\overline\mu$. We are going to describe this geodesic current now.  
\begin{definition}\label{confdensity}
A $\delta$-dimensional $\Gamma$-invariant  {\hd conformal density} is a continuous map $\mu\,$ of $\XX$ into the cone of positive finite Borel
measures on $\rand$ such that for all $x$, $y\in\XX$ and every $\gamma\in\Gamma$  we have
\begin{align}\label{conformality}
\nonumber &\supp(\mu_x)\subset \Lim,\\
\nonumber &\gamma_*\mu_x=\mu_{\gamma x},\quad\text{where } \ \gamma_*\mu_x(E):=\mu_x(\gamma^{-1}E)\quad\text{for all Borel sets }\ E\subset\rand,\\
  &\frac{\d \mu_x}{\d \mu_y}(\eta)=\e^{\delta \bs_{\eta}(y,x)} \quad\text{for any }\ \eta\in\supp(\mu_x).
\end{align}
\end{definition}
Recall the definition 
of the critical exponent $ \delta_\Gamma$ from (\ref{critexpdef}) and notice that in our setting it is strictly positive, since $\Gamma$ contains  a non-abelian free subgroup generated by two independent rank one elements. 
For $\delta=\delta_\Gamma$ a conformal density as above can be explicitly constructed following the idea of S.~J.~Patterson \mbox{(\cite{MR0450547})} originally used for Fuchsian groups (see for example \cite[Lemma 2.2]{MR1465601}). 
From here on we will therefore fix a $\delta_\Gamma$-dimensional $\Gamma$-invariant conformal density $\mu=(\mu_x)_{x\in\XX}$. 


With the Gromov product from (\ref{GromovProd}) we will now consider as in Section~7 of  \cite{Ricks} and in Section~8 of \cite{LinkHTS} for $x\in\XX$ the  geodesic current $\overline{\mu}_{x}\,$ on  
$\ndpt\SX \subset\rand\times\rand$ defined by 
\[ d\overline{\mu}_x(\xi,\eta)=\e^{2\delta_\Gamma\Gr_x(\xi,\eta)} {\mathbbm 1}_{\ndpt\mathcal{R}}(\xi,\eta)\d\mu_x(\xi)\d\mu_x(\eta).\]
%
As $\overline\mu_x$ does not depend on the choice of $x\in\XX$ we will write $\overline\mu$ in the sequel. 

Since we want to apply Theorem~\ref{mixthm} we will assume that $\mu_x(\radlim)=\mu_x(\rand)$; in view of 
Hopf-Tsuji-Sullivan dichotomy (Theorem 10.2 in \cite{LinkHTS}) this  is equivalent to the fact that  $\Gamma$ is divergent. Moreover, it  is 
well-known that in this case the conformal density $\mu$ from above is non-atomic and 
unique  up to scaling. 
So Theorem~\ref{currentisproduct} implies that for all $x$, $y\in\XX$ we have
 \begin{equation}\label{overlinemudef} 
d\overline{\mu}(\xi,\eta)=\e^{2\delta_\Gamma\Gr_x(\xi,\eta)}\d\mu_x(\xi)\d\mu_x(\eta)=\e^{2\delta_\Gamma\Gr_y(\xi,\eta)}\d\mu_y(\xi)\d\mu_y(\eta) 
\end{equation}
and \[(\mu_x\otimes\mu_x)(\ndpt\zero_\Gamma^{\small{\mathrm{rec}}})=(\mu_x\otimes\mu_x)(\ndpt \zero)=\mu_x(\rand)^2.\] 
The Ricks' measure $m_\Gamma$ on $\quotient{\Gamma}{\SX}$ associated to the geodesic current $\overline\mu$ from (\ref{overlinemudef}) is called the {\hl Ricks' Bowen-Margulis measure}. It generalizes the well-known Bowen-Margulis measure in the CAT$(-1)$-setting. Moreover, for the measure $m$ from which it descends we have the formula ~(\ref{measureformula}). Notice also that if $\XX$ is a manifold and $\Gamma$ is cocompact, then Ricks' Bowen-Margulis measure is equal to the measure of maximal entropy $m^{\text{\scriptsize Kn}}_\Gamma$ described in \cite{MR1652924} (this is Knieper's measure associated to $\overline\mu$ from (\ref{overlinemudef})). 
%
We further remark that the constant $\Delta$ defined in~(\ref{Deltadef}) is equal to $2 \delta_\Gamma$ in this case (compare the last paragraph in Section~8 of \cite{LinkHTS}), hence in particular finite. 

Fix $r>0$,  points $x$, $y\in\XX$ and subsets $A$, $B\subset\rand$. We will first compute the measure of the sets introduced in~(\ref{Kpm}). 
From (\ref{measureformula}),  
(\ref{overlinemudef}) and the remark below (\ref{Kpm}) we get
\begin{align*}
m(K^+)&=\int_{\ndpt\zero}\d\mu_x(\xi)\d\mu_x(\eta)\e^{2\delta_\Gamma \Gr_x(\xi,\eta)} \int {\mathbbm 1}_{K^+}\bigl(g^s v(x;\xi,\eta)\bigr)\d s\\
& =r \int_A \d \mu_x(\xi)\int_{\mathcal{O}_r(\xi,x)} \d \mu_x(\eta)\e^{2\delta_\Gamma \Gr_x(\xi,\eta)}, \nonumber
\end{align*}
and similarly
\[
m(K^-)=r \int_B \d \mu_y(\eta)\int_{\mathcal{O}_r(\eta,y)} \d \mu_y(\xi)\e^{2\delta_\Gamma \Gr_y(\xi,\eta)} .\]
From the non-negativity of the Gromov-product (\ref{GromovProd}) and the fact that
\[ \Gr_x(\xi,\eta)\le r\quad\text{if}\quad \eta\in\mathcal{O}_r(\xi,x)\]
we further get the useful estimates
\begin{align}\label{measureK+}
r \int_A \d \mu_x(\xi)\mu_x\bigl({\mathcal{O}_r(\xi,x)}\bigr)&\le m(K^+)\le r \e^{2\delta_\Gamma r} \int_A \d \mu_x(\xi)\mu_x\bigl({\mathcal{O}_r(\xi,x)}\bigr) , \\
\nonumber r \int_B \d \mu_y(\eta)\mu_y\bigl({\mathcal{O}_r(\eta,y)}\bigr)&\le m(K^-)\le r \e^{2\delta_\Gamma r} \int_B \d \mu_y(\eta)\mu_x\bigl({\mathcal{O}_r(\eta,y)}\bigr).
\end{align}

We continue with the important 
\begin{lemma}\label{flowintegrals}
Let $T_0>6r$, $T>T_0+3r$, $\gamma\in\Gamma$,  
$(\xi,\eta)\in {\mathcal L}_r(x,\gamma y)\cap \ndpt\zero$ and $s\in (-r/2,r/2)$. Then
\begin{enumerate}
\item[(a)] $\ \displaystyle \int_{T_0}^{T+3r}  \e^{\delta_\Gamma t} \mathbbm{1}_{K_r(\gamma y)}\bigl(g^{t+s} v(x;\xi,\eta)\bigr)\d t \ge r\cdot  \e^{-3\delta_\Gamma r} \e^{\delta_\Gamma d(x,\gamma y)}\ \  $\\ \hspace*{5cm} if  $\ \ T_0+3r<d(x,\gamma y)\le T$,
\item[(b)] 
$\ \displaystyle \int_{T_0}^{T-3r}  \e^{\delta_\Gamma t} \mathbbm{1}_{K_r(\gamma y)}\bigl(g^{t+s} v(x;\xi,\eta)\bigr)\d t \le r\cdot  \e^{3\delta_\Gamma r}  \e^{\delta_\Gamma d(x,\gamma y)} $,  \\
and $\ \ \displaystyle \int_{T_0}^{T-3r}  \e^{\delta_\Gamma t} \mathbbm{1}_{K_r(\gamma y)}\bigl(g^{t+s} v(x;\xi,\eta)\bigr)\d t =0\ \ $\\ \hspace*{5cm} if  $\ \ d(x,\gamma y)\le T_0-3r\ \ $ or $\ \ d(x,\gamma y)>T$.
\end{enumerate}
\end{lemma}
\begin{proof}
Denote $v= v(x;\xi,\eta)\in\zero$ and let $\tau>0$ \st $g^{\tau}v=v(\gamma y;\xi,\eta)$. Since $(\xi,\eta)\in {\mathcal L}_r(x,\gamma y)$, 
the triangle inequality yields
\[ |  d(x,\gamma y)-\tau|< 2r.\]
By definition of $K_r(\gamma y)$ we have $g^{t+s}v \in K_r(\gamma y)$ if and only if 
$|t+s-\tau|<r/2$. Hence if $\,\tau -s -r/2\ge T_0\, $ and $\,\tau-s+r/2\le T+ 3r$, then
\begin{align*}
 \int_{T_0}^{T+3r}  \e^{\delta_\Gamma t} \mathbbm{1}_{K_r(\gamma y)}\bigl(g^{t+s} v(x;\xi,\eta)\bigr)\d t &=\int_{\tau -s -r/2}^{\tau-s+r/2}  \e^{\delta_\Gamma t}\d t\\
 &\hspace*{-2cm}
 \ge  r\cdot \e^{\delta_\Gamma (\tau-s-r/2)}\ge r\cdot e^{-3\delta_\Gamma r}\e^{\delta_\Gamma d(x,\gamma y)} .
 \end{align*}
Now $d(x,\gamma y)\in (T_0+3r, T]$ and $s\in (-r/2,r/2)$ imply
\[ \,\tau -s -r/2\ge d(x,\gamma y)-2r  -r/2- r/2\ge T_0\quad\text{and} \]
\[ \tau -s +r/2\le d(x,\gamma y)+2r+r/2+r/2\le T+3r,\]
so (a) holds.

In order to prove (b) we first notice that 
\begin{align*}
 \int_{T_0}^{T- 3r}  \e^{\delta_\Gamma t} \mathbbm{1}_{K_r(\gamma y)}\bigl(g^{t+s} v(x;\xi,\eta)\bigr)\d t &\le \int_{\tau -s -r/2}^{\tau-s+r/2}  \e^{\delta_\Gamma t}\d t\\
& \hspace*{-2cm} \le  r\cdot \e^{\delta_\Gamma (\tau-s+r/2)} \le r\cdot e^{3\delta_\Gamma r}\e^{\delta_\Gamma d(x,\gamma y)};
 \end{align*}
this proves the first assertion in (b).

Now if  $\, d(x,\gamma y)\le T_0-3r$, then 
\[ \,\tau-s+r/2\le d(x,\gamma y)+2r+r\le T_0,\]
and if $\, d(x,\gamma y)\ge T$, then 
\[ \,\tau-s-r/2\ge d(x,\gamma y)-2r-r\ge T-3r,\]
hence the integral in (b) equals zero in both cases.
\end{proof}

Moreover,  from Lemma~\ref{K+-corridor} we 
immediately get the following
\begin{corollary}\label{K+-measure} For all $\gamma\in\is(\XX)$ with $d(x,\gamma y)>3r$ 
and all $\,t>0 $ we have 
\begin{align*}
 m \bigl(  K^+\cap g^{-t}\gamma  K^-\bigr)&=\int_{{\mathcal L}_{r}(x,\gamma y)\cap (\gamma B\times A )}\d\mu_x(\xi)\d\mu_x(\eta)\e^{2\delta_\Gamma \Gr_x(\xi,\eta)} \\
 &\hspace*{3cm} \cdot \int_{-r/2}^{r/2} {\mathbbm 1}_{K_r(\gamma y)} \bigl(g^{t+s} v(x;\xi,\eta)\bigr)\d s. \end{align*}
\end{corollary}

\section{Equidistribution}\label{equidistribution}

We keep the notation and the setting from the previous section and 
will now address the question of equidistribution of $\Gamma$-orbit points in $\XX$. In order to get a reasonable statement we will have to require that the Ricks' Bowen-Margulis measure $m_\Gamma$ is {\hl finite}:
\begin{theorem}\label{equithm}
Let $\Gamma<\is(\XX)$ be a discrete rank one group with  non-arithmetic length spectrum, 
$\zero_\Gamma\ne\emptyset$ 
and finite Ricks' Bowen-Margulis measure $m_\Gamma$.

Let $f$ be a continuous function from $\ganz\times \ganz$ to $\RR$, and  $x$, $y\in\XX$. Then  
\[ \lim_{T\to\infty}    
\delta_\Gamma  \e^{-\delta_\Gamma T} \sum_{\begin{smallmatrix}{\scriptstyle\gamma\in\Gamma}\\{\scriptstyle d(x,\gamma y )\le T}\end{smallmatrix}} f(\gamma y,\gamma^{-1} x)=\frac1{\Vert m_\Gamma \Vert} \int_{\rand\times\rand} f(\xi,\eta)\d \mu_x(\xi) \d \mu_y(\eta).\]
\end{theorem}

Our proof will closely follow Roblin's strategy for his Th{\'e}or{\`e}me~4.1.1 in \cite{MR2057305}:
Using mixing of the geodesic flow one proves that for all sufficiently small Borel sets $A,B\subset \rand$ the limit inferior and the limit superior of the measures 
\begin{equation}\label{measuredef}
\nu_{x,y}^T:=
\delta_\Gamma\cdot  \e^{-\delta_\Gamma T}\sum_{\begin{smallmatrix}{\scriptstyle\gamma\in\Gamma}\\{\scriptstyle d(x,\gamma y )\le T}\end{smallmatrix}} \mathcal{D}_{\gamma y}\otimes  \mathcal{D}_{\gamma^{-1} x}\end{equation}
as $T$ tends to infinity  evaluated on products of ``cones" of opening $A$, $B$  is approximately $\mu_x(A)\cdot \mu_y(B)/\Vert m_\Gamma\Vert$.


In the first step we only deal with sufficiently small open neighborhoods 
of pairs of boundary points which are in a ``nice" position with respect to $x$ and $y$; then one shows that the estimates hold for all pairs of sufficiently small Borel sets $A$ and $B$.  The final step consists in the full proof by globalisation with respect to $A$ and $B$.

The following Proposition provides the first step in the proof of Theorem~\ref{equithm}:
\begin{proposition}\label{firststep}
Let $\varepsilon>0$, $(\xi_0,\eta_0)\in\rand\times\rand$ and $x$, $y\in\XX$ with trivial stabilizer in $\Gamma$ and \st $x\in (\xi_0 v^+)$, $y\in (\eta_0 w^+)$ for some 
$\Gamma$-recurrent elements $v$, $w\in\zero$. 
Then there exist open neighborhoods $V$, $W\subset\rand$ of $\xi_0$,  $\eta_0$ \st for all Borel sets $A\subset V$, $B\subset W$
\begin{align*}
\limsup_{T\to\infty} \nu_{x,y}^T\bigl({\mathcal C}^-_{1}(x,A)\times {\mathcal C}^-_{1}(y,B)\bigr) &\le \e^\varepsilon \mu_x(A)\mu_y (B)/\Vert m_\Gamma\Vert,\\
\liminf_{T\to\infty} \nu_{x,y}^T\bigl({\mathcal C}^+_{1}(x,A)\times {\mathcal C}^+_{1}(y,B)\bigr) &\ge \e^{-\varepsilon} \mu_x(A)\mu_y (B)/\Vert m_\Gamma\Vert.
\end{align*}
\end{proposition}
\begin{proof}
Set $\rho:=\min\{ d(x,\gamma x),\ d(y,\gamma y) 
\colon \gamma\in \Gamma\setminus\{e\}\}$. 

Let $\varepsilon>0$ arbitrary. We first fix $r\in (0,\min\{1,\rho/3, \varepsilon/(30 \delta_\Gamma)\})$ \st
\[ \mu_x\bigl(\widetilde\partial\mathcal{O}_r(\xi_0, x)\bigr)=0=\mu_y\bigl(\widetilde\partial\mathcal{O}_r(\eta_0, y)\bigr).\]
Since $v^+\in \Lim\cap \mathcal{O}_r(\xi_0, x)$ and $w^+\in\Lim\cap \mathcal{O}_r(\eta_0, y)$, both shadows  $\mathcal{O}_r(\xi_0, x)$ and $\mathcal{O}_r(\eta_0, y)$ contain an open neighborhood of a limit point of $\Gamma$ by Lemma~\ref{joinrankone}. So from $\supp(\mu_x)=\supp(\mu_y)=\Lim$ and the definition~(\ref{supportdef}) of the support of a measure we have 
\[ \mu_x\bigl(\mathcal{O}_r(\xi_0, x)\bigr)\cdot\mu_y\bigl(\mathcal{O}_r(\eta_0, y)\bigr)>0.\]
Moreover, according to Lemma~\ref{joinrankone} and Corollary~\ref{measureofshadowsisclose} 
there exist open neighborhoods $\widehat V$, $\widehat W\subset\ganz$ of $\xi_0$, $\eta_0$ \st if $(a, b)\in \widehat V\times \widehat W$, then $a$ can be joined to $v^+$,  $b$ can be joined to $w^+$ by a rank one geodesic, 
and
\begin{align}\label{approxmeasures}
\e^{-\varepsilon/30}  \mu_x\bigl(\mathcal{O}_r(\xi_0, x) \bigr)&\le \mu_x\bigl(\mathcal{O}^\pm_{r}(a, x) \bigr)\le\e^{\varepsilon/30} \mu_x\bigl(\mathcal{O}_r(\xi_0, x) \bigr),\nonumber \\
\e^{-\varepsilon/30}\mu_y\bigl(\mathcal{O}_r(\eta_0, y) \bigr) & \le  \mu_y\bigl(\mathcal{O}^\pm_{r}(b, y) \bigr) \le\e^{\varepsilon/30}\mu_y\bigl(\mathcal{O}_r(\eta_0, y) \bigr).
\end{align}

Let $V$, $W\subset\rand$ be open neighborhoods of $\xi_0$, $\eta_0$ \st $\overline{V}  \subset \widehat V\cap\rand$ and $\overline{W}\subset \widehat W\cap\rand$.  Let $A\subset V$, $B\subset W$ be arbitrary Borel sets.  

Roblin's method consists in giving  upper and lower bounds for the asymptotics  of  the integrals
\[ \int_{T_0}^{T\pm 3r} \e^{\delta_\Gamma t} \sum_{\gamma\in\Gamma} m \bigl(  K^+\cap g^{-t}\gamma  K^-\bigr) \d t\]
as $T$ tends to infinity:
On the one hand one uses mixing to relate the integrals to $\mu_x(A)\cdot \mu_y(B)$; on the other hand one computes direct estimates for the integrals to get a relation to the measures $\nu_{x,y}^T\bigl({\mathcal C}^\pm_{1}(x,A)\times {\mathcal C}^\pm_{1}(y,B)\bigr)$.

Let us start by exploiting the mixing property. Notice that by choice of $r<\rho/3$ and 
the definition of $\rho$ we have 
\[ K_r(x) \cap \gamma K_r(x) =\emptyset\ \text{ and}\quad K_r(y) \cap \gamma K_r(y)=\emptyset\ \text{ for all}\quad \gamma\in\Gamma\setminus\{e\},\]
hence the projection map $\SX\to \quotient{\Gamma}{\SX}$ restricted to  $K^\pm$ is injective. 
So we can apply  Corollary~\ref{mixcor} 
to get
\[ \lim_{t\to \infty} \sum_{\gamma\in\Gamma} m(K^+\cap g^{-t}\gamma K^-)= 
\frac{m(K^+)\cdot m(K^-)}{\Vert m_\Gamma\Vert }.\]
Hence there exists $T_0> 6r$ \st for $t\ge T_0 $ we have
\begin{eqnarray}\label{summixing}
 \e^{-\varepsilon/3} m(K^+)\cdot m(K^-) &\le & \Vert m_\Gamma\Vert\cdot \sum_{\gamma\in\Gamma} m(K^+\cap g^{-t} \gamma K^-)\nonumber\\
 &\le & \e^{\varepsilon/3} m(K^+)\cdot m(K^-).\end{eqnarray}
Combining (\ref{measureK+}) and the estimates (\ref{approxmeasures}) we obtain from $A\subset \widehat V$ and $B\subset \widehat W$
\begin{align*} r \e^{-\varepsilon/30} \mu_x\bigl({\mathcal O}_r(\xi_0,x)\bigr)\mu_x(A) \le m(K^+)&\le r\e^{2\delta_\Gamma r}\e^{\varepsilon/30} \mu_x\bigl({\mathcal O}_r(\xi_0,x)\bigr)\mu_x(A),\\
r \e^{-\varepsilon/30} \mu_y\bigl({\mathcal O}_r(\eta_0,y)\bigr)\mu_y(B) \le m(K^-)&\le r\e^{2\delta_\Gamma r}\e^{\varepsilon/30} \mu_y\bigl({\mathcal O}_r(\eta_0,y)\bigr)\mu_y(B);
\end{align*} 
using the abbreviation $M=r^2 \mu_x\bigl({\mathcal O}_r(\xi_0,x)\bigr) \mu_y\bigl({\mathcal O}_r(\eta_0,y)\bigr)>0$ and $\delta_\Gamma r\le \varepsilon/30$  we get
\begin{equation}\label{measureproductKpKm}
\e^{-\varepsilon/15}M\mu_x(A)\mu_y(B)\le m(K^+)m(K^-)\le \e^{\varepsilon/5}M\mu_x(A)\mu_y(B).
\end{equation}
Hence according to (\ref{summixing}) we have for $t\ge T_0$ 
\begin{align*} 
M \mu_x(A)\mu_y(B)&\le \e^{\varepsilon/15}\e^{\varepsilon/3} \Vert m_\Gamma\Vert\cdot \sum_{\gamma\in\Gamma} m(K^+\cap g^{-t}\gamma K^-),\\
M \mu_x(A)\mu_y(B)&\ge \e^{-\varepsilon/5}\e^{-\varepsilon/3} \Vert m_\Gamma\Vert\cdot \sum_{\gamma\in\Gamma} m(K^+\cap g^{-t}\gamma K^-).
\end{align*}
We now integrate the inequalities to get
\begin{align} 
\nonumber \bigl(\e^{\delta_\Gamma (T-3r)}-\e^{\delta_\Gamma T_0}\bigr) &M \mu_x(A)\mu_y(B) = \delta_\Gamma \int_{T_0}^{T-3r} \e^{\delta_\Gamma t}  M \mu_x(A)\mu_y(B)\d t\\
\label{Tminus3r} &\le e^{2\varepsilon/5} \Vert m_\Gamma\Vert\cdot \delta_\Gamma \int_{T_0}^{T-3r} \e^{\delta_\Gamma t}\sum_{\gamma\in\Gamma} m(K^+\cap g^{-t}\gamma K^-),\\
\nonumber \bigl(\e^{\delta_\Gamma (T+3r)}-\e^{\delta_\Gamma T_0}\bigr) & M \mu_x(A)\mu_y(B) = \delta_\Gamma \int_{T_0}^{T+3r} \e^{\delta_\Gamma t}  M \mu_x(A)\mu_y(B)\d t \\
\label{Tplus3r} &\ge e^{-8\varepsilon/15} \Vert m_\Gamma\Vert\cdot \delta_\Gamma \int_{T_0}^{T+3r} \e^{\delta_\Gamma t}\sum_{\gamma\in\Gamma} m(K^+\cap g^{-t}\gamma K^-).\end{align}

We will next give upper and lower bounds for the integrals on the right-hand side:
For the upper bound we first remark that   $(\xi,\eta)\in  {\mathcal L}_{r}(x,\gamma y)\cap\ndpt\zero$ implies 
 $\Gr_x(\xi,\eta)<r$. Moreover,   our choice of $T_0 > 6r$ guarantees that 
 $K^+ \cap g^{-t} \gamma K^-\ne\emptyset$ for some $t \geq T_0$ implies  $d(x, \gamma y) > 3r$.  Applying Corollary \ref{K+-measure} we therefore get
\begin{align*}   
& \hspace*{-0.7cm} \int_{T_0}^{T- 3r} \e^{\delta_\Gamma t} \sum_{\gamma\in\Gamma} m \bigl(  K^+\cap g^{-t}\gamma  K^-\bigr) \d t\\
 &\le \sum_{\gamma\in\Gamma} \int_{{\mathcal L}_{r}(x,\gamma y)\cap (\gamma B\times A)}\d\mu_x(\xi)\d\mu_x(\eta)\e^{2\delta_\Gamma r}\\ 
&\hspace*{3cm} \cdot  \int_{-r/2}^{r/2} \bigl(\int_{T_0}^{T-3r} \mathbbm{1}_{K_r(\gamma y)} \bigl(g^{t+s} v(x;\xi,\eta)\bigr) \e^{\delta_\Gamma t} \d t\bigr) \d s \\
&\le \e^{2\delta_\Gamma r}\cdot r^2\cdot  \e^{3\delta_\Gamma r} \sum_{\begin{smallmatrix} {\scriptstyle \gamma\in\Gamma}\\ {\scriptstyle 
d(x,\gamma y)\le T}\end{smallmatrix}}\int_{{\mathcal L}_{r}(x,\gamma y)\cap (\gamma B\times A)}\d\mu_x(\xi)\d\mu_x(\eta)\cdot  \e^{\delta_\Gamma d(x,\gamma y)};
\end{align*}
here we used Lemma~\ref{flowintegrals} (b)  in the last step.
Lemma~\ref{largecones}, $r\le 1$ and the first estimate in (\ref{coneesti}) further imply
\begin{align*}   
& \hspace*{-0.1cm} \int_{T_0}^{T- 3r} \e^{\delta_\Gamma t} \sum_{\gamma\in\Gamma} m \bigl(  K^+\cap g^{-t}\gamma  K^-\bigr) \d t\\
&\le r^2 \e^{5\delta_\Gamma r}\hspace*{-4mm} \sum_{\begin{smallmatrix} {\scriptstyle \gamma\in\Gamma}\\ {\scriptstyle 
d(x,\gamma y)\le T}\\ {\scriptstyle (\gamma y,\gamma^{-1}x)\in  {\mathcal C}^+_{1}(x,A)\times {\mathcal C}^+_{1}(y,B)}\end{smallmatrix}} \hspace*{-4mm}  \int_{{\mathcal O}^+_{r}(\gamma y,x)} \d\mu_x(\xi) \int_{{\mathcal O}^+_{r}(x,\gamma y)}\d\mu_x(\eta) \e^{\delta_\Gamma d(x,\gamma y)}.
\end{align*}
Using the fact that for all $\eta\in{\mathcal O}^+_{r}(x,\gamma y)$ we have $\ \bs_\eta(x,
\gamma y)\ge d(x,\gamma y)-4r$,\break $\Gamma$-equivariance and conformality~(\ref{conformality}) of $\mu$ imply
\[ \int_{{\mathcal O}^+_{r}(x,\gamma y)}\d\mu_x(\eta) \e^{\delta_\Gamma d(x,\gamma y)}\le \e^{4\delta_\Gamma r}  \mu_y\bigl({\mathcal O}^+_{r}(\gamma^{-1}x, y)\bigr).\] 
Moreover, since by Lemma~\ref{orbitpointsincones} (a) there are only finitely many $\gamma\in\Gamma$ \st 
\[ (\gamma y,\gamma^{-1}x)\in \left( {\mathcal C}^+_{1}(x,A)\times {\mathcal C}^+_{1}(y,B)\right) \setminus (\widehat V\times \widehat W),\]
restricting the summation to $\gamma\in\Gamma$ with 
\[ (\gamma y,\gamma^{-1}x)\in \left( {\mathcal C}^+_{1}(x,A)\times {\mathcal C}^+_{1}(y,B)\right) \cap (\widehat V\times \widehat W)\]  only contributes a constant $C$ to the upper bound. 
So with our choice of $r\le 1$ and $r\le \varepsilon/(30 \delta_\Gamma)$  we conclude
\begin{align*}
& \hspace*{-7mm} \int_{T_0}^{T- 3r}  \e^{\delta_\Gamma t} \sum_{\gamma\in\Gamma} m \bigl(  K^+\cap g^{-t}\gamma  K^-\bigr) \d t\\
&\le  r^2 \e^{9\varepsilon/30} \hspace*{-14mm} \sum_{\begin{smallmatrix} {\scriptstyle \gamma\in\Gamma}\\ {\scriptstyle d(x,\gamma y)\le T}\\ {\scriptstyle (\gamma y,\gamma^{-1}x)\in ( {\mathcal C}^+_{1}(x,A)\times {\mathcal C}^+_{1}(y,B))\cap (\widehat V\times \widehat W)}\end{smallmatrix}} \hspace*{-14mm}  \mu_x\bigl({\mathcal O}^+_{r}(\gamma y,x)\bigr) \mu_y\bigl({\mathcal O}^+_{r}(\gamma^{-1}x, y)\bigr) +C\\
&\stackrel{(\ref{approxmeasures})}{\le} r^2  \e^{11\varepsilon/30}\hspace*{-14mm} \sum_{\begin{smallmatrix} {\scriptstyle \gamma\in\Gamma}\\ {\scriptstyle d(x,\gamma y)\le T}\\ {\scriptstyle (\gamma y,\gamma^{-1}x)\in ( {\mathcal C}^+_{1}(x,A)\times {\mathcal C}^+_{1}(y,B))\cap (\widehat V\times \widehat W)}\end{smallmatrix}} \hspace*{-14mm}  \mu_x\bigl({\mathcal O}_{r}(\xi_0,x)\bigr) \mu_y\bigl({\mathcal O}_{r}(\eta_0, y)\bigr) +C,\\
&\le \e^{11\varepsilon/30} M \frac{\e^{\delta_\Gamma T}}{\delta_\Gamma}\nu_{x,y}^T\bigl( {\mathcal C}^+_{1}(x,A)\times {\mathcal C}^+_{1}(y,B)\bigr)+ C.
\end{align*}
Plugging this in the inequality (\ref{Tminus3r}) divided by $M \e^{\delta_\Gamma (T-3r)}\cdot\Vert m_\Gamma\Vert$ we get (with a constant $C'$ independent of $T$)
\begin{align*}
\frac{1- \e^{\delta_\Gamma (-T+3r+T_0)}}{\Vert m_\Gamma\Vert} \mu_x(A)\mu_y(B) & \le 
 \e^{2\varepsilon/5} \e^{11\varepsilon/30}\e^{3\delta_\Gamma r}\nu_{x,y}^T\bigl( {\mathcal C}^+_{1}(x,A)\times {\mathcal C}^+_{1}(y,B)\bigr)\\
 +\  C'\e^{-\delta_\Gamma T}
 &\le  \e^{13\varepsilon/15} \nu_{x,y}^T\bigl( {\mathcal C}^+_{1}(x,A)\times {\mathcal C}^+_{1}(y,B)\bigr)+ C'\e^{-\delta_\Gamma T} ,
\end{align*}
which proves 
\[ \liminf_{T\to\infty} \nu_{x,y}^T\bigl( {\mathcal C}^+_{1}(x,A)\times {\mathcal C}^+_{1}(y,B)\bigr)\ge \e^{-\varepsilon} \mu_x(A)\mu_y(B)/\Vert m_\Gamma\Vert.\]

We finally turn to the lower bound. 
Using again Corollary \ref{K+-measure}  and the non-negativity of the Gromov product (\ref{GromovProd}) 
we estimate
\begin{align*}   
& \hspace*{-0.7cm} \int_{T_0}^{T+3r} \e^{\delta_\Gamma t} \sum_{\gamma\in\Gamma} m \bigl(  K^+\cap g^{-t}\gamma  K^-\bigr) \d t\\
 &\ge \sum_{\gamma\in\Gamma} \int_{{\mathcal L}_{r}(x,\gamma y)\cap (\gamma B\times A)}\d\mu_x(\xi)\d\mu_x(\eta)\e^{2\delta_\Gamma\cdot 0}\\ 
&\hspace*{3cm} \cdot  \int_{-r/2}^{r/2} \Bigl(\int_{T_0}^{T+3r} \mathbbm{1}_{K_r(\gamma y)} \bigl(g^{t+s} v(x;\xi,\eta)\bigr) \e^{\delta_\Gamma t} \d t\Bigr) \d s \\
&\ge r^2\e^{-3\delta_\Gamma r}  \sum_{\begin{smallmatrix} {\scriptstyle \gamma\in\Gamma}\\ {\scriptstyle T_0+3r< d(x,\gamma y)\le T}\end{smallmatrix}}\int_{{\mathcal L}_{r}(x,\gamma y)\cap (\gamma B\times A)}\d\mu_x(\xi)\d\mu_x(\eta)\cdot  \e^{\delta_\Gamma d(x,\gamma y)},
\end{align*}
where we used Lemma~\ref{flowintegrals} (a) in the last step.

By Lemma~\ref{smallcones}, $r\le 1$  and the second estimate in (\ref{coneesti})
we have for all $\gamma\in\Gamma$ with 
$(\gamma y, \gamma^{-1}x)  \in {\mathcal C}^-_{1}(x,A)\times {\mathcal C}^-_{1}(y,B)\subset {\mathcal C}^-_{r}(x,A)\times {\mathcal C}^-_{r}(y,B)$  
\begin{align*}  {\mathcal L}_{r}(x,\gamma y) \cap (\gamma B\times A)\supset \{(\zeta,\xi)\in\rand\times\rand\colon \xi\in 
 {\mathcal O}^-_{r}(x,\gamma y),\ \zeta\in {\mathcal O}_{r}(\xi,x)\},
\end{align*}
hence 
\begin{align*}   
& \hspace*{-0.7cm} \int_{T_0}^{T+3r} \e^{\delta_\Gamma t} \sum_{\gamma\in\Gamma} m \bigl(  K^+\cap g^{-t}\gamma  K^-\bigr) \d t\\
&\ge  r^2 \cdot \e^{-\varepsilon/10}  \hspace*{-14mm} \sum_{\begin{smallmatrix} {\scriptstyle \gamma\in\Gamma}\\ {\scriptstyle T_0+3r<d(x,\gamma y)\le T}\\ {\scriptstyle(\gamma y, \gamma^{-1}x)  \in ({\mathcal C}^-_{1}(x,A)\times {\mathcal C}^-_{1}(y,B))\cap (\widehat V\times \widehat W)}\end{smallmatrix}} \hspace*{-14mm} \int_{{\mathcal O}^-_{r}(x, \gamma y)}\d\mu_x(\xi) \e^{\delta_\Gamma d(x,\gamma y)}\cdot \mu_x\bigl({\mathcal O}_{r}(\xi,x)\bigr).
\end{align*}
Notice that $\gamma y\in {\mathcal C}^-_{1}(x,A)\subset {\mathcal C}^-_{r}(x,A)$ implies ${\mathcal O}^-_{r}(x, \gamma y)\subset {\mathcal O}^+_{r}(x, \gamma y)\subset A\subset\widehat V$ by definition of the small cones. Hence  (\ref{approxmeasures}) shows that for all $\xi\in {\mathcal O}^-_{r}(x, \gamma y)$ we have
%
\begin{align*} \mu_x\bigl({\mathcal O}_{r}(\xi,x)\bigr) &\ge  \e^{-\varepsilon/30}\mu_x\bigl({\mathcal O}_{r}(\xi_0,x)\bigr).
\end{align*}
By $\Gamma$-equivariance and  conformality of $\mu$ we further have
\[ \int_{{\mathcal O}^-_{r}(x, \gamma y)}\d\mu_x(\xi) \e^{\delta_\Gamma d(x,\gamma y)} \ge \mu_y\bigl( {\mathcal O}^-_{r}(\gamma^{-1} x, y)\bigr)\ge  \e^{-\varepsilon/30}\mu_y\bigl( {\mathcal O}_{r}(\eta_0, y)\bigr),\]
where the last inequality follows from $\gamma^{-1}x\in \widehat W$ and 
(\ref{approxmeasures}). Altogether this proves
\begin{align*}
& \hspace*{-0.7cm} \int_{T_0}^{T+3r} \e^{\delta_\Gamma t} \sum_{\gamma\in\Gamma} m \bigl(  K^+\cap g^{-t}\gamma  K^-\bigr) \d t\nonumber\\
&\ge r^2\cdot \e^{-\varepsilon/6}  \hspace*{-10mm} \sum_{\begin{smallmatrix} {\scriptstyle \gamma\in\Gamma}\\ {\scriptstyle T_0+3r<d(x,\gamma y)\le T}\\ {\scriptstyle(\gamma y, \gamma^{-1}x)  \in {\mathcal C}^-_{1}(x,A)\times {\mathcal C}^-_{1}(y,B)}\cap(\widehat V\times \widehat W)\end{smallmatrix}}\mu_x\bigl({\mathcal O}_{r}(\xi_0,x)\bigr) \mu_y\bigl( {\mathcal O}_{r}(\eta_0, y)\bigr).
\end{align*}
Since the  number of elements $\gamma\in\Gamma$ with $d(x,\gamma y)\le T_0+3r$ or with 
\[ (\gamma y,\gamma^{-1}x)\in \left( {\mathcal C}^-_{1}(x,A)\times {\mathcal C}^-_{1}(y,B)\right) \setminus (\widehat V\times \widehat W)\]
is finite thanks to Lemma~\ref{orbitpointsincones} (a), there exists a constant $C>0$ \st 
\begin{align}\label{lowerboundused}
& \hspace*{-0.7cm} \int_{T_0}^{T+3r} \e^{\delta_\Gamma t} \sum_{\gamma\in\Gamma} m \bigl(  K^+\cap g^{-t}\gamma  K^-\bigr) \d t\nonumber\\
&\ge    r^2 \cdot \e^{-\varepsilon/6}  \hspace*{-10mm} \sum_{\begin{smallmatrix} {\scriptstyle \gamma\in\Gamma}\\ {\scriptstyle d(x,\gamma y)\le T}\\ {\scriptstyle(\gamma y, \gamma^{-1}x)  \in {\mathcal C}^-_{1}(x,A)\times {\mathcal C}^-_{1}(y,B)}\end{smallmatrix}}\mu_x\bigl({\mathcal O}_{r}(\xi_0,x)\bigr) \mu_y\bigl( {\mathcal O}_{r}(\eta_0, y)\bigr)  -C\\
&\ge  \e^{-\varepsilon/6} M \frac{\e^{\delta_\Gamma T}}{ \delta_\Gamma}\nu_{x,y}^T\bigl( {\mathcal C}^-_{1}(x,A)\times {\mathcal C}^-_{1}(y,B)\bigr)- C.\nonumber
\end{align}
 Plugging this in the inequality (\ref{Tplus3r}) divided by $M \e^{\delta_\Gamma (T+3r)}\cdot \Vert m_\Gamma\Vert$ we get (with a constant $C'$ independent of $T$)
\begin{align*}
\frac{1- \e^{\delta_\Gamma (-T-3r+T_0)}}{\Vert m_\Gamma\Vert} \mu_x(A)\mu_y(B) & \ge 
 \e^{-8\varepsilon/15} \e^{-\varepsilon/6}\e^{-3\delta_\Gamma r}\nu_{x,y}^T\bigl( {\mathcal C}^-_{1}(x,A)\times {\mathcal C}^-_{1}(y,B)\bigr)\\
 -\  C'\e^{-\delta_\Gamma T}
 &=  \e^{-12\varepsilon/15} \nu_{x,y}^T\bigl( {\mathcal C}^-_{1}(x,A)\times {\mathcal C}^-_{1}(y,B)\bigr)+ C'\e^{-\delta_\Gamma T} ,
\end{align*}
which proves 
\[ \limsup_{T\to\infty} \nu_{x,y}^T\bigl( {\mathcal C}^-_{1}(x,A)\times {\mathcal C}^-_{1}(y,B)\bigr)\le \e^{\varepsilon} \mu_x(A)\mu_y(B)/\Vert m_\Gamma\Vert.\]
\end{proof}

The next Proposition is the second step in the proof of Theorem~\ref{equithm}:
\begin{proposition}\label{secondstep}
Let $\varepsilon>0$ and $x$, $y\in\XX$ arbitrary. Then for all $(\xi_0,\eta_0)\in\rand\times\rand$   there exists $r>0$ and  open neighborhoods $V\subset\rand$ of $\xi_0$,  $W\subset\rand$ of $\eta_0$ \st for all Borel sets $A\subset V$, $B\subset W$
\begin{align*}
\limsup_{T\to\infty} \nu_{x,y}^T\bigl({\mathcal C}^-_{r}(x,A)\times {\mathcal C}^-_{r}(y,B)\bigr) &\le \e^\varepsilon \mu_x(A)\mu_y (B)/\Vert m_\Gamma\Vert,\\
\liminf_{T\to\infty} \nu_{x,y}^T\bigl({\mathcal C}^+_{r}(x,A)\times {\mathcal C}^+_{r}(y,B)\bigr) &\ge \e^{-\varepsilon} \mu_x(A)\mu_y (B)/\Vert m_\Gamma\Vert.
\end{align*}
\end{proposition}
\begin{proof}
Let $(\xi_0,\eta_0)\in\rand\times\rand$ be arbitrary. Choose $\Gamma$-recurrent geodesics $v$, $w\in\zero$  and $x_0\in (\xi_0 v^+)$, $y_0\in (\eta_0 w^+)$ with trivial stabilizers in $\Gamma$. Let $V_0$, $W_0\subset\rand$ be open neighborhoods of $\xi_0$ and $\eta_0$ \st the statement of Proposition~\ref{firststep} is true for $x_0$, $y_0$ instead of $x$, $y$, $V_0$, $W_0$ instead of $V$, $W$  and $\varepsilon/3$ instead of $\varepsilon$. 

Choose open neighborhoods $\widehat V_0$, $\widehat W_0$ of $\xi_0$, $\eta_0$ \st $\widehat V_0\cap\rand\subset V_0$, $\widehat W_0\cap\rand\subset W_0$   and
\begin{equation}\label{estimatebusinnbhd}
|d(x_0,a)-d(x,a)-\bs_{\xi_0}(x_0,x)|<\frac{\varepsilon}{6\delta_\Gamma},\quad |d(y_0,b)-d(y,b)-\bs_{\eta_0}(y_0,y)|<\frac{\varepsilon}{6\delta_\Gamma}\end{equation}
for all $(a,b)\in \widehat V_0\times \widehat W_0$. Notice that if $a=\xi
\in\rand$ we use the convention that $d(x_0,a)-d(x,a)=\bs_{a}(x_0,x)$ and similar for $b=\eta\in\rand$.

Now let $V$, $W\subset\rand$ be neighborhoods of  $\xi_0$, $\eta_0$ \st for the closures we have
$\overline V\subset \widehat V_0\cap \rand$ and $\overline W\subset \widehat W_0\cap \rand$. We further set 
\[r=1+\max\{ d(x,x_0), d(y,y_0)\},\]
and let $A\subset V$, $B\subset W$ be arbitrary Borel sets. From the choice of $r$ above and Lemma~\ref{changepoint} (b) 
we immediately deduce 
%
that  $(\gamma y,\gamma^{-1}x)\in {\mathcal C}^-_{r}(x,A)\times {\mathcal C}^-_{r}(y,B)$ implies 
\[ (\gamma y_0,\gamma^{-1}x_0)\in {\mathcal C}^-_{1}(x_0,A)\times {\mathcal C}^-_{1}(y_0,B).\]

For $r>0$ we set
\[\widehat V_{-r}:=\{z\in X\colon \overline{B_r(z)}\subset \widehat V_0\}\cup \bigl(\widehat V_0 \cap \rand\bigr).\] 
If $d(x,\gamma y)\le T$ and $(\gamma y,\gamma^{-1}x)\in \widehat V_{-r}\times \widehat W_{0}$,  then $(\gamma y_0,\gamma^{-1}x)\in \widehat V_{0}\times \widehat W_{0}$ and hence
\begin{align*}
d(x_0,\gamma y_0)&\le d(x,\gamma y_0)+\bs_{\xi_0}(x_0,x)+\frac{\varepsilon}{6\delta_\Gamma}=d(y_0,\gamma^{-1}x)+\bs_{\xi_0}(x_0,x)+\frac{\varepsilon}{6\delta_\Gamma}\\
&\le d(y,\gamma^{-1}x)+\bs_{\eta_0}(y_0,y)+ \bs_{\xi_0}(x_0,x)+\frac{\varepsilon}{3\delta_\Gamma}\\
&\le T+ \bs_{\eta_0}(y_0,y)+ \bs_{\xi_0}(x_0,x)+\frac{\varepsilon}{3\delta_\Gamma}.
\end{align*}
So we conclude that for $T\gg 1$ 
\begin{align*}
 \e^{-\delta_\Gamma T}& \#\{\gamma\in\Gamma\colon d(x,\gamma y)\le T,\ (\gamma y,\gamma^{-1} x)  \in \bigl({\mathcal C}^-_{r}(x,A)\times  {\mathcal C}^-_{r}(y,B)\bigr)\cap ( \widehat V_{-r}\times \widehat W_0)\}\\
& \le \e^{\varepsilon/3}\cdot \e^{\delta_\Gamma\bigl( \bs_{\eta_0}(y_0,y)+ \bs_{\xi_0}(x_0,x)\bigr)}\cdot \e^{-\delta_\Gamma(T+ \bs_{\eta_0}(y_0,y)+ \bs_{\xi_0}(x_0,x)+\varepsilon/3\delta_\Gamma)}\cdot\\
&\hspace*{15mm}\#\{\gamma\in\Gamma\colon d(x_0,\gamma y_0)\le T+ \bs_{\eta_0}(y_0,y)+ \bs_{\xi_0}(x_0,x)+\varepsilon/3\delta_\Gamma,\\
&\hspace*{29mm} \ (\gamma y_0,\gamma^{-1} x_0)  \in \bigl({\mathcal C}^-_{1}(x_0,A)\times  {\mathcal C}^-_{1}(y_0,B)\bigr)\cap ( \widehat V_{0}\times \widehat W_0)\}.
\end{align*}
Since the number of elements $\gamma\in\Gamma$ with $(\gamma y,\gamma^{-1}x)\in \bigl({\mathcal C}^-_{r}(x,A)\times {\mathcal C}^-_{r}(y,B)\bigr)\setminus (\widehat V_{-r}\times \widehat W_{0})$ is finite by Lemma~\ref{orbitpointsincones} (a), we conclude that
\begin{align*}
& \hspace*{-1cm} \limsup_{T\to\infty} \nu_{x,y}^T\bigl({\mathcal C}^-_{r}(x,A)\times {\mathcal C}^-_{r}(y,B)\bigr)\\
&\le \e^{\varepsilon/3}\e^{\delta_\Gamma \bigl(\bs_{\xi_0}(x_0,x)+\bs_{\eta_0}(y_0,y)\bigr)} \\
&\qquad \cdot \limsup_{T\to\infty} \nu_{x_0,y_0}^{T+\bs_{\xi_0}(x_0,x)+\bs_{\eta_0}(y_0,y)+\varepsilon/3\delta_\Gamma}\bigl({\mathcal C}^-_{1}(x_0,A)\times {\mathcal C}^-_{1}(y_0,B)\bigr)\\
&\le  \e^{2\varepsilon/3}\e^{\delta_\Gamma \bigl(\bs_{\xi_0}(x_0,x)+\bs_{\eta_0}(y_0,y)\bigr)} \mu_{x_0}(A)\mu_{y_0}(B)/\Vert m_\Gamma\Vert,
\end{align*}
where we used Proposition~\ref{firststep} in the last estimate. 

Now for $\xi\in A\subset \widehat V_0\cap\rand$ and $\eta \in B\subset \widehat W_0\cap \rand$ we get from (\ref{estimatebusinnbhd})
\[\bs_{\xi_0}(x_0,x) < \bs_{\xi}(x_0,x) +\frac{\varepsilon}{6\delta_\Gamma},\quad \bs_{\eta_0}(y_0,y) < \bs_{\eta}(y_0,y) +\frac{\varepsilon}{6\delta_\Gamma},\]
hence 
\begin{align*} \e^{\delta_\Gamma  \bs_{\xi_0}(x_0,x)}\mu_{x_0}(A) &=\int_A \e^{\delta_\Gamma  \bs_{\xi_0}(x_0,x)}\d \mu_{x_0}(\xi)\\
&\le \e^{\varepsilon/6} \int_A \e^{\delta_\Gamma  \bs_{\xi}(x_0,x)}\frac{\d \mu_{x_0}}{\d\mu_x}(\xi)\d\mu_x(\xi)\stackrel{(\ref{conformality})}{=} \e^{\varepsilon/6}\mu_x(A),
\end{align*}
and similarly
\[ \e^{\delta_\Gamma  \bs_{\eta_0}(y_0,y)}\mu_{y_0}(B)\le \e^{\varepsilon/6}\mu_y(B).\]
This finally proves
\[  \limsup_{T\to\infty} \nu_{x,y}^T\bigl({\mathcal C}^-_{r}(x,A)\times {\mathcal C}^-_{r}(y,B)\bigr)\le \e^{\varepsilon}\mu_x(A)\mu_y(B)/\Vert m_\Gamma\Vert.\]
The proof of the inequality for the limit inferior is analogous.
\end{proof}

{\sl Proof of Theorem~\ref{equithm}.}\ 
Let $x$, $y\in\XX$ and $\varepsilon>0$ arbitrary. For $(\xi_0,\eta_0)\in\rand\times \rand$ we fix $r>0$ and open neighborhoods $V$, $W\subset\rand$ of $\xi_0$, $\eta_0$ \st the conclusion of Proposition~\ref{secondstep} holds. Choose open sets $\widehat V$, $\widehat W\subset\ganz$ with $\widehat V\cap\rand=V$ and $\widehat W\cap\rand=W$, and let $\widehat A$, $\widehat B\subset\ganz$ be Borel sets with $ \overline{\widehat A}\subset\widehat V$, $\overline{\widehat B}\subset\widehat W$ and
\begin{equation}\label{zeromeasureboundary} 
(\mu_x\otimes \mu_y)\bigl(\partial(\widehat A\times \widehat B)\bigr)=0.\end{equation}
Let $\alpha>0$ be arbitrary, and choose open sets $A^+, B^+\subset \rand $ and compact sets $A^-,  B^-\subset \rand $ with the properties
\begin{align*}
A^-&\subset \widehat A^\circ\cap\rand\subset  \overline{\widehat A}\cap\rand\subset A^+\subset V,\\
B^-&\subset \widehat B^\circ\cap\rand\subset  \overline{\widehat B}\cap\rand\subset B^+\subset W,\\
&\mu_x(\widehat A^\circ\setminus A^-)<\alpha,\quad \mu_x(A^+\setminus \overline{\widehat A})<\alpha,\\
&\mu_y(\widehat B^\circ\setminus B^-)<\alpha,\quad \mu_y(B^+\setminus \overline{\widehat B})<\alpha .
\end{align*}

Notice that according to Lemma~\ref{orbitpointsincones} (b)  the number of  $\gamma\in\Gamma$ with
\[ (\gamma y,\gamma^{-1}x)\in (\overline{\widehat A}\times\overline{\widehat B})\setminus \left( {\mathcal C}^-_{r}(x,A^+)\times {\mathcal C}^-_{r}(y,B^+)\right)\]
is finite; the same is true for 
the number of  $\gamma\in\Gamma$ with
\[ (\gamma y,\gamma^{-1}x)\in \left( {\mathcal C}^+_{r}(x,A^-)\times {\mathcal C}^+_{r}(y,B^-)\right) \setminus (\widehat A^\circ\times \widehat B^\circ)\]
by Lemma~\ref{orbitpointsincones} (a).
Hence 
\begin{align*}
\Vert m_\Gamma\Vert \cdot \limsup_{T\to\infty} \nu_{x,y}^T\bigl(\widehat A\times\widehat B\bigr) 
&\le \Vert m_\Gamma\Vert \cdot  \limsup_{T\to\infty} \nu_{x,y}^T\bigl({\mathcal C}^-_{r}(x,A^+)\times {\mathcal C}^-_{r}(y,B^+)\bigr),\\
\Vert m_\Gamma\Vert \cdot \liminf_{T\to\infty} \nu_{x,y}^T\bigl(\widehat A\times\widehat B\bigr)
&\ge \Vert m_\Gamma\Vert \cdot  \liminf_{T\to\infty} \nu_{x,y}^T\bigl({\mathcal C}^+_{r}(x,A^-)\times {\mathcal C}^+_{r}(y,B^-)\bigr).
\end{align*}
Proposition~\ref{secondstep} further implies 
\begin{align*}
\Vert m_\Gamma\Vert \cdot \limsup_{T\to\infty} \nu_{x,y}^T\bigl(\widehat A\times\widehat B\bigr) &\le \e^{\varepsilon} \mu_{x}(A^+)\mu_{y}(B^+)\\
&\le \e^{\varepsilon} \mu_{x}(\overline{\widehat A})\mu_{y}(\overline{\widehat B})+\alpha \e^\varepsilon \bigl(\mu_x(\rand)+\mu_y(\rand)\bigr)\\
&\stackrel{(\ref{zeromeasureboundary})}{\le} \e^{\varepsilon} \mu_{x}(\widehat A)\mu_{y}(\widehat B)+\alpha \e^\varepsilon \bigl(\mu_x(\rand)+\mu_y(\rand)\bigr)
\end{align*}
and
\begin{align*}
\Vert m_\Gamma\Vert \cdot \liminf_{T\to\infty} \nu_{x,y}^T\bigl(\widehat A\times\widehat B\bigr) &\ge \e^{-\varepsilon} \mu_{x}(A^-)\mu_{y}(B^-)\\
&\ge \e^{-\varepsilon} \mu_{x}(\widehat A^\circ)\mu_{y}(\widehat B^\circ)-\alpha \e^{-\varepsilon} \bigl(\mu_x(\rand)+\mu_y(\rand)\bigr)\\
&\stackrel{(\ref{zeromeasureboundary})}{\ge} \e^{-\varepsilon} \mu_{x}(\widehat A)\mu_{y}(\widehat B)-\alpha \e^{-\varepsilon }\bigl(\mu_x(\rand)+\mu_y(\rand)\bigr)
\end{align*}
As $\alpha$ was arbitrarily small we get in the limit as $\alpha$ tends to zero
\begin{align*}  \limsup_{T\to\infty} \nu_{x,y}^T\bigl(\widehat A\times\widehat B\bigr) &\le  \e^{\varepsilon} \mu_{x}(\widehat A)\mu_{y}(\widehat B)/\Vert m_\Gamma\Vert \quad\text{and}\\
 \liminf_{T\to\infty} \nu_{x,y}^T\bigl(\widehat A\times\widehat B\bigr) &\ge \e^{-\varepsilon} \mu_{x}(\widehat A)\mu_{y}(\widehat B)\Vert m_\Gamma\Vert .\end{align*}
So for every continuous and positive function $h$ with support in $\widehat V\times \widehat W$ we have
\begin{align*} \frac{\e^{-\varepsilon}}{\Vert m_\Gamma\Vert} \int h (\d \mu_x\otimes \d\mu_y) & \le  \liminf_{T\to\infty} \int h \d\nu_{x,y}^T\\
&\le \limsup_{T\to\infty} \int h \d\nu_{x,y}^T\le \frac{\e^\varepsilon}{\Vert m_\Gamma\Vert} \int h (\d \mu_x\otimes \d\mu_y).\end{align*}
Now the compact set $\rand\times\rand$ can be covered by a finite number of open sets of type $V\times W$ with $V$, $W\subset\rand$ as above, and similarly $\ganz\times\ganz$ by finitely many open sets $\widehat  V\times \widehat W$ with $\widehat V$, $\widehat  W\subset\ganz$ as above. Using a partition of unity subordinate to such a finite cover we see that the inequalities above remain true for every continuous and positive function on $\ganz\times\ganz$. The claim now follows by taking the limit $\varepsilon\to 0$, and passing from positive continuous functions to arbitrary continuous functions via a standard argument.
$\hfill\square$

We conclude this section with the following 
\begin{corollary}
Let $\Gamma<\is(\XX)$ be a discrete rank one group with  non-arithmetic length spectrum, 
$\zero_\Gamma\ne\emptyset$ 
and finite Ricks' Bowen-Margulis measure $m_\Gamma$.

Let $f:\ganz\to\RR$ be a continuous function, and  $x$, $y\in\XX$. Then  
\[ \lim_{T\to\infty}    
\delta_\Gamma  \e^{-\delta_\Gamma T} \sum_{\begin{smallmatrix}{\scriptstyle\gamma\in\Gamma}\\{\scriptstyle d(x,\gamma y )\le T}\end{smallmatrix}} f(\gamma y)=\frac{ \mu_y(\rand)}{\Vert m_\Gamma \Vert} \int_{\rand} f(\xi)\d \mu_x(\xi).\]
\end{corollary}

\section{Asymptotic estimates for the orbit counting function}\label{orbitcounting}

In this section we let $\XX$ be a proper Hadamard space and $\Gamma<\is(\XX)$ a discrete rank one group with $\zero_\Gamma\ne\emptyset$. 
Recall that  the orbit counting function with respect to $x$, $y\in \XX$ is defined by
\[ N_\Gamma:[0,\infty)\to\NN,\quad R\mapsto \#\{\gamma\in\Gamma\colon d(x,\gamma y)\leq R\}.\]

We first state a direct corollary of Theorem~\ref{equithm} (using $f=\mathbbm{1}_{\ganz\times\ganz}$):
\begin{proposition}
Let $\Gamma<\is(\XX)$ be a discrete rank one group with  non-arithmetic length spectrum, $\zero_\Gamma\ne\emptyset$ 
and finite Ricks' Bowen-Margulis measure $m_\Gamma$. Then for any $x$, $y\in\XX$ we have 
\[  \lim_{R\to\infty}  \delta_\Gamma\e^{-\delta_\Gamma R} N_\Gamma(R) =\frac{\mu_x(\rand)\mu_y(\rand)}{ \Vert m_\Gamma \Vert}.\]
\end{proposition}

We next deal with the case that the Ricks' Bowen-Margulis measure is not finite:
\begin{theorem}\label{lowgrowth}
Let $\Gamma<\is(\XX)$ be a discrete rank one group with $\zero_\Gamma\ne\emptyset$  and infinite Ricks' Bowen-Margulis measure $m_\Gamma$. If $\Gamma$ is divergent we further require that $\Gamma$ has non-arithmetic length spectrum. 
Then for the orbit counting function with respect to arbitrary points $x$, $y\in\XX$ we have
\[ \lim_{t\to\infty} N_\Gamma(t)\e^{-\delta_\Gamma t}=0.\]
\end{theorem}
As in the proof of Theorem~\ref{equithm} we define  the measure
\[
\nu_{x,y}^T:=  \delta_\Gamma \e^{-\delta_\Gamma T}\sum_{\begin{smallmatrix}{\scriptstyle\gamma\in\Gamma}\\{\scriptstyle d(x,\gamma y )\le T}\end{smallmatrix}} \mathcal{D}_{\gamma y}\otimes \mathcal{D}_{\gamma^{-1} x};\]
here we only have to show that \[\displaystyle \limsup_{T\to\infty} \nu_{x,y}^T(\ganz\times \ganz )=0.\]

Again, the first step of the proof is provided by the following 
\begin{lemma}\label{firststeplem}
Let $(\xi_0,\eta_0)\in\rand\times\rand$ and $x$, $y\in\XX$ with trivial stabilizer in $\Gamma$ and \st $x\in (\xi_0 v^+)$, $y\in (\eta_0 w^+)$ for some 
$\Gamma$-recurrent elements $v$, $w\in\zero$. 
Then there exist open neighborhoods $V$, $W\subset\rand$ of $\xi_0$,  $\eta_0$ \st for all Borel sets $A\subset V$, $B\subset W$
\begin{align*}
\limsup_{T\to\infty} \nu_{x,y}^T\bigl({\mathcal C}^-_{1}(x,A)\times {\mathcal C}^-_{1}(y,B)\bigr) &= 0.
\end{align*}
\end{lemma}
\begin{proof}
%
Let $\varepsilon>0$ arbitrary and set $\rho:= \min\{ d(x,\gamma x), d(y,\gamma y)\colon\gamma\in\Gamma\}$. 

As in the proof of Proposition~\ref{firststep} we fix $r\in (0, \min\{1, \rho/3, \varepsilon/(30\delta_\Gamma)\})$ \st 
\[ \mu_x\bigl(\widetilde\partial\mathcal{O}_r(\xi_0, x)\bigr)=0=\mu_y\bigl(\widetilde\partial\mathcal{O}_r(\eta_0, y)\bigr)\]
and choose  open neighborhoods $\widehat V$, $\widehat W\subset\ganz$   of $\xi_0$, $\eta_0$ \st if $(a, b)\in \widehat V\times \widehat W$, then $a$ can be joined to $v^+$,  $b$ can be joined to $w^+$ by a rank one geodesic 
and (\ref{approxmeasures}) holds.
Let $V\subset \widehat V\cap\rand$, $W\subset \widehat W\cap\rand$ be open neighborhoods of $\xi_0$, $\eta_0$, and $A\subset V$, $B\subset W$ arbitrary Borel sets;  
denote $K^+=K_r^+(x,A)$, $K^-=K_r^-(y,B)$, and $M=r^2 \mu_x\bigl({\mathcal O}_r(\xi_0,x)\bigr) \mu_y\bigl({\mathcal O}_r(\eta_0,y)\bigr)>0$. Then by mixing (or dissipativity in the case of a convergent group $\Gamma$) there exists $T_0\gg 1$ \st 
\[ \sum_{\gamma\in\Gamma} m(K^+\cap g^{-t}\gamma K^-)<M \varepsilon\cdot \e^{-\varepsilon/3}\]
 for all $t\ge T_0$, which implies 
\[ \bigl( \e^{\delta_\Gamma (T+3r)}- \e^{\delta_\Gamma T_0}\bigr)M \varepsilon\cdot \e^{-\varepsilon/3} > \delta_\Gamma \int_{T_0}^{T+ 3r} \e^{\delta_\Gamma t} \sum_{\gamma\in\Gamma} m(K^+\cap g^{-t}\gamma K^-)\d t.\]
We now use (\ref{lowerboundused}) to get
\begin{align*}   
& \hspace*{-0.7cm} \delta_\Gamma \int_{T_0}^{T+3r} \e^{\delta_\Gamma t} \sum_{\gamma\in\Gamma} m \bigl(  K^+\cap g^{-t}\gamma  K^-\bigr) \d t\\
 &\ge \e^{-\varepsilon/6} M \e^{\delta_\Gamma T} \nu_{x,y}^T\bigl( {\mathcal C}^-_{1}(x,A)\times {\mathcal C}^-_{1}(y,B)\bigr)- C
\end{align*}
with a constant $C$ independent of $T$. 
Dividing by $M \e^{\delta_\Gamma (T+3r)}$ then yields 
\begin{align*}
\bigl(1- \e^{\delta_\Gamma (-T-3r+T_0)}\bigr)  \varepsilon\cdot \e^{-\varepsilon/3} & >
 \e^{-\varepsilon/6}\e^{-3\delta_\Gamma r}\nu_{x,y}^T\bigl( {\mathcal C}^-_{1}(x,A)\times {\mathcal C}^-_{1}(y,B)\bigr)
 -  C'\e^{-\delta_\Gamma T}\\
 &=  \e^{-4\varepsilon/15} \nu_{x,y}^T\bigl( {\mathcal C}^-_{1}(x,A)\times {\mathcal C}^-_{1}(y,B)\bigr)+ C'\e^{-\delta_\Gamma T} ,
\end{align*}
where $C'$ is again  a constant  independent of $T$. We conclude 
\[ \limsup_{T\to\infty} \nu_{x,y}^T\bigl( {\mathcal C}^-_{1}(x,A)\times {\mathcal C}^-_{1}(y,B)\bigr)<\varepsilon  ,\]
and the claim follows from the fact that $\varepsilon>0$ was chosen arbitrarily small.
 \end{proof}
%


The next statement shows that in fact we can omit the conditions on $x$ and $y$ 
in Lemma~\ref{firststeplem}.
\begin{lemma}\label{secondsteplem}
Let $x$, $y\in\XX$ arbitrary. Then for all $(\xi_0,\eta_0)\in\rand\times\rand$   there exists $r>0$ and  open neighborhoods $V\subset\rand$ of $\xi_0$,  $W\subset\rand$ of $\eta_0$ \st for all Borel sets $A\subset V$, $B\subset W$
\begin{align*}
\limsup_{T\to\infty} \nu_{x,y}^T\bigl({\mathcal C}^-_{r}(x,A)\times {\mathcal C}^-_{r}(y,B)\bigr) &=0. 
\end{align*}
\end{lemma}
\begin{proof}
Let $(\xi_0,\eta_0)\in\rand\times\rand$ be arbitrary. Choose $\Gamma$-recurrent geodesics $v$, $w\in\zero$  and $x_0\in (\xi_0 v^+)$, $y_0\in (\eta_0 w^+)$ with trivial stabilizers in $\Gamma$. Let $V$, $W\subset\rand$ be open neighborhoods of $\xi_0$ and $\eta_0$ \st the statement of Lemma~\ref{firststeplem} holds for $x_0$, $y_0$ instead of $x$, $y$. 
Set \[r=1+\max\{ d(x,x_0), d(y,y_0)\}\]
and let $A\subset V$, $B\subset W$ be arbitrary Borel sets. From the choice of $r$ above and Lemma~\ref{changepoint} (b)
we know that  $(\gamma y,\gamma^{-1}x)\in {\mathcal C}^-_{r}(x,A)\times {\mathcal C}^-_{r}(y,B)$ implies 
\[ (\gamma y_0,\gamma^{-1}x_0)\in {\mathcal C}^-_{1}(x_0,A)\times {\mathcal C}^-_{1}(y_0,B).\]
%
If $d(x,\gamma y)\le T$, 
then obviously
\begin{align*}
d(x_0,\gamma y_0)&\le d(x_0, x)+ d(x,\gamma y)+d(y, y_0)\le T+d(x_0, x)+d(y, y_0),
\end{align*}
hence for $T\gg 1$ 
\begin{align*}
 \e^{-\delta_\Gamma T}& \#\{\gamma\in\Gamma\colon d(x,\gamma y)\le T,\ (\gamma y,\gamma^{-1} x)  \in \bigl({\mathcal C}^-_{r}(x,A)\times  {\mathcal C}^-_{r}(y,B)\bigr)
 \}\\
& \le \e^{\delta_\Gamma\bigl(d(x_0, x)+d(y, y_0)\bigr) }\cdot  \e^{-\delta_\Gamma\bigl(T+ d(x_0, x)+d(y, y_0)\bigr)}\cdot\\
&\hspace*{15mm}\#\{\gamma\in\Gamma\colon d(x_0,\gamma y_0)\le T+d(x_0, x)+d(y, y_0),\\
&\hspace*{29mm} \ (\gamma y_0,\gamma^{-1} x_0)  \in \bigl({\mathcal C}^-_{1}(x_0,A)\times  {\mathcal C}^-_{1}(y_0,B)\bigr)
\}.
\end{align*}
We conclude that
\begin{align*}
&  \limsup_{T\to\infty} \nu_{x,y}^T\bigl({\mathcal C}^-_{r}(x,A)\times {\mathcal C}^-_{r}(y,B)\bigr)\\
&\hspace*{1cm} \le
 \e^{\delta_\Gamma \left(d(x_0, x)+d(y, y_0) \right) } 
\limsup_{T\to\infty} \nu_{x_0,y_0}^{T+
d(x_0, x)+d(y, y_0)}\bigl({\mathcal C}^-_{1}(x_0,A)\times {\mathcal C}^-_{1}(y_0,B)\bigr)=0,\end{align*}
where we used Lemma~\ref{firststeplem} in the last estimate. 
\end{proof}

{\sl Proof of Theorem~\ref{lowgrowth}.}\ 
Let $x$, $y\in\XX$ and $\varepsilon>0$ arbitrary. For $(\xi_0,\eta_0)\in\rand\times \rand$ we fix $r>0$ and open neighborhoods $V$, $W\subset\rand$ of $\xi_0$, $\eta_0$ \st the conclusion of Lemma~\ref{secondsteplem} holds. Choose open sets $\widehat V$, $\widehat W\subset\ganz$ with $\widehat V\cap\rand=V$ and $\widehat W\cap\rand=W$, and let $\widehat A$, $\widehat B\subset\ganz$ be Borel sets with
\[ \overline{\widehat A}\subset\widehat V \quad\text{and }\quad \overline{\widehat B}\subset\widehat W. \]
Choose open sets $A $, $B \subset \rand $ 
with the properties
\begin{align*}
  \overline{\widehat A}\cap\rand\subset A  &\subset V\quad\text{and }\quad 
 \overline{\widehat B}\cap\rand\subset B \subset W;
\end{align*}
from Lemma~\ref{orbitpointsincones} (b) we know that the number of  $\gamma\in\Gamma$ with
\[ (\gamma y,\gamma^{-1}x)\in (\overline{\widehat A}\times\overline{\widehat B})\setminus \left( {\mathcal C}^-_{r}(x,A )\times {\mathcal C}^-_{r}(y,B )\right)\]
is finite. 
Hence
\begin{align*}
 \limsup_{T\to\infty} \nu_{x,y}^T\bigl(\widehat A\times\widehat B\bigr) 
&\le   \limsup_{T\to\infty} \nu_{x,y}^T\bigl({\mathcal C}^-_{r}(x,A )\times {\mathcal C}^-_{r}(y,B )\bigr)=0,
\end{align*}
which implies that for every continuous and positive function with support in $\widehat V\times \widehat W$ we have
\[  \limsup_{T\to\infty} \int h \d\nu_{x,y}^T=0.\]
Now the compact set $\rand\times\rand$ can be covered by a finite number of open sets of type $V\times W$ with $V$, $W\subset\rand$ as above, and similarly $\ganz\times\ganz$ by finitely many open sets $\widehat  V\times \widehat W$ with $\widehat V$, $\widehat  W\subset\ganz$ as above. Using a partition of unity subordinate to such a finite cover we see that the statement above remains true for every continuous and positive function on $\ganz\times\ganz$. $\hfill\square$

\section*{Acknowledgements}
The author is extremely grateful to the anonymous referee for his thorough and critical reading of the article. She would like to thank him for pointing out several inaccuracies and mistakes in previous versions of the paper, and in particular for providing the example and a suggestion how to fix a former error in the proof of Lemma 4.3.  
She is also very greatful for his many valuable suggestions to improve the exposition. 


\def\cprime{$'$}
\providecommand{\href}[2]{#2}
\providecommand{\arxiv}[1]{\href{http://arxiv.org/abs/#1}{arXiv:#1}}
\providecommand{\url}[1]{\texttt{#1}}
\providecommand{\urlprefix}{URL }

\end{document}